\documentclass[10pt,reqno]{amsart}
\usepackage{graphicx}
\usepackage{amsfonts}
\usepackage{amssymb}
\usepackage{amsmath}
\usepackage{amsxtra}
\usepackage{latexsym}
\usepackage{epstopdf}
\usepackage{mathrsfs}
\usepackage{esint}
\usepackage{enumitem}

\newtheorem{theorem}{Theorem}[section]
\newtheorem{lemma}[theorem]{Lemma}

\newtheorem{proposition}[theorem]{Proposition}
\newtheorem{Assumption}{Assumption}
\newtheorem{algorithm}{Algorithm}

\theoremstyle{definition}

\newtheorem{example}{Example}[section]

\theoremstyle{remark}

\numberwithin{equation}{section}

\begin{document}

\title[Stochastic mirror descent method]
{Stochastic mirror descent method for linear ill-posed problems in Banach spaces}

\author{Qinian Jin}

\address{Mathematical Sciences Institute, Australian National
University, Canberra, ACT 2601, Australia}
\email{qinian.jin@anu.edu.au} \curraddr{}

\author{Xiliang Lu}
\address{School of Mathematics and Statistics, Wuhan University, Wuhan 430072, China
\& Hubei Computational Science Key Laboratory, Wuhan University, Wuhan 430072, China}
\email{xllv.math@whu.edu.cn}

\author{Liuying Zhang}
\address{School of Mathematics and Statistics, Wuhan University, Wuhan 430072, China}
\email{lyzhang.math@whu.edu.cn}




\begin{abstract}
Consider linear ill-posed problems governed by the system $A_i x = y_i$ for $i =1, \cdots, p$, where each $A_i$ is a bounded linear operator from a Banach space $X$ to a Hilbert space $Y_i$. In case $p$ is huge, solving the problem by an iterative regularization method using the whole information at each iteration step can be very expensive, due to the huge amount of memory and excessive computational load per iteration. To solve such large-scale  ill-posed systems efficiently, we develop in this paper a stochastic mirror descent method which uses only a small portion of equations randomly selected at each iteration steps and incorporates convex regularization terms into the algorithm design. Therefore, our method scales very well with the problem size and has the capability of capturing features of sought solutions. The convergence property of the method depends crucially on the choice of step-sizes. We consider various rules for choosing step-sizes and obtain convergence results under {\it a priori} early stopping rules. In particular, by incorporating the spirit of the discrepancy principle we propose a choice rule of step-sizes which can efficiently suppress the oscillations in iterates and reduce the effect of semi-convergence. Furthermore, we establish an order optimal convergence rate result when the sought solution satisfies a benchmark source condition. Various numerical simulations are reported to test the performance of the method. 
\end{abstract}

\def\p{\partial}
\def\d{\delta}
\def\l{\langle}
\def\r{\rangle}
\def\C{\mathcal C}
\def\D{\mathscr D}
\def\a{\alpha}
\def\b{\beta}
\def\d{\delta}

\def\la{\lambda}
\def\ep{\varepsilon}
\def\Ga{\Gamma}
\def\R{{\mathcal R}}
\def\J{{\mathcal J}}
\def\E{\mathcal E}
\def\U{\mathcal U}
\def\Q{\mathcal Q}
\def\eps{\varepsilon}

\def\C{\mathcal C}
\def\G{\mathcal G}
\def\S{\mathcal S}
\def\M{\mathcal M}

\def\N{\mathcal N}
\def\X{\mathcal X}
\def\Y{\mathcal Y}
\def\B{\mathcal B}
\def\A{\mathcal A}
\def\EE{\mathbb E}
\def\H{\mathcal H}

\def\bQ{{\bf Q}}
\def\bD{{\bf D}}
\def\bB{{\bf B}}
\def\D{\mathscr D}
\def\yd{y^{\delta}}
\def\xtd{\tilde{x}}
\def\bA{{\bf A}}
\def\bx{{\bf x}}
\def\by{{\bf y}}
\def\bz{{\bf z}}

\newcommand{\of}[1]{\left( #1 \right)}
\newcommand{\off}[1]{(#1)}
\newcommand{\dom}[1]{\mathcal{D}(#1)}
\newcommand{\Sp}[1]{\mathcal{#1}}
\newcommand{\norm}[1]{\Vert #1 \Vert}
\newcommand{\xd}[1]{x_{#1}^{\delta}}
\newcommand{\brc}[1]{\left\lbrace #1 \right\rbrace}
\newcommand{\mf}[2]{\mathcal{S}\of{#1,#2}}
\newcommand{\arrow}[1]{\overset{\tau_{\Sp{#1}}}{\longrightarrow}}

\maketitle

\section{\bf Introduction}
\setcounter{equation}{0}

Due to rapid growth of data sizes in practical applications, in recent years stochastic optimization methods have received tremendous attention and have been proved to be efficient in various applications of science and technology including in particular the machine learning researches (\cite{BCN2018,GBC2017}). In this paper we will develop a stochastic mirror descent method for solving linear ill-posed inverse problems in Banach spaces. 

We will consider linear ill-posed inverse problems governed by the system 
\begin{align}\label{smd.1}
A_i x = y_i, \qquad  i =1, \cdots, p
\end{align}
consisting of $p$ linear equations, where, for each $i=1, \cdots, p$, $A_i: X \to Y_i$ is a bounded linear operator from a Banach space $X$ to a Hilbert space $Y_i$. Such systems arise naturally in many practical applications. For instance, many linear ill-posed inverse problems can be described by an integral equation of the first kind (\cite{EHN1996,G1984})
$$
(A x)(s):=\int_\D k(s, t) x(t) dt = y(s), \quad s\in \D', 
$$
where $\D$ and $\D'$ are bounded domains in Euclidean spaces and the kernel $k$ is a bounded continuous function on $\D'\times \D$. Clearly $A$ is a bounded linear operator from $L^r(\D)$ to $C(\D')$ for any $1\le r \le \infty$. 
By taking $p$ sample points $s_1, \cdots, s_p$ in $\D'$, then the problem of determining a solution using only the knowledge of $y_i:=y(s_i)$ for $i=1, \cdots, p$ reduces to solving a linear system of the form (\ref{smd.1}), where $A_i: L^r(\D) \to {\mathbb R}$ is given by 
$$
A_i x:=\int_\D k(s_i, t) x(t) dt
$$
for each $i$. Further examples of (\ref{smd.1}) can be found in various tomographic techniques using multiple measurements \cite{N2001}. 

Throughout the paper we always assume (\ref{smd.1}) has a solution. The system (\ref{smd.1}) may have many solutions. By taking into account of {\it a priori} information about the sought solution, we may use a proper, lower semi-continuous, convex function $\R: X \to (-\infty, \infty]$ to select a solution $x^\dag$ of (\ref{smd.1}) such that
\begin{align}\label{rbdgm.1}
\R(x^\dag) = \min\left\{\R(x): A_i x = y_i \mbox{ for } i =1, \cdots, p\right\}
\end{align}
which, if exists, is called a $\R$-minimizing solution of (\ref{smd.1}). In practical applications, the exact data $y:=(y_1, \cdots, y_p)$ is in general not available, instead we only have noisy data $y^\d:=(y_1^\d, \cdots, y_p^\d)$ satisfying
\begin{align}\label{rbdgm.2}
\|y_i^\d - y_i\| \le \d_i, \quad i=1, \cdots, p,
\end{align}
where $\d_i>0$ denotes the noise level corresponding to data in the space $Y_i$. How to use the noisy data $y^\d$ to construct an approximate solution of the $\R$-minimizing solution of (\ref{smd.1}) is an important topic. 

Let $Y:= Y_1 \times \cdots \times Y_p$ and define $A: X \to Y$ by 
$$
A x = (A_1 x, \cdots, A_p x), \quad  x\in X.
$$
Then (\ref{rbdgm.1}) can be equivalently stated as 
\begin{align}\label{smd.2}
\R(x^\dag) = \min\left\{\R(x): A x = y\right\}.
\end{align}
which has been considered by various variational and iterative regularization methods, see \cite{BH2012,BO2004,Jin2021,JW2013,SKHK2012} for instance. In particular, the Landweber iteration in Hilbert spaces has been extended for solving (\ref{smd.2}), leading to the iterative method of the form
\begin{align}\label{MD}
\begin{split} 
& x_n^\d  = \arg\min_{x\in X} \left\{ \R(x) - \l \xi_n^\d, x\r \right\}, \\
& \xi_{n+1}^\d = \xi_n^\d - t_n^\d A^* (A x_n^\d - y^\d),
\end{split}
\end{align}
where $A^*: Y \to X^*$ denotes the adjoint of $A$ and $t_n^\d$ is the step-size. 
This method can be derived as a special case of the mirror descent method; see Section \ref{sect2} for a brief account. The method (\ref{MD}) has been investigated in a number of  references, see \cite{BH2012,FS2010,Jin2021,JL2014,JW2013}. In particular, the convergence and convergence rates have been derived in \cite{Jin2021} when the method is terminated by either an {\it a priori} stopping rule or the discrepancy principle 
$$
\|A x_{n_\d}^\d - y^\d\| \le \tau \d <\|A x_n^\d - y^\d\|, \quad 0\le n <n_\d, 
$$
where 
$$
\d:=\sqrt{\d_1^2 + \cdots + \d_p^2}
$$
denotes the total noise level of the noisy data. We remark that the minimization problem in (\ref{MD}) for defining $x_n^\d$ 
can be solved easily in general as it does not depend on $A$; in fact $x_n^\d$ can be given by an explicit formula in many interesting cases; even if $x_n^\d$ does not have an explicit formula, there
exist fast algorithms for determining $x_n^\d$ efficiently. However, note that  
$$
A^* (A x_n^\d - y^\d) = \sum_{i=1}^p A_i^* (A_i x_n^\d - y_i^\d). 
$$
Therefore, updating $\xi_n^\d$ to $\xi_{n+1}^\d$ requires calculating $A_i^*(A_i x_n^\d - y_i^\d)$ for all $i =1, \cdots, p$. 
In case $p$ is huge, using the method (\ref{MD}) to solve (\ref{smd.2}) can be inefficient because it requires a huge amount of memory and excessive computational work per iteration. 

In order to relieve the drawback of the method (\ref{MD}),  by extending the Kaczmarz-type method \cite{HLS2007} in Hilbert spaces, a Landweber-Kaczmarz method in Banach spaces has been proposed in \cite{Jin2016,JW2013} to solve (\ref{rbdgm.1}) which cyclically considers each equation in (\ref{rbdgm.1}) in a Gauss-Seidel manner and the iteration scheme takes the form 
\begin{align}\label{MD-K}
\begin{split} 
& x_n^\d  = \arg\min_{x\in X} \left\{ \R(x) - \l \xi_n^\d, x\r \right\}, \\
& \xi_{n+1}^\d = \xi_n^\d - t_n^\d A_{i_n}^* (A_{i_n} x_n^\d - y_{i_n}^\d),
\end{split}
\end{align}
where $i_n = (n\mod p) + 1$. The convergence of this method has been shown in \cite{JW2013} in which the numerical results demonstrate its nice performance. However, it should be pointed out that the efficiency of the method (\ref{MD-K}) depends crucially on the order of the equations and its convergence speed is difficult to be quantified. In order to resolve these issues, in this paper we will consider a stochastic version of (\ref{MD-K}), namely, instead of taking $i_n$ cyclically, we will choose $i_n$ from $\{1, \cdots, p\}$ randomly at each iteration step. The corresponding method will be called the stochastic mirror descent method and  more details will be presented in Section \ref{sect2} where we also propose the mini-batch version of the stochastic mirror descent method.

The stochastic mirror descent method, that we will consider in this paper, includes the stochastic gradient descent as a special case. Indeed, when $X$ is a Hilbert space and $\R(x) = \|x\|^2/2$, the method (\ref{MD-K}) becomes \begin{align}\label{SGD}
x_{n+1}^\d = x_n^\d - t_n^\d A_{i_n}^* (A_{i_n} x_n^\d - y_{i_n}^\d)
\end{align}
which is exactly the stochastic gradient descent method studied in \cite{JL2019,JZZ2020,LM2021,RSL2022} for solving linear ill-posed problems in Hilbert spaces.
In many applications, however, the sought solution may sit in a Banach space instead of a Hilbert space, and the sought solution may have {\it a priori} known special features, such as nonnegativity, sparsity and piecewise constancy. Unfortunately, the stochastic gradient descent method does not have the capability to incorporate these information into the algorithm. 
However, this can be handled by the stochastic mirror descent method with careful choices of a suitable Banach space $X$ and a strongly convex penalty functional $\R$. 

In this paper we will use tools from convex analysis in Banach spaces to analyze the stochastic mirror descent method. The choice of the step-size plays a crucial role on the convergence of the method. We consider several rules for choosing the step-sizes and provide criteria for terminating the iterations in order to guarantee a convergence when the noise level tends to zero. The iterates produced by the stochastic mirror descent method exhibit salient oscillations and, due to the ill-posedness of the underlying problems, the method using noisy data demonstrates the semi-convergence property, i.e. the iterate tends to the sought solution at the beginning and, after a critical number of iterations, the iterates diverges. The oscillations and  semi-convergence make it difficult to determine an output with good approximation property, in particular when the noise level is relatively large. When the information on noise level is available, by incorporating the spirit of the discrepancy principle we propose a rule for choosing step-size. This rule enables us to efficiently suppress the oscillations of iterates and remove the semi-convergence of the method as indicated by the extensive numerical simulations. Furthermore, we obtain an order optimal convergence rate result for the stochastic mirror descent method with constant step-size when the sought solution satisfies a benchmark source condition. We achieve this by interpreting the stochastic mirror descent method equivalently as a randomized block gradient method applied to the dual problem of (\ref{smd.1}).
Even for the stochastic gradient descent method, our convergence rate result supplements the existing results since only sub-optimal convergence rates have been derived under diminishing step-sizes, see \cite{JL2019,LM2021}.

This paper is organized as follows. In Section \ref{sect2} we first collect some basic facts on convex analysis in Banach spaces and then give an account on the stochastic mirror descent method. 
In Section \ref{sect3} we prove some convergence results on the stochastic mirror descent method under various choices of the step-sizes. When the sought solution satisfies a benchmark source condition, in Section \ref{sect4} we establish an order optimal convergence rate result. Finally, in Section \ref{sect5} we present extensive numerical simulations to test the performance of the stochastic mirror descent method.

\section{\bf The method}\label{sect2}
\setcounter{equation}{0}

\subsection{Preliminaries}

In this section, we will collect some basic facts on convex analysis in Banach
spaces which will be used in the analysis of the stochastic mirror descent method; for
more details one may refer to \cite{R1970,Z2002} for instance.

Let $X$ be a Banach space whose norm is denoted by $\|\cdot\|$, we use $X^*$ to denote its dual space. Given $x\in X$ and $\xi\in X^*$ we write $\l \xi, x\r = \xi(x)$ 
for the duality pairing; in case $X$ is a Hilbert space, we also use $\l \cdot, \cdot\r$ to denote the inner product. For a convex function $f : X \to  (-\infty, \infty]$,  its effective domain is denoted by 
$$
\mbox{dom}(f) := \{x \in X : f(x) < \infty\}.
$$
If $\mbox{dom}(f) \ne \emptyset$, $f$ is called proper. 
Given $x\in \mbox{dom}(f)$, an element $\xi\in X^*$ is called a subgradient of $f$ at $x$ if
$$
f(\bar x) \ge  f(x) + \l\xi, \bar x - x\r,  \quad \forall \bar x \in X.
$$
The collection of all subgradients of $f$ at $x$ is denoted as $\p f(x)$ and is called the
subdifferential of $f$ at $x$. If $\p f(x) \ne \emptyset$, then $f$ is called subdifferentiable at $x$. Thus
$x \to \p f(x)$ defines a set-valued mapping $\p f$ whose domain of definition is defined as
$$
\mbox{dom}(\p f) := \{x \in \mbox{dom}(f) : \p f(x) \ne \emptyset\}.
$$
Given $x\in \mbox{dom}(\p f)$ and $\xi \in \p f(x)$, the Bregman distance induced by $f$ at $x$ in
the direction $\xi$ is defined by
$$
D_f^\xi(\bar x, x) := f(\bar x) -  f(x) - \l \xi, \bar x - x\r,  \quad \forall \bar x \in X
$$
which is always nonnegative.

A proper function $f : X \to (-\infty, \infty]$ is called strongly convex if there exists a
constant $\sigma>0$ such that
\begin{align}\label{dgm.23}
f(t\bar x + (1-t) x) + \sigma t(1-t) \|\bar x -x\|^2 \le  tf(\bar x) + (1-t)f(x)
\end{align}
for all $\bar x, x\in \mbox{dom}(f)$ and $t\in [0, 1]$. The largest number $\sigma>0$ such that (\ref{dgm.23})
holds true is called the modulus of convexity of $f$. It is easy to see that for a proper,
strongly convex function $f : X \to (-\infty, \infty]$ with modulus of convexity $\sigma>0$ there holds 
\begin{align}\label{dgm.24}
D_f^\xi(\bar x, x) \ge \sigma \|x-\bar x\|^2
\end{align}
for all $\bar x\in \mbox{dom}(f)$, $x\in \mbox{dom}(\p f)$ and $\xi\in \p f(x)$.

For a proper function $f : X\to (-\infty, \infty]$, its Legendre-Fenchel conjugate is
defined by
$$
f^*(\xi) :=  \sup_{x\in X}  \{\l \xi, x\r - f(x)\}, \quad  \xi \in X^* 
$$
which is a convex function taking values in $(-\infty, \infty]$. By definition we have the Fenchel-Young inequality
\begin{align}\label{dgm.21}
f^*(\xi) + f(x) \ge \l \xi, x\r 
\end{align}
for all $x\in X$ and $\xi\in X^*$. If $f : X\to  (-\infty, \infty]$ is proper, lower semi-continuous
and convex, $f^*$ is also proper and 
\begin{align}\label{dgm.22}
\xi \in \p f(x) \Longleftrightarrow x\in \p f^*(\xi) \Longleftrightarrow  f(x) + f^*(\xi) = \l \xi, x\r.
\end{align}
The following important result gives further properties of $f^*$ which in particular shows that the strong convexity of $f$ implies the continuous differentiability of $f^*$ with gradient $\nabla f^*$ mapping $X^*$ to $X$; see \cite[Corollary 3.5.11]{Z2002} 

\begin{proposition}\label{dgm.prop22}
Let $X$ be a Banach space and let $f : X \to (-\infty, \infty]$ be a proper,
lower semi-continuous, strongly convex function with modulus of convexity $\sigma>0$.
Then $\mbox{dom}(f^*) = X^*$, $f^*$ is Fr\'{e}chet differentiable and its gradient $\nabla f^*$ maps $X^*$ into $X$ satisfying 
$$
\|\nabla f^*(\xi) -\nabla f^*(\eta) \| \le \frac{\|\xi-\eta\|}{2\sigma} 
$$
for all $\xi, \eta \in X^*$. 
\end{proposition} 

Given a proper strongly convex function $f: X\to (-\infty, \infty]$, we may consider for each $\xi\in X^*$ the convex minimization problem
\begin{align}\label{convmin}
\min_{x\in X}\left\{f(x)-\l \xi, x\r\right\}
\end{align}
which is involved in the  formulation of the stochastic mirror descent method below. According to \cite[Theorem 3.5.8]{Z2002}, (\ref{dgm.22}) and Proposition \ref{dgm.prop22} we have

\begin{proposition}\label{dgm:prop23}
If $f:X\to (-\infty, \infty]$ is a proper, lower semi-continuous, strongly convex function, then for any $\xi\in X^*$ the minimization problem (\ref{convmin}) has a unique minimizer given by $\nabla f^*(\xi)$.
\end{proposition}

\subsection{Description of the method}

In order to motivate our method, we first briefly review how to extend the gradient method in Hilbert spaces to solve the minimization problem
\begin{align}\label{smd.21}
\min_{x\in X} f(x)  
\end{align}
in Banach spaces, where $f : X \to {\mathbb R}$ is a Fr\'{e}chet differentiable function defined
on a Banach space $X$. When $X$ is a Hilbert space, the gradient method for solving (\ref{smd.21}) takes the form $x_{n+1} = x_n-t_n \nabla f(x_n)$ which can be equivalently
stated as
\begin{align}\label{smd.22}
x_{n+1} & = \arg\min_{x\in X} \left\{\frac{1}{2} \|x-(x_n - t_n \nabla f(x_n))\|^2 \right\} \nonumber \\
& = \arg \min_{x\in X} \left\{ \frac{1}{2} \|x-x_n\|^2 + t_n \l \nabla f(x_n), x\r \right\},
\end{align}
where $\l \cdot, \cdot\r$ denotes the inner product on $X$ and $\|\cdot\|$ denotes the induced norm. When $X$ is a general Banach space, inner product is no longer available and $\nabla f(x_n)$ is not necessarily an element in $X$; instead $\nabla f(x_n)$ is an element in $X^*$. Therefore, in order to guarantee $\l \nabla f(x_n), x\r$ to be meaningful, this expression should be understood as the dual pairing between $X^*$ and $X$. 
On the other hand, since no Hilbert space norm is available, $\frac{1}{2} \|x-x_n\|^2$ in (\ref{smd.22}) should be replaced by other suitable distance-like functionals. 
In order to capture the feature of the sought solution, a suitable convex penalty function $\R : X \to (-\infty, \infty]$ is usually chosen to enhance the feature. 
In such a situation, one may use the Bregman distance induced by $\R$ to fulfill the purpose. To be more precise, if $\p \R(x_n) \ne \emptyset$, we may take $\xi_n\in \p \R(x_n)$ suitably and then use the Bregman distance
$$
D_\R^{\xi_n}(x, x_n) = \R(x) - \R(x_n) - \l \xi_n, x-x_n\r 
$$
to replace $\frac{1}{2} \|x-x_n\|^2$ in (\ref{smd.22}). This leads to the new updating formula
$$
x_{n+1} \in \arg\min_{x\in X} \left\{ D_\R^{\xi_n}(x, x_n) + t_n \l \nabla f(x_n), x\r \right\}
$$
for $x_{n+1}$. By the expression of $D_\R^{\xi_n}(x, x_n)$, we have
$$
x_{n+1} \in \arg\min_{x\in X} \left\{ \R(x) -\l \xi_n - t_n \nabla f(x_n), x\r \right\}. 
$$
Let $\xi_{n+1} = \xi_n - t_n \nabla f(x_n)$, then 
$$
x_{n+1} \in \arg\min_{x\in X} \left\{ \R(x) - \l \xi_{n+1}, x\r\right\}. 
$$
If $x_{n+1}$ is well-defined, then by the optimality condition on $x_{n+1}$ we have $\xi_{n+1} \in \p \R(x_{n+1})$. Therefore, we can repeat the above procedure, leading to the algorithm 
\begin{align}\label{smd.23}
\begin{split}
& x_n = \arg\min_{x\in X} \left\{\R(x) - \l \xi_n, x\r\right\},\\
& \xi_{n+1} = \xi_n - t_n \nabla f(x_n) 
\end{split}
\end{align}
for solving (\ref{smd.21}) in Banach spaces, which is called the mirror descent method in optimization community; see \cite{BT2003,B2015,NY1983,N2007}.

In order to apply the mirror descent method to solve (\ref{rbdgm.1}) when only noisy data $y_i^\d$ are available, we may consider a problem of the form (\ref{smd.21}) with 
$$
f(x) = \frac{1}{2} \sum_{i=1}^p \|A_i x - y_i^\d\|^2. $$
Note that $\nabla f(x) = \sum_{i=1}^p A_i^* (A_i x - y_i^\d)$. Therefore, an application of (\ref{smd.23}) gives
\begin{align}\label{smd.24}
\begin{split}
& x_n = \arg\min_{x\in X} \left\{\R(x) - \l \xi_n, x\r\right\},\\
& \xi_{n+1} = \xi_n -t_n \sum_{i=1}^p A_i^*(A_i x_n - y_i^\d).  
\end{split}
\end{align}
This is exactly the method (\ref{MD}) which has been considered in \cite{BH2012,Jin2021,JW2013}. 
Clearly, the implementation of (\ref{smd.24}) requires to calculate $A_i^*(A_i x_n - y_i^\d)$ for all $i = 1, \cdots, p$ at each iteration. In case $p$ is huge, each iteration step in (\ref{smd.24}) requires a huge amount of computational work. In order to reduce the computational load per iteration, one may randomly choose a subset $I_n$ from $\{1, \cdots, p\}$ with small size to form the partial term $$
f_{I_n}(x) := \frac{1}{2} \sum_{i\in I_n} \|A_i x -y_{i}^\d\|^2
$$ 
of $f$ and use its gradient $\nabla f_{I_n}(x_n) = \sum_{i\in I_n} A_i^* (A_i x_n - y_i^\d)$ at $x_n$ as a replacement of  $\nabla f(x_n)$ in (\ref{smd.24}). This leads to the following mini-batch stochastic mirror descent method for solving (\ref{rbdgm.1}) with noisy data. 

\begin{algorithm}\label{alg:SMD}
Fix a batch size $b$, pick the initial guess $\xi_0 =0$ in $X^*$ and set $\xi_0^\d := \xi_0$. For $n\ge 0$ do the following:

\begin{enumerate} [leftmargin = 0.8cm]
\item[\emph{(i)}] Calculate $x_n^\d \in X$ by solving \begin{align}\label{smd.35}
x_n^\d = \arg\min_{x\in X} \left\{\R(x) - \l \xi_n^\d, x\r\right\}; 
\end{align}

\item[\emph{(ii)}] Randomly select a subset $I_n\subset \{1, \cdots, p\}$ with $|I_n|=b$ via the uniform distribution;

\item[\emph{(iii)}] Choose a step-size $t_n^\d\ge 0$;

\item[\emph{(iv)}] Define $\xi_{n+1}^\d \in X^*$ by
\begin{align}\label{smd.25}
\xi_{n+1}^\d = \xi_n^\d - t_n^\d \sum_{i\in I_n} A_i^*\left(A_i x_n^\d - y_i^\d\right).
\end{align}
\end{enumerate}
\end{algorithm}

It should be pointed out that the choice of the step-size $t_n^\d$ plays a crucial role on the convergence of Algorithm \ref{alg:SMD}. We will consider in Section \ref{sect3} several rules for choosing the step-size. Note also that, at each iteration step of Algorithm \ref{alg:SMD}, $x_n^\d$ is defined by a minimization problem (\ref{smd.35}).
When $\R$ is proper, lower semi-continuous and strongly convex, it follows from Proposition \ref{dgm:prop23} that $x_n^\d$ is well-defined. Moreover, for many important choices of $\R$, $x_n^\d$ can be given by an explicit formula, see Section \ref{sect5} for instance, and thus the calculation of $x_n^\d$ does not take much time.

There exist extensive studies on the mirror descent method and its stochastic variants in optimization, see \cite{B2015,D2018,NL2014,NJLS2009,ZMBBG2020} for instance. The existing works either depend crucially on the finite-dimensionality of the underlying spaces or establish only error estimates in terms of objective function values, and therefore they are not applicable to our Algorithm \ref{alg:SMD} for ill-posed problems. We need to develop new analysis. Our analysis of Algorithm \ref{alg:SMD} is based on the following assumption which is assumed throughout the paper. 

\begin{Assumption}\label{dgm.ass1}
\begin{enumerate}[leftmargin = 0.8cm]
\item[\emph{(i)}] $X$ is a Banach space, $Y_i$ is a Hilbert space and $A_i: X \to Y_i$ is a bounded linear operator for each $i =1, \cdots, p$;

\item[\emph{(ii)}] $\R : X\to (-\infty, \infty]$ is a proper, lower semi-continuous, strongly convex function with modulus of convexity $\sigma>0$;

\item[\emph{(iii)}] The system $A_ix = y_i$, $i=1, \cdots, p$, has a solution in $\emph{dom}(\R)$.
\end{enumerate}
\end{Assumption}

According to Assumption \ref{dgm.ass1} (ii) and Proposition \ref{dgm.prop22}, the Legendre-Fenchel conjugate $\R^*$ of $\R$ is continuous differentiable with Lipschitz continuous gradient, i.e.
\begin{align}\label{rbdgm.41}
\|\nabla \R^*(\xi)-\nabla \R^*(\eta)\| \le \frac{\|\xi-\eta\|}{2\sigma}, \quad \forall \xi, \eta \in X^*.
\end{align}
Moreover, by virtue of Assumption \ref{dgm.ass1} and Proposition \ref{dgm:prop23}, the linear system (\ref{smd.1}) has a unique $\R$-minimizing solution
$x^\dag$ and Algorithm \ref{alg:SMD} is well-defined. By the optimality condition on $x_n^\d$ and Proposition \ref{dgm:prop23} we have 
\begin{align}\label{rdbgm.61}
\xi_n^\d \in \p \R(x_n^\d) \quad \mbox{ and } \quad x_n^\d = \nabla \R^*(\xi_n^\d)
\end{align}
which is the starting point of our convergence analysis.

\section{\bf Convergence}\label{sect3}
\setcounter{equation}{0}

In order to establish a convergence result on  Algorithm \ref{alg:SMD}, we need to specify a probability space on which
the analysis will be carried out. Let $\Lambda_p:=\{1, \cdots, p\}$ and let ${\mathcal S}_p$ denote the $\sigma$-algebra consisting of all subsets of $\Lambda_p$. Recall that, at each iteration step, a subset of indices is randomly chosen from $\Lambda_p$ via the uniform distribution. Therefore, for each $n\ge 1$, it is natural to consider $x_n^\d$ and $\xi_n^\d$ on the sample space
$$
\Lambda_p^n : = \underbrace{\Lambda_p \times \cdots \times \Lambda_p}_{n \mbox{ copies}}
$$
equipped with the $\sigma$-algebra ${\mathcal S}_p^{\otimes n}$ and the uniform distributed probability, denoted as ${\mathbb P}_n$.
According to the Kolmogorov extension theorem (\cite{BW2016}), there exists a unique probability ${\mathbb P}$ defined on the
measurable space $(\Omega, {\mathcal F}):= (\Lambda_p^\infty, {\mathcal S}_p^{\otimes \infty})$ such that each ${\mathbb P}_n$ is
consistent with ${\mathbb P}$. Let ${\mathbb E}$ denote the expectation on the probability space $(\Omega, {\mathcal F}, {\mathbb P})$.
Given a Banach space $X$, we use $L^2(\Omega, X)$ to denote the space consisting of all random variables $x$ with values in $X$
such that $\EE[\|x\|^2]$ is finite; this is a Banach space under the norm
$$
\left(\EE[\|x\|^2]\right)^{1/2}.
$$
Concerning Algorithm \ref{alg:SMD} we will use $\{{\mathcal F}_n\}$ to denote the natural filtration,
where ${\mathcal F}_n := \sigma(I_0, \cdots, I_{n-1})$ for each $n \ge 1$. We will frequently use the identity
\begin{align}\label{cond.1}
\EE[\phi] = \EE[\EE[\phi|{\mathcal F}_n]]
\end{align}
for any random variable $\phi$ on $(\Omega, {\mathcal F}, {\mathbb P})$. where $\EE[\phi|{\mathcal F}_n]$ denotes the expectation of $\phi$ conditioned on ${\mathcal F}_n$.

In this section we will prove some convergence results on Algorithm \ref{alg:SMD} under suitable choices of the step-sizes. The convergence of Algorithm \ref{alg:SMD} will be established by investigating the convergence property of its counterpart for exact data together with its stability property. For simplicity of exposition, for each index set $I =\{i_1, \cdots, i_b\}\subset \{1, \cdots, p\}$ of size $b$ we set 
$$
Y_I := Y_{i_1} \times \cdots \times Y_{i_b}, \quad y_I^\d := (y_{i_1}^\d, \cdots, y_{i_b}^\d), \quad 
\d_I^2 = \d_{i_1}^2 + \cdots + \d_{i_b}^2
$$
and define $A_I: X \to Y_I$ by 
$$
A_I x := (A_{i_1} x, \cdots, A_{i_b} x), \quad \forall x \in X.
$$
Let $A_I^*$ denote the adjoint of $A_I$. Then the updating formula (\ref{smd.25}) of $\xi_{n+1}^\d$ from $\xi_n^\d$ can be rephrased as 
$$
\xi_{n+1}^\d =\xi_n^\d - t_n^\d A_{I_n}^* (A_{I_n} x_n^\d - y_{I_n}^\d). 
$$
We start with the following result. 

\begin{lemma}\label{rdbgm.lem11}
Let Assumption \ref{dgm.ass1} hold. Consider Algorithm \ref{alg:SMD} and assume that 
$$
0\le t_n^\d \le \min\left\{\frac{\mu_0 \|A_{I_n} x_n^\d - y_{I_n}^\d\|^2}{\|A_{I_n}^*(A_{I_n} x_n^\d - y_{I_n}^\d)\|^2}, \mu_1\right\}
$$ 
for all $n \ge 0$, where $\mu_0$ and $\mu_1$ are two positive constants with $c_0:= 1-\mu_0/(4\sigma) >0$. Let $\hat x$ be any solution of (\ref{smd.1}) in $\mbox{dom}(\R)$ and let
$$
\Delta_n^\d := D_\R^{\xi_n^\d} (\hat x, x_n^\d), \quad n=0, 1, \cdots.
$$
Then
\begin{align*}
\EE\left[\Delta_{n+1}^\d\right] - \EE\left[\Delta_n^\d\right] \le \frac{\mu_1 b}{4 c_0 p} \d^2.
\end{align*}
\end{lemma}

\begin{proof}
Note that
$$
\Delta_{n+1}^\d - \Delta_n^\d = \left(\l \xi_{n+1}^\d, x_{n+1}^\d - \hat x\r - \R(x_{n+1}^\d)\right)
+ \left(\R(x_n^\d) + \l \xi_n^\d, \hat x - x_n^\d\r \right).
$$
By using (\ref{rdbgm.61}) and (\ref{dgm.22}) we have
$$
\R(x_n^\d) + \R^*(\xi_n^\d) = \l \xi_n^\d, x_n^\d\r
$$
for all $n \ge 0$. Therefore
\begin{align*}
\Delta_{n+1}^\d - \Delta_n^\d
& = \left(\R^*(\xi_{n+1}^\d) - \l \xi_{n+1}^\d, \hat x\r\right) - \left(\R^*(\xi_n^\d) - \l \xi_n^\d, \hat x\r \right).
\end{align*}
Using $x_n^\d = \nabla \R^*(\xi_n^\d)$ in (\ref{rdbgm.61}) and the inequality (\ref{rbdgm.41}) on $\nabla \R^*$, we can obtain
\begin{align*}
\Delta_{n+1}^\d - \Delta_n^\d
& = \left(\R^*(\xi_{n+1}^\d) - \R^*(\xi_n^\d) - \l \xi_{n+1}^\d -\xi_n^\d, \nabla \R^*(\xi_n^\d)\r\right) \\
& \quad \,  + \l \xi_{n+1}^\d - \xi_n^\d, x_n^\d-\hat x\r \\
& \le \frac{1}{4\sigma} \|\xi_{n+1}^\d - \xi_n^\d\|^2 + \l \xi_{n+1}^\d - \xi_n^\d, x_n^\d-\hat x\r.
\end{align*}
According to the definition of $\xi_{n+1}^\d$ and $A_{I_n}\hat x = y_{I_n}$, we can further obtain
\begin{align*}
\Delta_{n+1}^\d - \Delta_n^\d
& \le \frac{1}{4\sigma} \left(t_n^\d\right)^2 \| A_{I_n}^* (A_{I_n} x_n^\d - y_{I_n}^\d)\|^2
- t_n^\d \l A_{I_n} x_n^\d-y_{I_n}^\d, A_{I_n} x_n^\d - y_{I_n}\r.
\end{align*}
By the given condition on $t_n^\d$ and the Cauchy-Schwarz inequality, we then obtain 
\begin{align*} 
\Delta_{n+1}^\d - \Delta_n^\d 
& \le \frac{\mu_0}{4\sigma} t_n^\d \|A_{I_n} x_n^\d - y_{I_n}^\d\|^2 - t_n^\d \l A_{I_n} x_n^\d - y_{I_n}^\d, A_{I_n} x_n^\d - y_{I_n}\r \\
& \le - c_0 t_n^\d \|A_{I_n} x_n^\d - y_{I_n}^\d\|^2
+  t_n^\d \d_{I_n} \|A_{I_n} x_n^\d-y_{I_n}^\d\|\\
& \le \frac{1}{4c_0} t_n^\d \d_{I_n}^2 \le \frac{\mu_1}{4c_0} \d_{I_n}^2.
\end{align*}
By taking the expectation and using $\sum_{I: |I|=b}$ to denote a sum over all subsets $I\subset \{1, \cdots, n\}$ with $|I|=b$, we have
\begin{align*}
\EE[\Delta_{n+1}^\d] - \EE[\Delta_n^\d] 
& \le \frac{\mu_1}{4c_0} \EE[\d_{I_n}^2] = \frac{\mu_1}{4 c_0} \frac{1}{{p \choose b}} \sum_{I: |I|=b} \d_I^2 = \frac{\mu_1}{4 c_0} \frac{1}{{p\choose b}} \sum_{I: |I|=b} \sum_{i \in I} \d_i^2\\ 
& = \frac{\mu_1}{4 c_0} \frac{1}{{p\choose b}} \sum_{i=1}^p \sum_{I: |I|=b \, \& \, i \in I } \d_i^2 = \frac{\mu_1}{4 c_0} \frac{{p-1\choose b-1}}{{p\choose b}} \sum_{i=1}^p \d_i^2 \\
& = \frac{\mu_1 b}{4 c_0 p} \d^2
\end{align*}
which shows the desired inequality.
\end{proof}

Next we consider Algorithm \ref{alg:SMD} with exact data and drop the superscript $\d$ for every quantity defined by the algorithm,  Thus $x_n, \xi_n$ denote the corresponding iterative sequences and $t_n$ denotes the step-size. 
We now show a convergence result for Algorithm \ref{alg:SMD} with exact data by demonstrating that
$\{x_n\}$ is a Cauchy sequence in $L^2(\Omega, X)$.

\begin{theorem}\label{rdbgm.thm2}
Let Assumption \ref{dgm.ass1} hold. Consider Algorithm \ref{alg:SMD} with exact data and assume that 
\begin{align}\label{smd.tn}
\mu_2\le  t_n \le  \min\left\{\frac{\mu_0 \|A_{I_n} x_n - y_{I_n}\|^2}{\|A_{I_n}^*(A_{I_n} x_n - y_{I_n})\|^2} , \mu_1 \right\} \quad \mbox{ when } A_{I_n} x_n \ne y_{I_n},
\end{align}
where $\mu_0$, $\mu_1$ and $\mu_2$ are positive numbers with $c_0 := 1- \mu_0/(4 \sigma)>0$. Then
$$
\EE[\|x_n-x^\dag\|^2] \to 0 \quad \mbox{ and } \quad
\EE\left[D_\R^{\xi_n}(x^\dag, x_n)\right] \to 0
$$
as $n \to \infty$, where $x^\dag$ denotes the unique $\R$-minimizing solution of (\ref{smd.1}).
\end{theorem}

\begin{proof}
Let $\hat x$ be any solution of (\ref{smd.1}) in $\mbox{dom}(\R)$ and define 
$$
\Delta_n := D_\R^{\xi_n} (\hat x, x_n). 
$$ 
By the similar argument in the proof of Lemma \ref{rdbgm.lem11} we can obtain 
$$
\Delta_{n+1} - \Delta_n \le - c_0 t_n \|A_{I_n} x_n - y_{I_n}\|^2 \le - c_0 \mu_2 \|A_{I_n} x_n - y_{I_n}\|^2. 
$$
Consequently 
\begin{align}\label{rdbgm.62}
&\EE[\Delta_{n+1}] - \EE[\Delta_n] \nonumber \\
& \le - c_0 \mu_2 \EE \left[ \|A_{I_n} x_n - y_{I_n}\|^2 \right] = - c_0 \mu_2 \EE\left[ \EE \left[ \|A_{I_n} x_n - y_{I_n}\|^2| {\mathcal F_n}\right]\right] \nonumber \\
& = - c_0 \mu_2 \EE\left[\frac{1}{{p\choose b}} \sum_{I: |I| = b} \|A_I x_n - y_I\|^2 \right]. 
\end{align}
This shows that $\{\EE[\Delta_n]\}$ is monotonically decreasing and therefore 
\begin{align}\label{rdbgm.63}
\lim_{n\to \infty} \EE[\Delta_n] \mbox{ exists} \quad \mbox{and} \quad
\lim_{n\to \infty} \Phi_n =0
\end{align}
where 
$$
\Phi_n:= \EE\left[\frac{1}{{p\choose b}} \sum_{I: |I|=b}  \|A_I x_n - y_I\|^2\right].
$$
If $\Phi_n=0$ for some $n$, we must have 
$$
\sum_{I: |I|=b} \|A_I x_n - y_I\|^2 =0 
$$
along any sample path $(I_0, \cdots, I_{n-1})$ since there exist only finite many such sample paths each with a positive probability; consequently $\xi_m = \xi_n$ and $x_m = x_n$ and hence $\Phi_m = 0$ for all $m\ge n$. Based on these properties of $\Phi_n$, it is possible to choose a
strictly increasing sequence $\{n_l\}$ of integers by setting $n_0:=0$ and, for each $l\ge 1$, by letting
$n_l$ be the first integer satisfying
$$
n_l \ge n_{l-1}+1 \quad \mbox{ and } \quad \Phi_{n_l} \le \Phi_{n_{l-1}}.
$$
It is easy to see that for this sequence $\{n_l\}$ there holds
\begin{align}\label{rdbgm.64}
\Phi_{n_l} \le \Phi_n, \quad \forall 0\le n \le n_l.
\end{align}

For the above chosen sequence $\{n_l\}$ of integers, we are now going to show that
\begin{align}\label{rdbgm.65}
\sup_{l\ge k} \EE\left[D_\R^{\xi_{n_k}}(x_{n_l}, x_{n_k})\right] \to 0 \quad \mbox{as } k \to \infty.
\end{align}
By the definition of the Bregman distance we have for any $l>k$ that
$$
D_\R^{\xi_{n_k}}(x_{n_l}, x_{n_k}) = \Delta_{n_k} - \Delta_{n_l} + \l \xi_{n_l}-\xi_{n_k}, x_{n_l} - \hat x\r.
$$
Taking the expectation gives
\begin{align}\label{rdbgm.66}
\EE\left[D_\R^{\xi_{n_k}}(x_{n_l}, x_{n_k})\right]
= \EE[\Delta_{n_k}] - \EE[\Delta_{n_l}] + \EE\left[\l \xi_{n_l}-\xi_{n_k}, x_{n_l} - \hat x\r\right].
\end{align}
By using the definition of $\{\xi_n\}$ we have
\begin{align*}
\l \xi_{n_l}-\xi_{n_k}, x_{n_l} - \hat x\r
& = \sum_{n=n_k}^{n_l-1} \l \xi_{n+1}-\xi_{n}, x_{n_l} - \hat x\r \\
& = -\sum_{n=n_k}^{n_l-1} t_n \l A_{I_n}^* (A_{I_n} x_n - y_{I_n}), x_{n_l}-\hat x\r \\
& = -\sum_{n=n_k}^{n_l-1} t_n \l A_{I_n} x_n - y_{I_n}, A_{I_n} x_{n_l}- y_{I_n}\r.
\end{align*}
Therefore, by taking the expectation and using the Cauchy-Schwarz inequality, we can obtain
\begin{align*}
& \left|\EE\left[\l \xi_{n_l}-\xi_{n_k}, x_{n_l} - \hat x\r\right]\right|
\le \mu_1 \sum_{n=n_k}^{n_l-1} \left|\EE\left[\l A_{I_n} x_n - y_{I_n}, A_{I_n} x_{n_l}- y_{I_n}\r \right]\right| \nonumber \\
& \qquad \qquad \qquad  \le \mu_1 \sum_{n=n_k}^{n_l-1} \left(\EE\left[\|A_{I_n} x_n - y_{I_n}\|^2\right]\right)^{1/2}
\left(\EE\left[\| A_{I_n} x_{n_l}- y_{I_n}\|^2 \right]\right)^{1/2}.
\end{align*}
Since $x_n $ is ${\mathcal F}_n$-measurable, we have
\begin{align*}
\EE\left[\|A_{I_n} x_n - y_{I_n}\|^2\right]
= \EE\left[\EE\left[\|A_{I_n} x_n - y_{I_n}\|^2|{\mathcal F}_n\right]\right] = \Phi_n.
\end{align*}
We can not treat the term $\EE\left[\| A_{I_n} x_{n_l}- y_{I_n}\|^2 \right]$ in the same way because
$x_{n_l}$ is not necessarily ${\mathcal F}_n$-measurable. However, by noting that
$$
\|A_{I_n} x_{n_l} -y_{I_n}\|^2 \le \sum_{I: |I|=b} \|A_I x_{n_l} - y_I\|^2,
$$
we have
$$
\EE\left[\| A_{I_n} x_{n_l}- y_{I_n}\|^2 \right]
\le \EE\left[\sum_{I: |I|=b} \|A_I x_{n_l} - y_I\|^2\right] \le {p\choose b} \Phi_{n_l}.
$$
Therefore, with $c_{p,b}:= {p\choose b}^{1/2}$, we obtain 
\begin{align}\label{rdbgm.70}
\left|\EE\left[\l \xi_{n_l}-\xi_{n_k}, x_{n_l} - \hat x\r\right]\right|
\le \mu_1 c_{p,b}\sum_{n=n_k}^{n_l-1} \Phi_n^{1/2} \Phi_{n_l}^{1/2}.
\end{align}
By virtue of (\ref{rdbgm.64}) and (\ref{rdbgm.62}) we then have
\begin{align}\label{rdbgm.69}
\left|\EE\left[\l \xi_{n_l}-\xi_{n_k}, x_{n_l} - \hat x\r\right]\right| \le \mu_1 c_{p,b} \sum_{n=n_k}^{n_l-1} \Phi_n
\le \frac{\mu_1 c_{p,b}}{c_0\mu_2} \left(\EE[\Delta_{n_k}] - \EE[\Delta_{n_l}]\right).
\end{align}
Combining this with (\ref{rdbgm.66}) yields
\begin{align*}
\EE\left[D_\R^{\xi_{n_k}}(x_{n_l}, x_{n_k})\right]
\le \left(1+\frac{\mu_1 c_{p,b}}{c_0\mu_2} \right)  \left(\EE[\Delta_{n_k}] - \EE[\Delta_{n_l}]\right)
\end{align*}
which together with the first equation in (\ref{rdbgm.63}) shows (\ref{rdbgm.65}) immediately.

By using the strong convexity of $\R$, we can obtain from (\ref{rdbgm.65}) that
$$
\sup_{l\ge k} \EE[\|x_{n_l}-x_{n_k}\|^2] \to 0 \quad \mbox{as } k\to \infty
$$
which means that $\{x_{n_l}\}$ is a Cauchy sequence in $L^2(\Omega, X)$. Thus there exists a random vector
$x^*\in L^2(\Omega, X)$ such that
\begin{align}\label{rdbgm.67}
\EE[\|x_{n_l} - x^*\|^2 ] \to 0 \quad \mbox{ as } l \to \infty.
\end{align}
By taking a subsequence of $\{n_l\}$ if necessary, we can obtain from (\ref{rdbgm.67}) and the second equation in
(\ref{rdbgm.63}) that
\begin{align}\label{rdbgm.68}
\lim_{l\to \infty} \|x_{n_l} - x^*\|=0 \quad \mbox{and} \quad
\lim_{l\to \infty} \sum_{I: |I|=b} \|A_I x_{n_l} - y_I\|^2 =0
\end{align}
almost surely. Consequently
$$
\sum_{I: |I|=b} \|A_I x^* - y_I\|^2 =0 \ \ \mbox{almost surely},
$$
i.e. $x^*$ is a solution of (\ref{smd.1}) almost surely.

We next show that $x^* \in \mbox{dom}(\R)$ almost surely. It suffices to show $\EE[\R(x^*)] <\infty$.
Recall that $\xi_{n_l} \in \p \R(x_{n_l})$, we have
\begin{align}\label{rdbgm.79}
\R(x_{n_l}) \le \R(x) + \l \xi_{n_l}, x_{n_l}- x\r, \quad \forall x \in X.
\end{align}
By using (\ref{rdbgm.79}) with $x = x^\dag$, taking the expectation, and noting $\xi_{n_0} = \xi_0 =0$,
we can obtain from (\ref{rdbgm.69}) that
\begin{align*}
\EE[\R(x_{n_l})] & \le \R(x^\dag) + \EE\left[\l \xi_{n_l}-\xi_{n_0}, x_{n_l}-x^\dag\r \right] \\
& \le \R(x^\dag) + \frac{\mu_1 c_{p,b}}{c_0\mu_2} \EE\left[D_{\R}^{\xi_0}(x^\dag, x_0)\right] \\
& =: C <\infty.
\end{align*}
Therefore, it follows from the first equation in (\ref{rdbgm.68}), the lower semi-continuity of $\R$
and Fatou's Lemma that
\begin{align}\label{rdbgm.71}
\EE[\R(x^*)] \le \EE\left[\liminf_{l\to \infty} \R(x_{n_l})\right]
\le \liminf_{l\to \infty} \EE\left[\R(x_{n_l})\right] \le C <\infty.
\end{align}
We have thus shown that $x^*\in L^2(\Omega, X)$ is a solution of (\ref{smd.1}) in $\mbox{dom}(\R)$ almost surely.

In order to proceed further, we will show that, for any $\hat x\in L^2(\Omega, X)$ that is a solution of (\ref{smd.1}) in $\mbox{dom}(\R)$ almost surely, there holds
\begin{align}\label{rdbgm.74}
\lim_{l\to \infty} \EE\left[\l \xi_{n_l}, x_{n_l}-\hat x\r\right] =0.
\end{align}
To see this, for any $k<l$ we write
\begin{align*}
\EE\left[\l \xi_{n_l}, x_{n_l}-\hat x\r\right]
 = \EE\left[\l \xi_{n_l} - \xi_{n_k}, x_{n_l} - \hat x\r\right] + \EE\left[\l \xi_{n_k}, x_{n_l} - \hat x\r \right].
\end{align*}
Since $\hat x\in L^2(\Omega, X)$ is a solution of (\ref{smd.1}) in $\mbox{dom}(\R)$ almost surely, by the definition
of $\{\xi_n\}$ and $\xi_0=0$ we can use the similar argument for deriving (\ref{rdbgm.70}) to obtain
\begin{align*}
\left|\EE\left[\l \xi_{n_k}, x_{n_l}-\hat x\r\right]\right| \le \mu_1 c_{p,b} \left(\sum_{n=0}^{n_k-1} \Phi_n^{1/2}\right) \Phi_{n_l}^{1/2}.
\end{align*}
This and the second equation in (\ref{rdbgm.63}) imply for any fixed $k$ that
$\EE[\l \xi_{n_k}, x_{n_l}-\hat x\r]\to 0$ as $l\to \infty$. Hence
\begin{align*}
\limsup_{l\to \infty} \left| \EE\left[\l \xi_{n_l}, x_{n_l}- \hat x\r\right]\right|
& \le \sup_{l\ge k} \left|\EE\left[\l \xi_{n_l} - \xi_{n_k}, x_{n_l} - \hat x\r\right]\right|
\end{align*}
for all $k$.  By virtue of (\ref{rdbgm.69}) and letting $k \to \infty$ we thus obtain (\ref{rdbgm.74}).

Based on (\ref{rdbgm.79}) with $x = x^*$ and (\ref{rdbgm.74}), we can obtain
\begin{align*}
\limsup_{l\to \infty} \EE[\R(x_{n_l})]
\le \EE\left[\R(x^*)\right] + \lim_{l\to \infty} \EE\left[\l \xi_{n_l}, x_{n_l}-x^*\r\right]
= \EE[\R(x^*)]
\end{align*}
which together with (\ref{rdbgm.71}) then implies
\begin{align}\label{rdbgm.75}
\lim_{l\to \infty} \EE[\R(x_{n_l})] = \EE[\R(x^*)]
\end{align}
By using (\ref{rdbgm.75}), (\ref{rdbgm.79}) with $x=x^\dag$,
and (\ref{rdbgm.74}) with $\hat x = x^\dag$, we have
\begin{align*}
\EE[\R(x^*)] = \lim_{l\to \infty} \EE[\R(x_{n_l})]
\le \R(x^\dag) + \lim_{l\to \infty} \EE[\l \xi_{n_l}, x_{n_l}-x^\dag\r] = \R(x^\dag).
\end{align*}
Since $x^*$ is a solution of (\ref{smd.1}) in $\mbox{dom}(\R)$ almost surely, we have $\R(x^*) \ge \R(x^\dag)$ almost surely which implies $\EE[\R(x^*)]\ge \R(x^\dag)$. Consequently $\EE[\R(x^*)] = \R(x^\dag)$
and thus it follows from (\ref{rdbgm.75}) that
\begin{align}\label{rdbgm.77}
\lim_{l\to \infty} \EE[\R(x_{n_l})] = \R(x^\dag).
\end{align}
Finally, from (\ref{rdbgm.77}) and (\ref{rdbgm.74}) with $\hat x = x^\dag$ it follows that
\begin{align*}
\lim_{l\to \infty} \EE\left[D_\R^{\xi_{n_l}}(x^\dag, x_{n_l})\right]
= \lim_{l\to \infty} \left(\EE[\R(x^\dag) - \R(x_{n_l}) - \l \xi_{n_l}, x^\dag-x_{n_l}\r ]\right) = 0.
\end{align*}
By the monotonicity of $\{\EE[D_\R^{\xi_n}(x^\dag, x_n)]\}$ we can conclude
$$
\lim_{n\to \infty} \EE\left[D_\R^{\xi_n}(x^\dag, x_n)\right] = 0
$$
and hence $\lim_{n\to \infty} \EE[\|x_n-x^\dag\|^2] = 0$ by the strong convexity of $\R$.
The proof is therefore complete.
\end{proof}

In order to use Theorem \ref{rdbgm.thm2} to establish the convergence of Algorithm \ref{alg:SMD} with noisy data, we need to investigate the stability of the algorithm, i.e. the behavior of $x_n^\d$ and $\xi_n^\d$ as $\d\to 0$ for each fixed $n$. In the following we will provide stability results of Algorithm \ref{alg:SMD} under the following three choices of step-sizes:

\begin{enumerate}[leftmargin = 1cm]
\item[(s1)] $t_n^\d$ depends only on $I_n$, i.e.
$t_n^\d = t_{I_n}$ and $0<t_I < 4\sigma/\|A_I\|^2$
for all $I\subset \{1, \cdots, p\}$ with $|I|=b$.

\item[(s2)] $t_n^\d$ is chosen according to the formula
$$
t_n^\d =\min\left\{\frac{\mu_0 \|A_{I_n} x_n^\d - y_{I_n}^\d\|^2}{\|A_{I_n}^*(A_{I_n} x_n^\d - y_{I_n}^\d)\|^2}, \tilde \mu_1\right\},
$$
where $\mu_0$ and $\tilde \mu_1$ are two positive constants with $0<\mu_0<4\sigma$. 

\item[(s3)] In case the information of $\d_1, \cdots, \d_p$ is available, choose the step-size $t_n^\d$ according to the rule
\begin{align*}
t_n^\d = \left\{\begin{array}{lll}
\min\left\{\frac{\mu_0 \|A_{I_n} x_n^\d - y_{I_n}^\d\|^2}{\|A_{I_n}^*(A_{I_n} x_n^\d - y_{I_n}^\d)\|^2}, \tilde \mu_1\right\}  & \mbox{ if } \|A_{I_n} x_n^\d - y_{I_n}^\d\| >\tau \d_{I_n}, \\[1.2ex]
0 & \mbox{ otherwise},
\end{array}\right.
\end{align*}
where $\tau\ge 1$, $\mu_0>0$ and $\tilde \mu_1>0$ are preassigned numbers with $0<\mu_0<4\sigma$.
\end{enumerate}

The step-size chosen in (s1) is motivated by the Landweber iteration which uses constant step-size. Along the iterations in Algorithm \ref{alg:SMD}, when the same subset $I$ is repeatedly chosen, (s1) suggests to use the same step-size in computation. The step-sizes chosen by (s1) could be small and thus slows down the computation. The rules given in (s2) and (s3) use the adaptive strategy which may produce large step-sizes. The step-size chosen by (s2) is motivated by the deterministic minimal error method, see \cite{KNS2008}. 
The step-size given in (s3) is motivated by the work on deterministic Landweber-Kaczmarz method, see \cite{HLS2007,Jin2016}, and it incorporates the spirit of the discrepancy principle into the selection. 

Let $\{x_n, \xi_n\}$ be defined by Algorithm \ref{alg:SMD} with exact data, where the step-size $t_n$ is chosen by (s1), (s2) or (s3) with the superscript $\d$ dropped and with $\d_{I_n}$ replaced by $0$. It is easy to see that $t_n$ satisfies (\ref{smd.tn}) in Theorem \ref{rdbgm.thm2}.
Therefore, Theorem \ref{rdbgm.thm2} is applicable to $\{x_n, \xi_n\}$. 

The following result gives a stability property of Algorithm \ref{alg:SMD} with the step-sizes chosen by (s1), (s2) or (s3). 

\begin{lemma}\label{SMD.lem7}
Let Assumption \ref{dgm.ass1} hold. Consider Algorithm \ref{alg:SMD} with the step-size chosen by (s1), (s2) or (s3). Then for each fixed integer $n\ge 0$ there hold
$$
\EE[\|x_n^\d-x_n\|^2] \to 0 \quad \mbox{and}\quad 
\EE[\|\xi_n^\d - \xi_n\|^2] \to 0
$$
as $\d\to 0$.
\end{lemma}

\begin{proof}
We show the result by induction on $n$. The result is trivial for $n=0$ because $\xi_0^\d = \xi_0$. Assuming the result holds for some $n\ge 0$, we will show it also holds for $n+1$. Since there are only finitely many sample paths of the form $(I_0, \cdots, I_{n-1})$ each with positive probability, we can conclude from the induction hypothesis that 
\begin{align}\label{SMD-DP.1}
\|x_n^\d - x_n\| \to 0 \quad \mbox{ and } \quad 
\|\xi_n^\d - \xi_n\| \to 0 \quad \mbox{as } \d \to 0
\end{align}
along every sample path $(I_0, \cdots, I_{n-1})$. We now show that along each sample path $(I_0, \cdots, I_{n-1}, I_n)$ there hold 
\begin{align}\label{SMD-DP.2}
\|x_{n+1}^\d - x_{n+1}\|\to 0 \quad \mbox{ and } \quad 
\|\xi_{n+1}^\d - \xi_{n+1}\| \to 0  \quad \mbox{as } \d \to 0. 
\end{align}
To see this, we first show $\|\xi_{n+1}^\d -\xi_{n+1}\| \to 0$ as $\d \to 0$ by considering two cases. 

We first consider the case $A_{I_n} x_n - y_{I_n}=0$. For this case we have $\xi_{n+1} = \xi_n$ and hence 
\begin{align*}
\xi_{n+1}^\d -\xi_{n+1} 
= \xi_n^\d - \xi_n - t_n^\d A_{I_n}^*\left(A_{I_n} (x_n^\d-x_n) - (y_{I_n}^\d-y_{I_n})\right).
\end{align*}
Note that for all the choices of the step-sizes given in (s1), (s2) and (s3) we have $0\le t_n^\d \le \hat \mu_1$, where $\hat \mu_1$ is a constant independent of $n$ and $\d$. Note also that $\d_{I_n}\le \d$. Therefore  
\begin{align*}
\|\xi_{n+1}^\d -\xi_{n+1}\|
\le \|\xi_n^\d - \xi_n\| + \hat \mu_1 \|A_{I_n}\| \left( \|A_{I_n}\| \|x_n^\d-x_n\| + \d\right).
\end{align*}
Consequently, by using (\ref{SMD-DP.1}) we can obtain $\|\xi_{n+1}^\d - \xi_{n+1}\|\to 0$ as $\d \to 0$. 

Next we consider the case $A_{I_n} x_n-y_{I_n}\ne 0$. By using (\ref{SMD-DP.1}) we have $\|A_{I_n} x_n^\d - y_{I_n}^\d\| >\tau \d_{I_n}$ when $\d>0$ is sufficiently small. Note that 
\begin{align*}
\l A_{I_n}^* (A_{I_n} x_n -y_{I_n}), x_n - x^\dag\r 
=\| A_{I_n} x_n - y_{I_n}\|^2 >0 
\end{align*}
which implies that $A_{I_n}^*(A_{I_n} x_n - y_{I_n})\ne 0$. Thus the step-size chosen by (s2) or (s3) is given by  
$$
t_n^\d = \min\left\{\frac{\mu_0\|A_{I_n} x_n^\d - y_{I_n}^\d\|^2}{\|A_{I_n}^* (A_{I_n} x_n^\d - y_{I_n}^\d)\|^2}, \tilde \mu_1\right\}.
$$
By using again (\ref{SMD-DP.1}) we have 
\begin{align}\label{SMD-DP.19}
t_n^\d \to t_n \quad \mbox{ as } \d\to 0. 
\end{align}
For the step-size chosen by (s1) the equation (\ref{SMD-DP.19}) holds trivially as $t_n^\d = t_n = t_{I_n}$. Therefore, by noting that 
\begin{align*}
\xi_{n+1}^\d - \xi_{n+1} 
& = \xi_n^\d - \xi_n - t_n^\d A_{I_n}^*(A_{I_n} x_n^\d - y_{I_n}^\d) + t_n A_{I_n}^* (A_{I_n} x_n - y_{I_n})\\
& = \xi_n^\d - \xi_n + (t_n-t_n^\d) A_{I_n}^*(A_{I_n} x_n - y_{I_n}) \\
& \quad \, + t_n^\d A_{I_n}^* \left(A_{I_n}(x_n-x_n^\d) + (y_{I_n}^\d - y_{I_n})\right),
\end{align*}
we can use (\ref{SMD-DP.1}) and (\ref{SMD-DP.19}) to conclude again $\|\xi_{n+1}^\d - \xi_{n+1}\|\to 0$ as $\d \to 0$. 

Now by virtue of the formula $x_{n+1}^\d = \nabla \R^*(\xi_{n+1}^\d)$, $x_{n+1} = \nabla \R^*(\xi_{n+1})$ and the continuity of $\nabla \R^*$, we have $\|x_{n+1}^\d-x_{n+1}\|\to 0$ as $\d \to 0$. We thus obtain (\ref{SMD-DP.2}). 

Finally, since there are only finitely many sample paths of the form $(I_0, \cdots, I_n)$, we can use (\ref{SMD-DP.2}) to conclude 
$$
\EE\left[\|x_{n+1}^\d - x_{n+1}\|^2\right] \to 0 \quad \mbox{and} \quad \EE\left[\|\xi_{n+1}^\d - \xi_{n+1}\|^2\right] \to 0 
$$
as $\d \to 0$. The proof is thus complete by the induction principle. 
\end{proof}

\begin{theorem}\label{SMD.thm2}
Let Assumption \ref{dgm.ass1} hold. Consider Algorithm \ref{alg:SMD} with the step-size chosen by (s1), (s2) or (s3). If the integer $n_\d$ is chosen such that $n_\d \to \infty$ and $\d^2 n_\d\to 0$ as $\d \to 0$, then
$$
\EE\left[\|x_{n_\d}^\d - x^\dag\|^2\right] \to 0 \quad \mbox{ and } \quad
\EE\left[D_\R^{\xi_{n_\d}^\d} (x^\dag, x_{n_\d}^\d)\right] \to 0
$$
as $\d \to 0$.
\end{theorem}

\begin{proof}
Let $\Delta_n^\d:= D_\R^{\xi_n^\d}(x^\dag, x_n^\d)$ for all $n \ge 0$. By the strong convexity of $\R$
it suffices to show $\EE[\Delta_{n_\d}^\d]\to 0$ as $\d \to 0$. According to Lemma \ref{rdbgm.lem11} we have
$$
\EE\left[\Delta_{n+1}^\d\right] - \EE\left[\Delta_n^\d\right] \le \frac{\mu_1 b}{4 c_0 p} \d^2
$$
for all $n \ge 0$. Let $n$ be any fixed integer. Since $n_\d\to \infty$ as $\d \to 0$, we have $n_\d>n$ for small $\d>0$.
Thus, we may repeatedly use the above inequality to obtain
\begin{align*}
\EE\left[\Delta_{n_\d}^\d\right] \le \EE\left[\Delta_n^\d\right] + \frac{\mu_1 b}{4 c_0 p} (n_\d-n) \d^2.
\end{align*}
Since $\d^2 n_\d\to 0$ as $\d \to 0$, we thus have
\begin{align*}
& \limsup_{\d\to 0} \EE\left[\Delta_{n_\d}^\d\right] \\
& \le \limsup_{\d\to 0} \EE\left[\Delta_n^\d\right]
= \limsup_{\d\to 0} \left(\R(x^\dag) - \EE\left[\R(x_n^\d)\right] - \EE\left[\l \xi_n^\d, x^\dag-x_n^\d\r\right]\right)
\end{align*}
for any $n\ge 0$. With the help of Cauchy-Schwarz inequality we have
\begin{align*}
 &\left|\EE\left[\l \xi_n^\d, x^\dag-x_n^\d\r\right] - \EE\left[\l \xi_n, x^\dag - x_n\r \right]\right|  \\
& \qquad \qquad \qquad \le \left|\EE\left[\l \xi_n^\d-\xi_n, x^\dag-x_n^\d\r\right]\right|
+ \left|\EE\left[\l \xi_n, x_n-x_n^\d\r\right]\right| \\
& \qquad \qquad\qquad\le \left(\EE\left[\|\xi_n^\d-\xi_n\|^2\right]\right)^{1/2} \left(\EE\left[\|x_n^\d-x^\dag\|^2\right]\right)^{1/2} \\
& \qquad \qquad\qquad\quad \, + \left(\EE[\|\xi_n\|^2]\right)^{1/2} \left(\EE\left[\|x_n-x_n^\d\|^2\right]\right)^{1/2}.
\end{align*}
Thus, we may use Lemma \ref{SMD.lem7} to conclude
\begin{align*}
\lim_{\d\to 0} \EE\left[\l \xi_n^\d, x^\dag-x_n^\d\r\right]
= \EE\left[\l \xi_n, x^\dag - x_n\r \right].
\end{align*}
According to the proof of Lemma \ref{SMD.lem7} we have $\|x_n^\d - x_n\|\to 0$ as $\d \to 0$ along each sample path $(I_0, \cdots, I_{n-1})$. Thus, from the lower
semi-continuity of $\R$ and Fatou's lemma it follows
$$
\EE[\R(x_n)] \le \EE\left[\liminf_{\d\to 0} \R(x_n^\d)\right] \le \liminf_{\d\to 0} \EE\left[\R(x_n^\d)\right].
$$
Consequently
\begin{align*}
\limsup_{\d \to 0} \EE\left[\Delta_{n_\d}^\d\right]
& \le \R(x^\dag) -\liminf_{\d\to 0} \EE\left[\R(x_n^\d)\right]-\lim_{\d \to 0} \EE\left[\l \xi_n^\d, x^\dag-x_n^\d\r\right] \\
& \le \R(x^\dag) - \EE[\R(x_n)] - \EE\left[\l \xi_n, x^\dag - x_n\r \right] \\
& = \EE[D_\R^{\xi_n}(x^\dag, x_n)]
\end{align*}
for all $n\ge 0$. Letting $n\to \infty$ and using Theorem \ref{rdbgm.thm2} we therefore obtain $\EE[\Delta_{n_\d}^\d] \to 0$
as $\d \to 0$.
\end{proof}

Theorem \ref{SMD.thm2} provides convergence results on Algorithm \ref{alg:SMD} under various choices of the step-sizes. For the step-size chosen by (s3), if further conditions are imposed and $\tau$ and $\mu_0$, we can show that the iteration error of Algorithm \ref{alg:SMD} always decreases. 

\begin{lemma}\label{SMD.lem6}
Let Assumption \ref{dgm.ass1} hold. Consider Algorithm \ref{alg:SMD} with the step-size chosen by (s3). If $\mu_0>0$ and $\tau>1$ are chosen such that
\begin{align}\label{smd.mutau}
1-\frac{1}{\tau}-\frac{\mu_0}{4\sigma}\ge 0,
\end{align}
then $\Delta_{n+1}^\d\le \Delta_n^\d$ and consequently 
$
\EE\left[\Delta_{n+1}^\d\right] \le \EE\left[\Delta_n^\d\right]
$
for all $n \ge 0$, where $\Delta_n^\d := D_\R^{\xi_n^\d}(x^\dag, x_n^\d)$.
\end{lemma}

\begin{proof}
According to the proof of Lemma \ref{rdbgm.lem11}, we have 
\begin{align*}
\Delta_{n+1}^\d - \Delta_n^\d 
& \le \frac{1}{4\sigma} \left(t_n^\d\right)^2 \|A_{I_n}^*(A_{I_n} x_n^\d - y_{I_n}^\d)\|^2 - t_n^\d \l A_{I_n} x_n^\d - y_{I_n}^\d, A_{I_n} x_n^\d - y_{I_n}\r \\
& \le \frac{1}{4\sigma} \left(t_n^\d\right)^2 \|A_{I_n}^*(A_{I_n} x_n^\d - y_{I_n}^\d)\|^2 - t_n^\d \| A_{I_n} x_n^\d - y_{I_n}^\d\|^2 \\
& \quad\,+ \d_{I_n} t_n^\d \|A_{I_n} x_n^\d - y_{I_n}^\d\|.  
\end{align*}
By the definition of $t_n^\d$ we have 
\begin{align*}
&t_n^\d \|A_{I_n}^* (A_{I_n} x_n^\d - y_{I_n}^\d)\|^2 \le \mu_0 \|A_{I_n} x_n^\d - y_{I_n}^\d\|^2,\\
& \d_{I_n} t_n^\d \|A_{I_n} x_n^\d - y_{I_n}^\d\| \le \frac{1}{\tau} t_n^\d \|A_{I_n} x_n^\d - y_{I_n}^\d\|^2. 
\end{align*}
Therefore 
\begin{align*}
\Delta_{n+1}^\d - \Delta_n^\d 
\le - \left(1-\frac{1}{\tau} - \frac{\mu_0}{4\sigma}\right) t_n^\d \|A_{I_n} x_n^\d - y_{I_n}^\d\|^2. 
\end{align*}
Since $\mu_0$ and $\tau$ satisfy (\ref{smd.mutau}), this implies $\Delta_{n+1}^\d\le \Delta_n^\d$. By taking the expectation we then obtain $\EE[\Delta_{n+1}^\d] \le \EE[\Delta_n^\d]$. 
\end{proof}

\begin{theorem}\label{SMD.thm3}
Let Assumption \ref{dgm.ass1} hold. Consider Algorithm \ref{alg:SMD} with step-size chosen by (s3), where $\tau>1$ and $\mu_0>0$ are chosen to satisfy (\ref{smd.mutau}). If the integer $n_\d$ is chosen such that $n_\d \to \infty$ as $\d \to 0$, then
$$
\EE\left[\|x_{n_\d}^\d - x^\dag\|^2\right] \to 0 \quad \mbox{ and } \quad
\EE\left[D_\R^{\xi_{n_\d}^\d} (x^\dag, x_{n_\d}^\d)\right] \to 0
$$
as $\d \to 0$.
\end{theorem}

\begin{proof}
Let $n\ge 0$ be any fixed integer. Since $n_\d \to \infty$ as $\d \to 0$, we have $n_\d>n$ for sufficiently small $\d>0$. 
It then follows from Lemma \ref{SMD.lem6} that 
$$
\EE\left[\Delta_{n_\d}^\d\right] \le \EE\left[\Delta_n^\d\right]
$$
By using Lemma \ref{SMD.lem7} and following the proof of Theorem \ref{SMD.thm2} we can further conclude 
$$
\limsup_{\d\to 0} \EE\left[\Delta_{n_\d}^\d\right] \le \EE[\Delta_n]
$$
for all integers $n\ge 0$. Finally, by letting $n \to \infty$, we can use Theorem \ref{rdbgm.thm2} to obtain $\EE[\Delta_{n_\d}^\d] \to 0$ as $\d \to 0$. 
\end{proof}

We remark that, unlike Theorem \ref{SMD.thm2}, the convergence result given in Theorem \ref{SMD.thm3} requires only $n_\d$ to satisfy $n_\d\to \infty $ as $\d \to 0$. This is not surprising because the discrepancy principle has been incorporated into the choice of the step-size by (s3). Theorem \ref{SMD.thm3} can be viewed as a stochastic extension of the corresponding result for the deterministic Landweber-Kaczmarz method (\cite{Jin2016}) for solving ill-posed problems. 
It should be pointed out that, due to (\ref{smd.mutau}), the application of Theorem \ref{SMD.thm3} requires either $\tau$ to be sufficiently large or $\mu_0$ to be sufficiently small which may lead the corresponding algorithm to lose accuracy or converge slowly. However,  our numerical simulations in Section \ref{sect5} demonstrate that using step-sizes chosen by (s3) without satisfying (\ref{smd.mutau}) still has the effect of decreasing errors as the iteration proceeds.

\section{\bf Rate of convergence}\label{sect4} 
\setcounter{equation}{0}

In Section \ref{sect3} we have proved the convergence of Algorithm \ref{alg:SMD} under various choices of the step-sizes. In this section we will focus on deriving convergence rates. For ill-posed problems, this requires the sought solution to satisfy suitable source conditions. We will consider the benchmark source condition of the form (\cite{BO2004})
\begin{align}\label{source}
A^* \la^\dag \in \p \R(x^\dag)
\end{align}
for some $\la^\dag := (\la_{,1}^\dag, \cdots, \la_{, p}^\dag) \in Y:=Y_1\times \cdots \times Y_p$, where, here and below, for any $\la \in Y$ we use $\la_{,i}$ to denote its $i$-th component in $Y_i$.  
Deriving convergence rates under general choices of step-sizes can be very challenging. Therefore, in this section we will restrict ourselves to the step-sizes chosen by (s1) that includes the constant step-sizes. The corresponding algorithm can be reformulated as follows. 

\begin{algorithm}\label{alg:SMD2}
Fix a batch size $b$, pick the initial guess $\xi_0 =0$ in $X^*$ and step-sizes $t_I$ for all $I \subset \{1, \cdots, p\}$ with $|I|=b$. Set $\xi_0^\d := \xi_0$. For $n\ge 0$ do the following:

\begin{enumerate} [leftmargin = 0.8cm]
\item[\emph{(i)}] Calculate $x_n^\d \in X$ by solving (\ref{smd.35});

\item[\emph{(ii)}] Randomly select a subset $I_n\subset \{1, \cdots, p\}$ with $|I_n|=b$ via the uniform distribution;

\item[\emph{(iii)}] Define $\xi_{n+1}^\d \in X^*$ by (\ref{smd.25}). 
\end{enumerate}
\end{algorithm}

For Algorithm \ref{alg:SMD2} we have the following convergence rate result. 

\begin{theorem}\label{SMD.thm11}
Let Assumption \ref{dgm.ass1} hold and consider Algorithm \ref{alg:SMD2}. Assume that $0<t_I < 4 \sigma/\|A_I\|^2$ for all subsets $I \subset \{1, \cdots, p\}$ with $|I|=b$; furthermore $t_I = t$ with a constant $t$ for all such $I$ in case $b>1$. If the sought solution $x^\dag$ satisfies the source condition (\ref{source}) and the integer $n_\d$ is chosen such that $1+ \frac{b}{p} n_\d$ is of the magnitude of $\d^{-1}$, then
$$
{\mathbb E}\left[\|x_{n_\d}^\d-x^\dag\|^2\right] \le C \d,
$$
where $C$ is a positive constant independent of $\d$, $p$ and $b$.
\end{theorem}

When $X$ is a Hilbert space and $\R(x) = \|x\|^2/2$, the stochastic mirror descent method becomes the stochastic gradient method. In the existing literature on stochastic gradient method for solving ill-posed problems, sub-optimal convergence rates have been derived for diminishing step-sizes under general source conditions on the sought solution, see \cite{JL2019,LM2021}. Our result given in Theorem \ref{SMD.thm11} supplements these results by demonstrating that the stochastic gradient method can achieve the order optimal convergence rate $O(\d)$ under constant step-sizes if the sought solution satisfies the source condition (\ref{source}). 

The result in Theorem \ref{SMD.thm11} also demonstrates the role played by the batch size $b$: to achieve the same convergence rate, less number of iterations is required if a larger batch size $b$ is used. However, this does not mean the corresponding algorithm with larger batch size $b$ is faster because the computational time at each iteration can increase as $b$ increases. How to determine a batch size with the best performance is an outstanding challenging question. 

The proof of Theorem \ref{SMD.thm11} is based on considering an equivalent formulation of Algorithm \ref{alg:SMD2} which we will derive in the following. The derivation is based on applying the randomized 
block gradient method (\cite{LX2015,N2012,RT2014}) to the dual problem of (\ref{smd.2}) with $y$ replaced by $y^{\d}$.
The associated Lagrangian function is
$$
L(x; \la) := \R(x) -  \l \la, A x - y^\d\r,
$$
where $x \in X$ and $\la \in Y$. Thus the dual function is
\begin{align*}
\inf_{x\in X} L(x; \la)
& = -\sup_{x\in X} \left\{ \left\l A^* \la, x \right\r - \R(x)\right\} + \l \la, y^\d \r = - \R^*\left(A^* \la\right) + \l \la, y^\d\r
\end{align*}
which means the dual problem is
\begin{align}\label{rbdgm.3}
\min_{\la \in Y} \left\{d_{y^\d}(\la):=\R^*\left(A^* \la\right) - \l \la, y^\d\r\right\}.
\end{align}
Note that $\R^*$ is continuous differentiable and
$$
d_{y^\d}(\la) = \R^* \left(\sum_{i=1}^p A_i^* \la_{,i}\right) - \sum_{i=1}^p \l \la_{, i}, y_i^\d\r. 
$$
Therefore we may solve (\ref{rbdgm.3}) by the randomized block gradient method which iteratively selects partial components of $\la$ at random to be updated by a partial gradient of $d_{y^\d}$ and leave other components unchanged. To be more precisely, for any $\la \in Y$ and any subset $I \subset \{1, \cdots, p\}$, let $\la_I$ denote the group of components $\la_{,i}$ of $\la$ with $i \in I$. Assume $\la_n^\d:=(\la_{n,1}^\d, \cdots, \la_{n,p}^\d) \in Y$ is a current iterate. We then choose a subset of indices $I_n$ from $\{1, \cdots, p\}$ with $|I_n|=b$ randomly with uniform distribution and
define $\la_{n+1}^\d := (\la_{n+1,1}^\d, \cdots, \la_{n+1,p}^\d)$ by setting $\la_{n+1,i}^\d = \la_{n,i}^\d$ if $i \not\in I_n$ and
\begin{align}\label{rbdgm.4}
\la_{n+1, I_n}^\d = \la_{n,I_n}^\d - t_{I_n} \nabla_{I_n} d_{y^\d}(\la_n^\d),
\end{align}
with step-sizes $t_{I_n}>0$ depending only on $I_n$, where, for each $I\subset \{1, \cdots, p\}$, $\nabla_I d_{y^\d}$ denotes the partial gradient of $d_{y^\d}$ with respect to $\la_I$. Note that
$$
\nabla_I d_{y^\d}(\la) = A_I \nabla \R^*\left(A^* \la\right) - y_I^\d.
$$
Therefore, by setting $x_n^\d := \nabla \R^*\left(A^* \la_n^\d\right)$, we can rewrite (\ref{rbdgm.4}) as
$$
\la_{n+1,I_n}^\d = \la_{n,I_n}^\d - t_{I_n} \left(A_{I_n} x_n^\d - y_{I_n}^\d\right).
$$
By using the definition of $x_n^\d$ and (\ref{dgm.22}) we have $A^* \la_n^\d \in \p \R(x_n^\d)$ which implies
$$
0 \in \p \left(\R - \l \la_n^\d, A \cdot \r\right)(x_n^\d)
$$
and thus
\begin{align}\label{rbdgm.5}
x_n^\d = \arg\min_{x\in X} \left\{ \R(x) - \l \la_n^\d, A x \r\right\}.
\end{align}
Combining the above analysis we thus obtain the following mini-batch randomized dual block gradient method for solving (\ref{rbdgm.1}) with noisy data.

\begin{algorithm}\label{alg:RDBGM}
{\it Fix a batch size $b$, pick the initial guess $\la_0:=(0, \cdots, 0)\in Y:=Y_1\times \cdots\times Y_p$ and step-sizes
$t_I$ for all $I \subset \{1, \cdots, p\}$ with $|I|=b$. Set $\la_0^\d := \la_0$. For $n\ge 0$ do the following:

\begin{enumerate} [leftmargin = 0.8cm]
\item[\emph{(i)}] Calculate $x_n^\d \in X$ by using (\ref{rbdgm.5});

\item[\emph{(ii)}] Randomly select a subset $I_n \subset \{1, \cdots, p\}$ with $|I_n|=b$ via the uniform distribution;

\item[\emph{(iii)}] Define $\la_{n+1}^\d := (\la_{n+1,1}^\d, \cdots, \la_{n+1,p}^\d)\in Y$ by
setting $\la_{n+1,I_n^c}^\d = \la_{n,I_n^c}^\d$ and
\begin{align}\label{rbdgm.6}
\la_{n+1, I_n}^\d = \la_{n, I_n}^\d - t_{I_n} \left(A_{I_n} x_n^\d - y_{I_n}^\d\right),
\end{align}
where $I_n^c$ denotes the complement of $I_n$ in $\{1, \cdots, p\}$. 
\end{enumerate}
}
\end{algorithm}

Let us demonstrate the equivalence between Algorithm \ref{alg:SMD2} and Algorithm \ref{alg:RDBGM}. If $\{x_n^\d, \la_n^\d\}$ is defined by Algorithm \ref{alg:RDBGM}, by defining $\xi_n^\d := A^* \la_n^\d$ we have 
\begin{align*}
\xi_{n+1}^\d & = A^* \la_{n+1}^\d = A_{I_n^c}^* \la_{n+1, I_n^c}^\d + A_{I_n}^* \la_{n+1, I_n}^\d \\
& = A_{I_n^c}^* \la_{n, I_n^c}^\d + A_{I_n}^* \left(\la_{n, I_n}^\d - t_{I_n} (A_{I_n} x_n^\d - y_{I_n}^\d)\right) \\
& = \left(A_{I_n^c}^* \la_{n, I_n^c}^\d + A_{I_n}^* \la_{n, I_n}^\d\right) - t_{I_n} A_{I_n}^* (A_{I_n} x_n^\d - y_{I_n}^\d) \\
& = \xi_n^\d - t_{I_n} A_{I_n}^* (A_{I_n} x_n^\d - y_{I_n}^\d).
\end{align*}
Therefore $\{x_n^\d, \xi_n^\d\}$ can be produced by Algorithm \ref{alg:SMD}. Conversely, if $\{x_n^\d, \xi_n^\d\}$ is defined by Algorithm \ref{alg:SMD}, by using $\xi_0^\d =0$ one can easily see that $\xi_n^\d \in \mbox{Ran}(A)$, the range of $A$. By writing $\xi_n^\d = A^* \la_n^\d$ for some $\la_n^\d = (\la_{n,1}^\d, \cdots, \la_{n, p}^\d) \in Y$, we have
\begin{align*}
\xi_{n+1}^\d  & = \xi_n^\d - t_{I_n} A_{I_n}^*(A_{I_n} x_n^\d - y_{I_n}^\d)  = A^* \la_n^\d - t_{I_n} A_{I_n}^*(A_{I_n} x_n^\d - y_{I_n}^\d) \\
& = A_{I_n^c}^* \la_{n, I_n^c}^\d 
+ A_{I_n}^* \left(\la_{n, I_n}^\d - t_{I_n} (A_{I_n} x_n^\d - y_{I_n}^\d)\right)
\end{align*}
Define $\la_{n+1}^\d \in Y$ by setting $\la_{n+1, I_n^c}^\d = \la_{n, I_n^c}^\d$ and $\la_{n+1, I_n}^\d = \la_{n, I_n}^\d - t_{I_n} (A_{I_n} x_n^\d - y_{I_n}^\d)$. We can see that $\xi_{n+1}^\d = A^* \la_{n+1}^\d$ and $\{x_n^\d, \la_n^\d\}$ can be generated by Algorithm \ref{alg:RDBGM}. Therefore we obtain 

\begin{lemma}\label{lem:equiv}
Algorithm \ref{alg:SMD2} and Algorithm \ref{alg:RDBGM} are equivalent.
\end{lemma}

Based on Algorithm \ref{alg:RDBGM}, we will derive the estimate on ${\mathbb E}[\|x_n^\d-x^\dag\|^2]$ under the benchmark source condition (\ref{source}). Let
$$
d_y(\la) := \R^*\left(A^* \la\right) - \l \la, y\r
$$
which is obtained from $d_{y^\d}(\la)$ with $y^\d$ replaced by $y$. From (\ref{source}) and (\ref{dgm.22}) it follows that $x^\dag = \nabla \R^*(A^*\la^\dag)$. Thus, by using $A x^\dag = y$ we have
\begin{align*}
\nabla d_y(\la^\dag) & = A \nabla \R^*\left(A^*\la^\dag\right) - y = A x^\dag - y = 0.
\end{align*}
Since $d_y(\la)$ is convex, this shows that $\la^\dag$ is a global minimizer of $d_y$ on $Y$.

According to (\ref{rbdgm.5}) we have $A^* \la_n^\d \in \p \R(x_n^\d)$. Thus, we may consider the Bregman distance
$$
\Delta_n^\d:= D_\R^{A^*\la_n^\d} (x^\dag, x_n^\d).
$$
By using (\ref{source}), $A^* \la_n^\d \in \p \R(x_n^\d)$, and (\ref{dgm.22}) we have
\begin{align*}
\R(x^\dag) + \R^*\left( A^*\la^\dag\right) = \left\l A^* \la^\dag, x^\dag\right\r, \quad
\R(x_n^\d) + \R^*\left(A^* \la_n^\d\right) = \left\l A^*\la_n^\d, x_n^\d\right\r.
\end{align*}
Therefore
\begin{align*}
\Delta_n^\d & = \R(x^\dag) - \R(x_n^\d) - \l A^*\la_n^\d, x^\dag - x_n^\d\r \\
& = \R^*\left( A^* \la_n^\d\right) -  \R^*\left( A^* \la^\dag\right)
-\l A^*\la_n^\d, x^\dag\r + \l A^*\la^\dag, x^\dag\r  \\
& = \left(\R^*\left( A^* \la_n^\d\right) - \l \la_n^\d, y\r \right)
- \left(\R^*(A^*\la^\dag) - \l \la^\dag, y\r\right) \\
& = d_y(\la_n^\d) - d_y(\la^\dag).
\end{align*}
Consequently, by virtue of the strong convexity of $\R$ we have
$$
\sigma \|x_n^\d -x^\dag\|^2 \le \Delta_n^\d \le d_y(\la_n^\d) - d_y(\la^\dag).
$$
Taking the expectation gives
\begin{align}\label{rbdgm.20}
\sigma \EE\left[\|x_n^\d-x^\dag\|^2\right] \le \EE\left[d_y(\la_n^\d) - d_y(\la^\dag)\right].
\end{align}
This demonstrates that we can achieve our goal by bounding $\EE[d_y(\la_n^\d)-d_y(\la^\dag)]$.

Note that the sequence $\{\la_n^\d\}$ is defined via the function $d_{y^\d}(\la)$. Although $\la^\dag$ is a minimizer of $d_y$, it may not be a minimizer of $d_{y^\d}$. In order to overcome this gap, we will make use of the
relation
\begin{align}\label{rbdgm.11}
d_{y^\d}(\la) - d_y(\la) = \l \la, y-y^{\d}\r
\end{align}
between $d_{y^\d}(\la)$ and $d_y(\la)$ for all $\la \in Y$.

\begin{lemma}\label{rbdgm.lem1}
Let Assumption \ref{dgm.ass1} hold and consider Algorithm \ref{alg:RDBGM}. Assume that $0<t_I <4\sigma/\|A_I\|^2$ for all $I \subset \{1, \dots, p\}$ with $|I|=b$. Then
$$
\EE\left[d_y(\la_{n+1}^\d)\right]
\le \EE\left[d_y(\la_n^\d)\right] + \frac{t_{\max} b}{2c_1 p} \d^2 - \frac{c_1}{2{p\choose b}} \sum_{I: |I|=b} t_I \EE\left[\|\nabla_I d_{y^\d}(\la_n^\d)\|^2\right]
$$
for all $n\ge 0$, where 
$$
t_{\max}:= \max_{I: |I|=b} t_I \quad \mbox{ and } \quad  c_1:= \min_{I: |I|=b} \left\{1- \frac{t_I\|A_I\|^2}{4\sigma}\right\}>0.
$$
\end{lemma}

\begin{proof}
For any $\la \in Y$ and any $I\subset \{1, \cdots, p\}$ with $|I|=b$ we may write $\la=(\la_I, \la_{I^c})$ up to a permutation. Since the partial gradient of $d_{y^\d}$ with respect to $\la_{I}$ is given by $\nabla_I d_{y^\d}(\la) = A_I\nabla \R^*(A^* \la) - y_I^{\d}$, it follows from (\ref{rbdgm.41}) that
$$
\|\nabla_I d_{y^\d}(\la_I, \la_{I^c}) - \nabla_I d_{y^\d}(\la_I + h, \la_{I^c})\| \le L_I\|h\|, \quad \forall h \in Y_I,
$$
where $L_I := \|A_I\|^2/(2\sigma)$.  Consequently, by recalling the relation between $\la_{n+1}^\d$ and $\la_n^\d$, we can obtain
\begin{align*}
d_{y^\d}(\la_{n+1}^\d)
& \le d_{y^\d}(\la_n^\d) + \l \nabla_{I_n} d_{y^\d}(\la_n^\d), \la_{n+1, I_n}^\d-\la_{n,I_n}^\d\r
+ \frac{L_{I_n}}{2} \|\la_{n+1,I_n}^\d-\la_{n,I_n}^\d\|^2\\
& = d_{y^\d}(\la_n^\d) - \left(1- \frac{L_{I_n} t_{I_n}}{2} \right) t_{I_n}  \| \nabla_{I_n} d_{y^\d}(\la_n^\d)\|^2.
\end{align*}
By using (\ref{rbdgm.11}) and the definition of $\la_{n+1}^\d$, we can see that
\begin{align*}
d_{y^\d}(\la_{n+1}^\d) - d_{y^\d}(\la_n^\d)
& = d_y(\la_{n+1}^\d) - d_y(\la_n^\d) + \l \la_{n+1,I_n}^\d - \la_{n,I_n}^\d, y_{I_n} - y_{I_n}^{\d}\r\\
& = d_y(\la_{n+1}^\d) - d_y(\la_n^\d) - t_{I_n} \l \nabla_{I_n} d_{y^\d}(\la_n^\d), y_{I_n} - y_{I_n}^{\d}\r.
\end{align*}
Combining the above two equations and using the definition of $c_1$, we thus have
\begin{align*}
& d_y(\la_{n+1}^\d) -d_y(\la_n^\d) \\
& \le t_{I_n} \l \nabla_{I_n} d_{y^\d}(\la_n^\d), y_{I_n} - y_{I_n}^{\d}\r
- \left(1- \frac{L_{I_n} t_{I_n}}{2} \right) t_{I_n}  \| \nabla_{I_n} d_{y^\d}(\la_n^\d) \|^2\\
& \le t_{I_n} \d_{I_n} \|\nabla_{I_n} d_{y^\d}(\la_n^\d)\|
- c_1 t_{I_n} \| \nabla_{I_n} d_{y^\d}(\la_n^\d) \|^2\\
& \le \frac{1}{2c_1} t_{I_n} \d_{I_n}^2 -  \frac{1}{2} c_1 t_{I_n} \| \nabla_{I_n} d_{y^\d}(\la_n^\d) \|^2.
\end{align*}
By taking the expectation, using (\ref{cond.1}), and noting that $\EE[\d_{I_n}^2] = \frac{b}{p}\d^2$, we finally obtain
\begin{align*}
\EE\left[d_y(\la_{n+1}^\d) -d_y(\la_n^\d)\right]
& \le \frac{t_{\max}}{2c_1}\EE\left[\d_{I_n}^2 \right]
- \frac{1}{2} c_1 \EE\left[\EE\left[t_{I_n} \|\nabla_{I_n} d_{y^\d}(\la_n^\d) \|^2|{\mathcal F}_n \right] \right]\\
& = \frac{t_{\max} b}{2c_1 p} \d^2 
- \frac{1}{2} c_1 \EE\left[\frac{1}{{p\choose b}}\sum_{I: |I|=b} t_I \|\nabla_I d_{y^\d}(\la_n^\d)\|^2\right]
\end{align*}
which completes the proof.
\end{proof}

\begin{lemma}\label{rbdgm.lem2}
Consider Algorithm \ref{alg:RDBGM} in which $0<t_I < 4\sigma/\|A_I\|^2$ for all $I \subset \{1, \cdots, p\}$ with $|I|=b$, furthermore $t_I = t$ with a constant $t$ for all such $I$ in case $b>1$. Assume (\ref{source}) holds and let  
$$
r_n := \left\{\begin{array}{lll}
\displaystyle{\sum_{i=1}^p \frac{1}{t_i} \|\la_{n, i}^\d - \la_{,i}^\dag\|^2} & \mbox{ if } b = 1,\\[2.8ex]
\displaystyle{\frac{1}{t} \|\la_n^\d - \la^\dag\|^2} & \mbox{ if } b>1
\end{array}\right.
$$ 
for all $n\ge 0$. Then 
\begin{align*}
\EE[r_{n+1}]
& \le \EE[r_n] + \frac{2 b}{p} \EE\left[d_y(\la^\dag)-d_y(\la_n^\d)\right]
+ \frac{2t_{\max}^{1/2} b}{p} \d \left(\EE[r_n]\right)^{1/2} \\
& \quad \, + \frac{1}{{p\choose b}} \sum_{I: |I|=b} t_I \EE\left[\|\nabla_I d_{y^\d}(\la_n^\d)\|^2\right]
\end{align*}
for all integers $n\ge 0$.
\end{lemma}

\begin{proof}
We will only prove the result for $b>1$; the result for $b=1$ can be proved similarly. By using the definition of $\la_{n+1}^\d$ and $t_{I_n} = t$ we have
\begin{align*}
r_{n+1} & = \frac{1}{t}\|\la_{n+1,I_n^c}^\d - \la_{I_n^c}^\dag\|^2 + \frac{1}{t} \|\la_{n+1, I_n}^\d - \la_{I_n}^\dag\|^2 \\
& = \frac{1}{t} \|\la_{n,I_n^c}^\d - \la_{I_n^c}^\dag\|^2 + \frac{1}{t} \|\la_{n, I_n}^\d - \la_{I_n}^\dag - t \nabla_{I_n}  d_{y^\d}(\la_n^\d) \|^2 \\
& = \frac{1}{t} \|\la_{n,I_n^c}^\d - \la_{I_n^c}^\dag\|^2 
+ \frac{1}{t} \|\la_{n,I_n}^\d - \la_{I_n}^\dag\|^2 - 2 \l \nabla_{I_n} d_{y^\d}(\la_n^\d), \la_{n,I_n}^\d - \la_{I_n}^\dag\r \\
& \quad \, + t \|\nabla_{I_n} d_{y^\d}(\la_n^\d)\|^2 \\
& = r_n - 2 \l \nabla_{I_n} d_{y^\d}(\la_n^\d), \la_{n, I_n}^\d -\la_{I_n}^\dag\r + t \|\nabla_{I_n} d_{y^\d}(\la_n^\d)\|^2.
\end{align*}
Taking the expectation and using (\ref{cond.1}), it gives
\begin{align*}
\EE[r_{n+1}]
& = \EE[r_n] - 2 \EE\left[\EE\left[\l \nabla_{I_n} d_{y^\d}(\la_n^\d), \la_{n, I_n}^\d -\la_{I_n}^\dag\r |{\mathcal F}_n\right]\right] \\
& \quad \,+ t \EE\left[\EE\left[\|\nabla_{I_n} d_{y^\d}(\la_n^\d)\|^2 |{\mathcal F}_n\right]\right] \\
& = \EE[r_n] - 2 \EE\left[\frac{1}{{p\choose b}} \sum_{I: |I|=b} \l \nabla_I d_{y^\d}(\la_n^\d), \la_{n,I}^\d -\la_I^\dag\r \right] \\
& \quad\,+ t \EE\left[\frac{1}{{p\choose b}} \sum_{I: |I|=b}  \|\nabla_I d_{y^\d}(\la_n^\d)\|^2\right].
\end{align*}
Noting that 
\begin{align*}
\sum_{I:|I|=b} \l \nabla_I d_{y^\d}(\la_n^\d), \la_{n,I}^\d -\la_I^\dag\r 
& = \sum_{I: |I|=b} \sum_{i\in I} \l \nabla_i d_{y^\d}(\la_n^\d), \la_{n,i}^\d-\la_{,i}^\dag\r \\
& = \sum_{i=1}^p \sum_{I: |I|=b \mbox{ and } i\in I} \l \nabla_i d_{y^\d}(\la_n^\d), \la_{n,i}^\d-\la_{,i}^\dag\r \\
& = {p-1\choose b-1} \sum_{i=1}^p \l \nabla_i d_{y^\d}(\la_n^\d), \la_{n,i}^\d-\la_{,i}^\dag\r\\
& = {p-1\choose b-1} \l \nabla d_{y^\d}(\la_n^\d), \la_n^\d -\la^\dag\r.
\end{align*}
Therefore 
\begin{align*}
\EE[r_{n+1}] & = \EE[r_n] + \frac{2 b}{p} \EE\left[\l \nabla d_{y^\d}(\la_n^\d), \la^\dag - \la_n^\d\r\right]
+ \frac{t}{{p\choose b}} \sum_{I:|I|=b}  \EE\left[\|\nabla_I d_{y^\d}(\la_n^\d)\|^2\right].
\end{align*}
By the convexity of $d_{y^\d}(\la)$, the Cauchy-Schwarz inequality and the definition of $r_n$, we can obtain 
\begin{align*}
\l \nabla d_{y^\d}(\la_n^\d), \la^\dag - \la_n^\d \r & \le d_{y^\d}(\la^\dag) - d_{y^\d}(\la_n^\d) \\
& = d_y(\la^\dag) - d_y(\la_n^\d) + \l \la_n^\d-\la^\dag, y^\d - y\r\\
& \le d_y(\la^\dag) - d_y(\la_n^\d) + \d \| \la_n^\d - \la^\dag\| \\
& = d_y(\la^\dag) - d_y(\la_n^\d) + \sqrt{t} \d r_n^{1/2}.
\end{align*}
Therefore
\begin{align*}
\EE[r_{n+1}]  & \le \EE[r_n] + \frac{2 b}{p} \EE\left[d_y(\la^\dag)-d_y(\la_n^\d)\right]
+ \frac{2\sqrt{t} b}{p} \d\EE[r_n^{1/2}] \\
& \quad \, + \frac{t}{{p\choose b}} \sum_{I: |I|=b} \EE\left[\|\nabla_I d_{y^\d}(\la_n^\d)\|^2\right].
\end{align*}
Since $\EE[r_n^{1/2}] \le (\EE[r_n])^{1/2}$, we thus obtain the desired inequality.
\end{proof}

\begin{lemma}\label{rbdgm.lem3}
Consider Algorithm \ref{alg:RDBGM} with the step-sizes chosen as in Lemma \ref{rbdgm.lem2}, there holds
\begin{align*}
& \EE \left[d_y(\la_{n+1}^\d) -d_y(\la^\dag) + \frac{1}{2} c_1 r_{n+1}\right]
+ \frac{b c_1}{p} \sum_{k=0}^n \EE\left[d_y(\la_k^\d) -d_y(\la^\dag)\right] \\
& \le M_0 + \frac{t_{\max} b}{2 c_1p} (n+1) \d^2 
+ \frac{t_{\max}^{1/2} b c_1}{p} \d \sum_{k=0}^n \left(\EE[r_k]\right)^{1/2}
\end{align*}
for all $n \ge 0$, where $M_0:= \frac{1}{2} c_1 r_0 + \left(d_y(\la_0) - d_y(\la^\dag)\right)$.
\end{lemma}

\begin{proof}
From Lemma \ref{rbdgm.lem1} and Lemma \ref{rbdgm.lem2} it follows that
\begin{align*}
& \EE\left[d_y(\la_{n+1}^\d) -d_y(\la^\dag) + \frac{1}{2} c_1 r_{n+1}\right] \\
& \le \EE\left[d_y(\la_n^\d) - d_y(\la^\dag) + \frac{1}{2} c_1 r_n\right]
- \frac{b c_1}{p} \EE\left[d_y(\la_n^\d) -d_y(\la^\dag)\right] \\
& \quad \, + \frac{t_{\max} b}{2 c_1 p} \d^2 + \frac{t_{\max}^{1/2} bc_1}{p} \d \left(\EE[r_n]\right)^{1/2}.
\end{align*}
Recursively using this inequality then completes the proof.
\end{proof}

In order to proceed further, we need an estimate on ${\mathbb E}[r_n]$. We will use the following elementary result.

\begin{lemma}\label{rbdgm.lem4}
Let $\{a_n\}$ and $\{b_n\}$ be two sequences of nonnegative numbers such that
$$
a_n^2 \le b_n^2 + c \sum_{j=0}^{n-1} a_j, \quad n=0, 1,\cdots,
$$
where $c \ge 0$ is a constant. If $\{b_n\}$ is non-decreasing, then
$$
a_n \le b_n + c n, \quad n=0, 1, \cdots.
$$
\end{lemma}

\begin{proof}
We show the result by induction. The result is trivial for $n =0$. Assume that the result is valid for all
$0\le n \le k$ for some $k\ge 0$. We show it is also true for $n = k+1$. If $a_{k+1} \le \max\{a_0, \cdots, a_k\}$, then
$a_{k+1} \le a_l$ for some $0\le l\le k$. Thus, by the induction hypothesis and the monotonicity of $\{b_n\}$ we have
$$
a_{k+1} \le a_l \le b_l + c l \le b_{k+1} + c (k+1).
$$
If $a_{k+1} > \max\{a_0, \cdots, a_k\}$, then
\begin{align*}
a_{k+1}^2 \le b_{k+1}^2 + c \sum_{j=0}^k a_j \le b_{k+1}^2 + c(k+1) a_{k+1}
\end{align*}
which implies that
\begin{align*}
\left(a_{k+1} - \frac{1}{2} c (k+1)\right)^2
& = a_{k+1}^2 - c (k+1) a_{k+1} + \frac{1}{4} c^2 (k+1)^2\\
& \le b_{k+1}^2 + \frac{1}{4} c^2 (k+1)^2 \\
& \le \left(b_{k+1} + \frac{1}{2} c (k+1)\right)^2.
\end{align*}
Taking square roots shows $a_{k+1} \le b_{k+1} + c (k+1)$ again.
\end{proof}

Now we are ready to show the following result which together with Lemma \ref{lem:equiv} implies Theorem \ref{SMD.thm11} immediately.

\begin{theorem}\label{rbdgm.thm1}
Let Assumption \ref{dgm.ass1} hold and consider Algorithm \ref{alg:RDBGM} with the step-sizes chosen as in Theorem \ref{SMD.thm11}. If the sought solution $x^\dag$ satisfies the source condition (\ref{source}), then for all $n\ge 0$ there holds
\begin{align}\label{rbdgm.21}
\EE\left[d_y(\la_n^\d)-d_y(\la^\dag)\right]
\le \frac{2M_0}{1+\frac{b}{p}c_1 n} + \frac{5t_{\max} b}{2 c_1 p} n \d^2
\end{align}
and consequently
\begin{align*}
\EE\left[\|x_n^\d - x^\dag\|^2\right]
\le \frac{1}{\sigma} \left(\frac{2 M_0}{1+\frac{b}{p}c_1 n} + \frac{5 t_{\max} b}{2c_1 p} n \d^2\right),
\end{align*}
where $t_{\max}$ and $c_1$ are defined in Lemma \ref{rbdgm.lem1}. 
\end{theorem}

\begin{proof}
According to (\ref{rbdgm.20}), it suffices to show (\ref{rbdgm.21}). Since $\la^\dag$ is a minimizer of $d_y(\la)$ over $Y$, we have $d_y(\la_n^\d) -d_y(\la^\dag) \ge 0$ for all $n \ge 0$. Thus, by noting that $\EE[r_0] = r_0 \le 2 M_0/c_1$, it follows from Lemma \ref{rbdgm.lem3} that
\begin{align*}
\EE\left[r_{n}\right]
& \le \frac{2 M_0}{c_1} + \frac{t_{\max} b }{c_1^2 p} n \d^2 + \frac{2 t_{\max}^{1/2} b}{p} \d \sum_{k=0}^{n-1} \left(\EE[r_k]\right)^{1/2}
\end{align*}
for all $n \ge 0$. Applying Lemma \ref{rbdgm.lem4} and using the inequality $\sqrt{a+b}\le \sqrt{a} +\sqrt{b}$ for any $a, b\ge 0$,
we can obtain
\begin{align*}
\left(\EE[r_n]\right)^{1/2}
\le \sqrt{\frac{2 M_0}{c_1} + \frac{t_{\max} b}{c_1^2 p} n \d^2} + \frac{2t_{\max}^{1/2} b}{p} n \d 
\le \sqrt{\frac{2M_0}{c_1}} + \left(\sqrt{\frac{b n}{c_1^2 p}} + \frac{2 b n}{p}\right) t_{\max}^{1/2}\d.
\end{align*}
Consequently
\begin{align*}
\sum_{k=0}^n \left(\EE[r_k]\right)^{1/2}
& \le \sqrt{\frac{2 M_0}{c_1}} (n+1)
+ \sum_{k=0}^n  \left(\sqrt{\frac{b k}{c_1^2 p}} + \frac{2 b k}{p}\right) t_{\max}^{1/2} \d \\
& \le \sqrt{\frac{2M_0}{c_1}} (n+1)
+ \left(\frac{2b^{1/2} (n+1)^{3/2}}{3 c_1 p^{1/2}} + \frac{b n(n+1)}{p}\right) t_{\max}^{1/2} \d.
\end{align*}
Combining this with the estimate in Lemma \ref{rbdgm.lem3} we obtain
\begin{align*}
&\EE\left[d_y(\la_{n+1}^\d) - d_y(\la^\dag)\right] +\frac{b c_1}{p} \sum_{k=0}^n \EE\left[d_y(\la_k^\d) - d_y(\la^\dag)\right] \\
& \le M_0 + \frac{t_{\max} b}{2c_1 p} (n+1) \d^2 + \frac{ b \sqrt{2 t_{\max} c_1 M_0}}{p} (n+1) \d \\
& \quad \, + \left(\frac{2 b^{3/2} (n+1)^{3/2}}{3 p^{3/2}} + \frac{b^2 c_1 (n+1)^2}{p^2}\right) t_{\max}\d^2.
\end{align*}
By using the inequality
$$
\frac{b\sqrt{2t_{\max} c_1 M_0}}{p} (n+1) \d
\le M_0 + \frac{b^2 c_1 (n+1)^2}{2 p^2} t_{\max} \d^2 
$$
we can further obtain
\begin{align}\label{rbdgm.7}
& \EE\left[d_y(\la_{n+1}^\d) - d_y(\la^\dag)\right] +\frac{b c_1}{p} \sum_{k=0}^n \EE\left[d_y(\la_k^\d) - d_y(\la^\dag)\right] \nonumber \\
& \le 2 M_0 + \frac{t_{\max} b}{p} \left(\frac{1}{2c_1} + \frac{2 b^{1/2} (n+1)^{1/2}}{3 p^{1/2}} + \frac{3 b c_1 (n+1)}{2 p}\right) (n+1) \d^2.
\end{align}
According to Lemma \ref{rbdgm.lem1} we have for any $0\le k\le n$ that
$$
\EE\left[d_y(\la_k^\d)\right] \ge \EE\left[d_y(\la_{n+1}^\d)\right] - \frac{t_{\max} b}{2c_1 p} (n+1-k) \d^2.
$$
Therefore
\begin{align*}
& \sum_{k=0}^n \EE\left[d_y(\la_k^\d)-d_y(\la^\dag)\right] \\
& \ge (n+1) \EE\left[d_y(\la_{n+1}^\d)-d_y(\la^\dag)\right] - \frac{t_{\max} b}{4 c_1 p} (n+1)(n+2) \d^2.
\end{align*}
Combining this with (\ref{rbdgm.7}) and noting that $n+2 \le 2 (n+1)$ and
$$
\frac{2 b^{1/2} (n+1)^{1/2}}{3 p^{1/2}} \le \frac{2}{c_1} + \frac{b c_1 (n+1)}{18 p},
$$
we obtain
\begin{align*}
& \left(1+ \frac{b c_1 (n+1)}{p}\right)  \EE\left[d_y(\la_{n+1}^\d)-d_y(\la^\dag)\right] \\
& \le 2 M_0 + \frac{t_{\max} b}{p} \left(\frac{1}{2c_1} + \frac{2 b^{1/2} (n+1)^{1/2}}{3 p^{1/2}} + \left(\frac{1}{2} + \frac{3c_1}{2}\right) \frac{b(n+1)}{p}\right) (n+1) \d^2 \\
& \le 2 M_0 + \frac{t_{\max} b}{p} \left(\a + \frac{\beta b(n+1)}{p}\right) (n+1) \d^2,
\end{align*}
where
$$
\a := \frac{5}{2c_1} \quad \mbox{and} \quad \beta := \frac{1}{2} + \frac{14}{9} c_1.
$$
Note that $\beta/\a<c_1$ as $c_1<1$. Dividing the both sides of the above equation by $1+b c_1(n+1)/p$ shows 
\begin{align*}
\EE\left[d_y(\la_{n+1}^\d)-d_y(\la^\dag)\right] 
& \le \frac{2 M_0}{1+\frac{b}{p}c_1 (n+1)} + \frac{5t_{\max} b}{2 c_1 p} (n+1) \d^2
\end{align*}
which completes the proof. 
\end{proof}

\section{\bf Numerical results}\label{sect5}
\setcounter{equation}{0}

In this section we will present numerical simulations to test the performance of the stochastic mirror descent method for solving linear ill-posed problems in which the sought solutions have various special features that require to make particular choices of the regularization functional $\R$ and the solution space $X$. Except Example \ref{SMD.ex2}, all the numerical simulations are performed on the linear ill-posed system 
\begin{align}\label{smd.testexample}
A_i x: = \int_a^b k(s_i, t) x(t) dt = y_i, \quad i = 1, \cdots, p
\end{align}
obtained from linear integral equations of first kind on $[a, b]$ by sampling at $s_i \in [a, b]$ with $i =1, \cdots, p$, where the kernel $k(\cdot, \cdot)$ is continuous on $[a, b]\times [a,b]$ and $s_i = a + (i-1) (b-a)/(p-1)$ for $i =1, \cdots, p$. 

\begin{example} \label{SMD.ex1}
Consider the linear system (\ref{smd.1}) where $X$ and $Y_i$ are all Hilbert spaces. In case the minimal norm solution is of interest, we may take $\R(x) = \|x\|^2/2$ in Algorithm \ref{alg:SMD}, then the definition of $x_n^\d$ shows $x_n^\d = \xi_n^\d$. Consequently, Algorithm \ref{alg:SMD} becomes the form
\begin{align}\label{SGD.1}
x_{n+1}^\d = x_n^\d - t_n^\d A_{I_n}^* (A_{I_n} x_n^\d - y_{I_n}^\d),
\end{align}
where $I_n \subset \{1, \cdots, p\}$ is a randomly selected subset with $|I_n|=b$ via the uniform distribution for a preassigned batch size $b$. This is exactly the mini-batch stochastic gradient descent method for solving linear ill-posed problems in Hilbert spaces which has been studied recently in \cite{JJL2017,JL2019,LM2021}. In these works the convergence analysis has been performed under diminishing step-sizes. Our work supplements the existing results by providing convergence results under new choices of step-sizes, in particular, for the method (\ref{SGD.1}) with batch size $b = 1$, we obtain convergence and convergence rate results for the step-size 
\begin{align}\label{SGD.s1}
t_n^\d = \frac{\mu_0}{\|A_{i_n}\|^2} \quad \mbox{ with } 0<\mu_0 <2
\end{align}
which obeys (s1) and (s2), where $i_n$ is the index chosen at the $n$-th iteration step.

We now test the performance of the method (\ref{SGD.1}) by considering the system (\ref{smd.testexample}) with $[a, b] = [-6, 6]$, $p= 1000$ and $k(s, t) := \varphi(s-t)$, where $\varphi(s) = (1+\cos(\pi s/3)) \chi_{\{|s|<3\}}$. We assume the sought solution is 
$$
x^\dag(t) = \sin(\pi t/12) + \sin(\pi t/3) + \frac{1}{200} t^2 (1-t).
$$
Instead of the exact data $y_i:=A_i x^\dag$, we use the noisy data  
\begin{align}\label{nd}
y_i^\d = y_i + \d_{rel} |y_i| \ep_i, \quad i = 1, \cdots, p, 
\end{align}
where $\d_{rel}$ is the relative noise level and $\ep_i$ are random noises obeying the standard Gaussian distribution, i.e. $\ep_i \sim N(0, 1)$. We execute the method (\ref{SGD.1}) with the batch size $b = 1$ and the initial guess $x_0^\d =0$ together with the step-sizes given by (\ref{SGD.s1}); the integrals involved in the method are approximated by the trapezoidal rule based on the partition of $[-6,6]$ into $p-1$ subintervals of equal length. 
To illustrate the dependence of convergence on the magnitude of step-size, we consider the three distinct values $\mu_0 = 1.0, 0.5$ and $0.1$. We also use noisy data with three distinct relative noise levels $\d_{rel} = 10^{-1}, 10^{-2}$ and $10^{-3}$. In Figure \ref{fig:SGD1} we plot the corresponding reconstruction errors; the first row plots the relative mean square errors $\EE[\|x_n^\d-x^\dag\|_{L^2}^2/\|x^\dag\|_{L^2}^2]$ which are calculated approximately by the average of $100$ independent runs and the second row plots $\|x_n^\d - x^\dag\|_{L^2}^2/\|x^\dag\|_{L^2}^2$ for a particular individual run. From these plots we can observe that, for each individual run, the iteration errors exhibit oscillations which are in particular dramatic when the noisy data have relative large noise levels. We can also observe that the method (\ref{SGD.1}) demonstrates the semi-convergence property, i.e. the iterates converges to the sought solution at the beginning and, after a critical number of iterations, the iterates begin to diverge. This is typical for any iterative methods for solving ill-posed problems. Furthermore, for a fixed noise level, the semi-convergence occurs earlier if a larger step-size is used. We also note that, for noisy data with relatively large noise level, if a large step-size is used, the iterates can quickly produce a reconstruction result with minimal error and then diverge quickly from the sought solution; this makes it difficult to decide how to stop the iteration to produce a good reconstruction result.  
The use of small step-sizes has the advantage of suppressing oscillations and deferring semi-convergence. However, it slows down the convergence and hence makes the method time-consuming.

\begin{figure}[htpb]
\centering
\includegraphics[width = 0.32\textwidth]{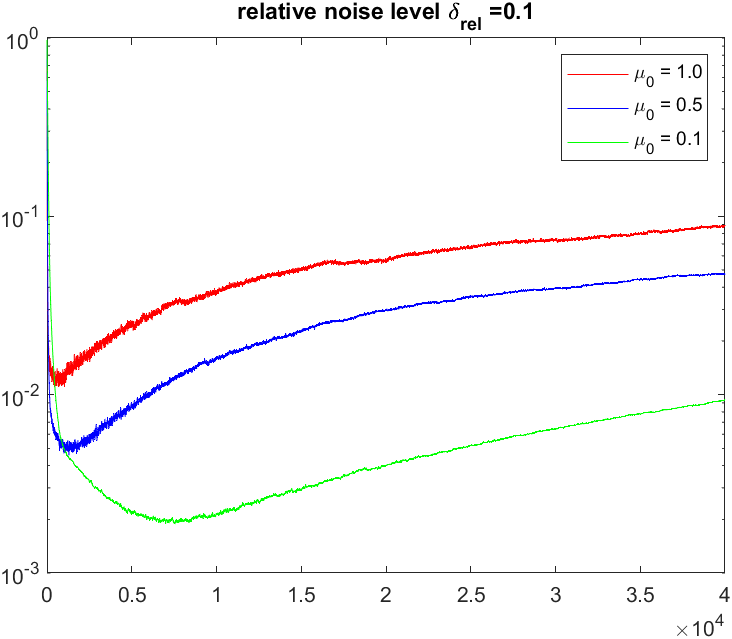}
\includegraphics[width = 0.32\textwidth]{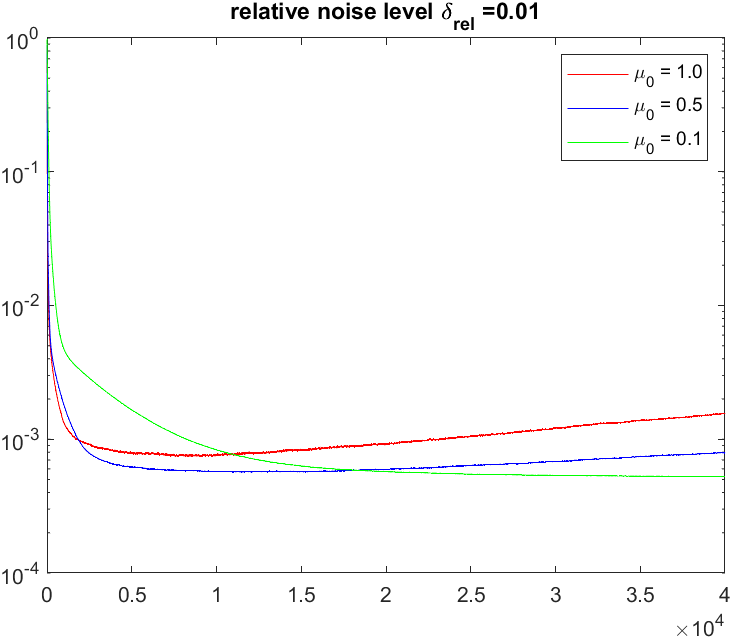}
\includegraphics[width = 0.32\textwidth]{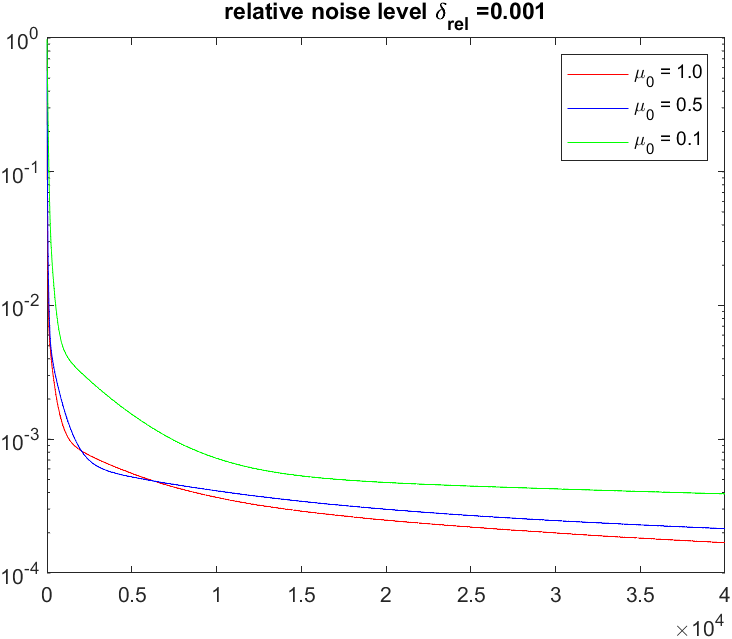}
\includegraphics[width = 0.32\textwidth]{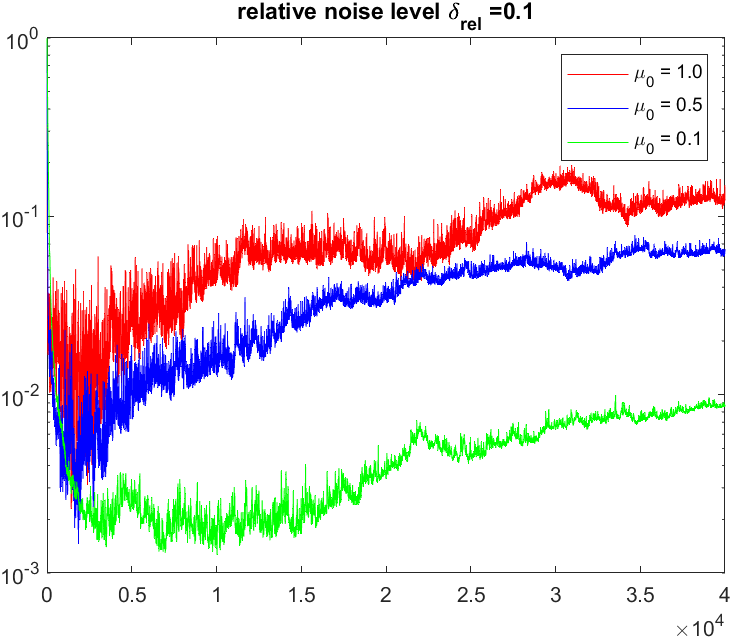}
\includegraphics[width = 0.32\textwidth]{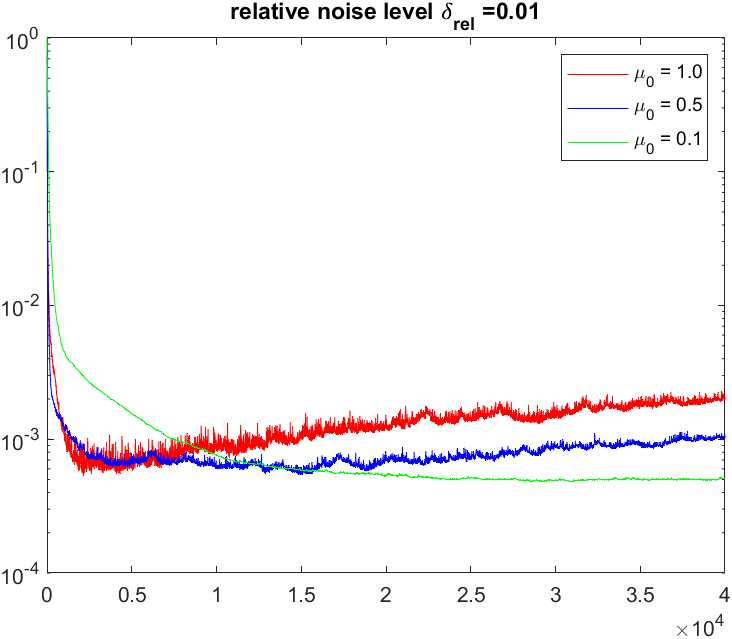}
\includegraphics[width = 0.32\textwidth]{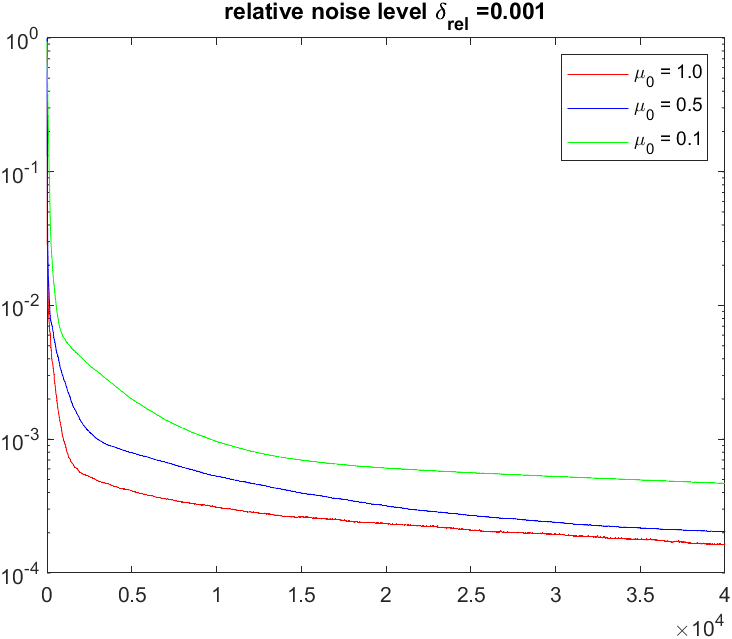}
\caption{Reconstruction errors of the stochastic gradient descent (\ref{SGD.1}) using  step-size (\ref{SGD.s1}) with various values of $\mu_0$. } \label{fig:SGD1}
\end{figure}

In order to efficiently suppress the oscillations and reduce the effect of semi-convergence, we next consider the method (\ref{SGD.1}) using the step-size chosen by (s3) whose realization relies on the information of noise level. We now assume that the noisy data have the form (\ref{nd}), where $\ep_i$ are noise uniformly distributed on $[-1, 1]$. Note that $\d_i:=\d_{rel} |y_i|$ are the noise levels with $|y_i^\d - y_i|\le \d_i$. Assuming the information on $\d_i$ is known, the step-size chosen by (s3) takes the form
\begin{align}\label{SGD.s3}
t_n^\d = \left\{\begin{array}{lll}
\mu_0/\|A_{i_n}\|^2 & \mbox{ if } |A_{i_n} x_n^\d - y_{i_n}^\d| >\tau \d_{i_n},\\
0 & \mbox{ otherwise} 
\end{array}\right.
\end{align}
which incorporates the spirit of the discrepancy principle. To illustrate the advantage of this choice of step-size, we compare the computed results with the ones obtained by the step-size chosen by (\ref{SGD.s1}). In Figure \ref{fig:SGD2} we present the numerical results of reconstruction errors by the method (\ref{SGD.1}) with batch size $b=1$ for various noise levels using the step-sizes chosen by (\ref{SGD.s1}) and (\ref{SGD.s3}) with $\mu_0 = 1$ and $\tau = 1$. where ``\texttt{DP}" and ``\texttt{No DP}" represent the results corresponding to the step-sizes chosen by (\ref{SGD.s3}) and (\ref{SGD.s1}) respectively. The first row in Figure \ref{fig:SGD2} plots the mean square errors which are calculate approximately by the average of 100 independent runs and the second row plots the reconstruction errors for a typical individual run. The results show clearly that using the step-size (\ref{SGD.s3}) can significantly suppress the oscillations in reconstruction errors in particular when the data are corrupted by noise with relatively large noise level. We have performed extensive simulations which indicate that, due to the regularization effect of the discrepancy principle,  using the step-size (\ref{SGD.s3}) has the tendency to decrease the iteration error as the number of iterations increases and has the ability to reduce the effect of semi-convergence.

\begin{figure}[htpb]
\centering
\includegraphics[width = 0.32\textwidth]{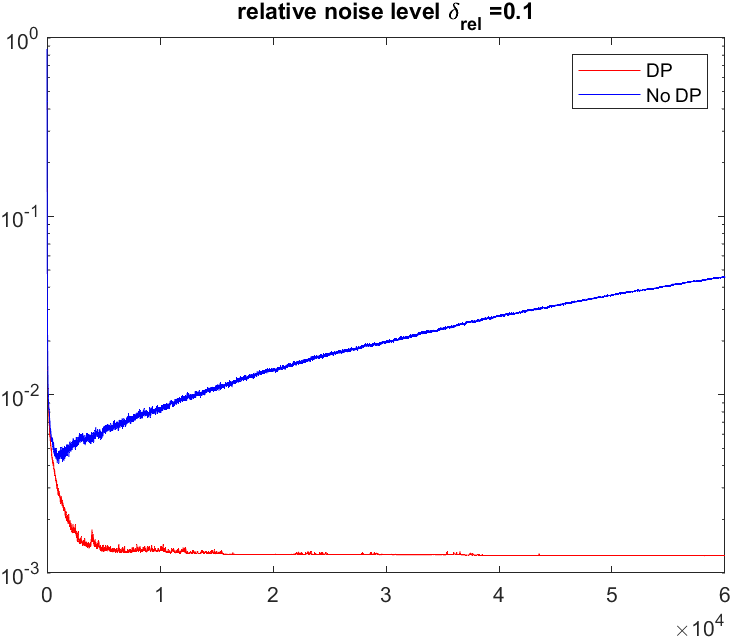}
\includegraphics[width = 0.32\textwidth]{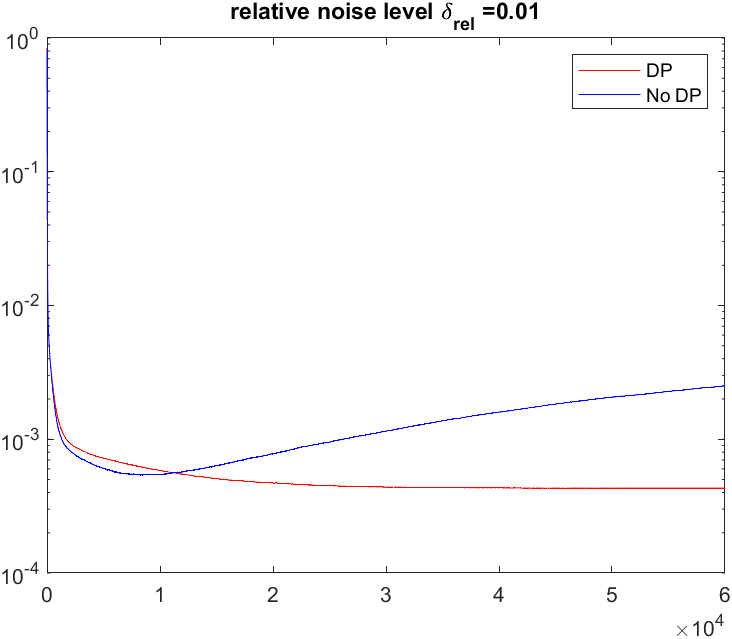}
\includegraphics[width = 0.32\textwidth]{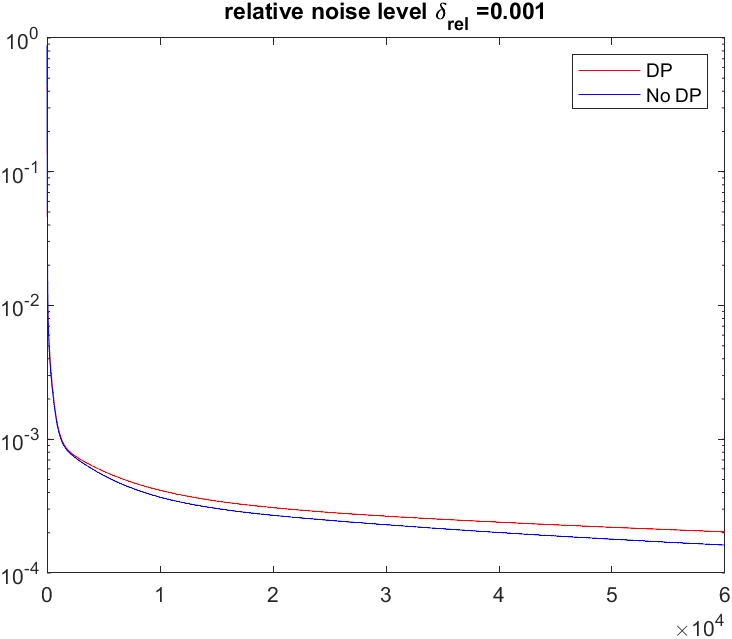}
\includegraphics[width = 0.32\textwidth]{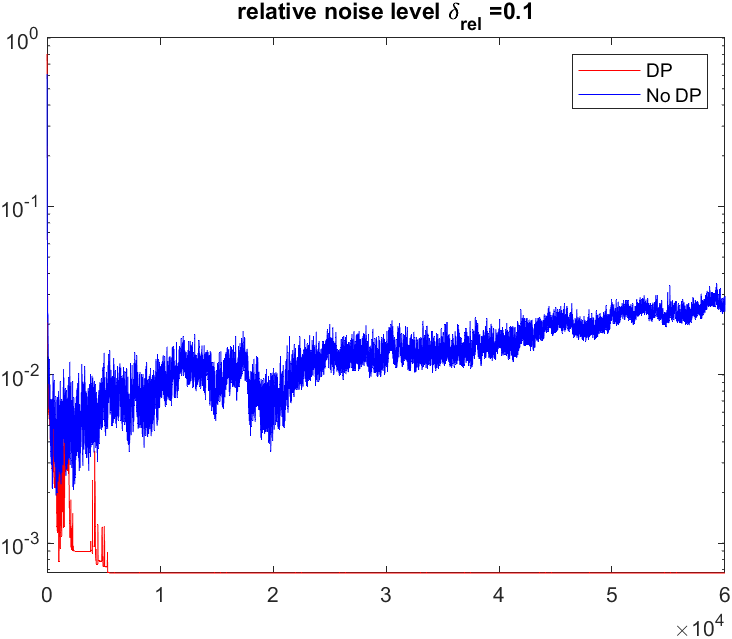}
\includegraphics[width = 0.32\textwidth]{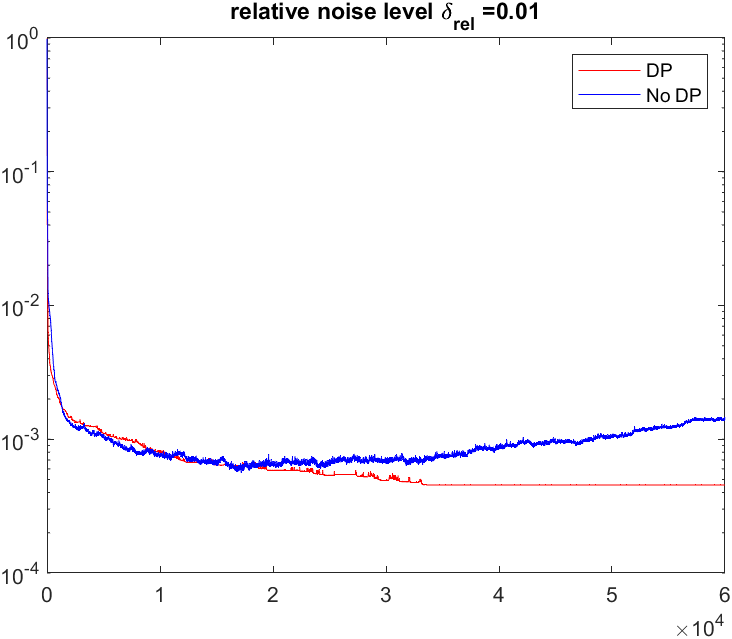}
\includegraphics[width = 0.32\textwidth]{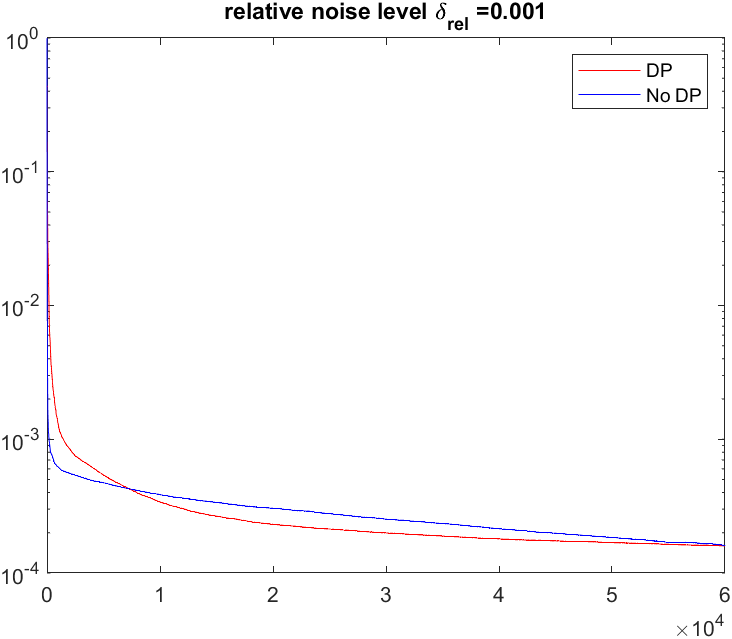}
%
\caption{Comparison of reconstruction results by the stochastic gradient descent (\ref{SGD.1}) using step-sizes (\ref{SGD.s1}) and (\ref{SGD.s3}). }\label{fig:SGD2}
\end{figure}
\end{example}

\begin{example}\label{SMD.ex2}
We consider the  linear system (\ref{smd.1}) in which $X$ and $Y_i$ are all Hilbert spaces and the sought solution satisfies the constraint $x\in \C$, where $\C\subset X$ is a closed convex set. Finding the unique solution $x^\dag$ of (\ref{smd.1}) in $\C$ with minimal norm can be stated as (\ref{rbdgm.1}) with $\R(x) := \frac{1}{2} \|x\|^2 + \iota_\C(x)$, where $\iota_\C$ denotes the indicator function of $\C$. Let $P_\C$ denote the metric projection of $X$ onto $\C$. Then the mini-batch stochastic mirror descent method for determining $x^\dag$ takes the form
\begin{align}\label{dgm.42}
x_n^\d = P_\C (\xi_n^\d), \qquad \xi_{n+1}^\d = \xi_n^\d - t_n^\d A_{I_n}^*(A_{I_n} x_n^\d - y_{I_n}^\d).
\end{align}
In case $X = L^2(\D)$ for some domain $\D \subset {\mathbb R}^d$ and $\C= \{x\in X: x \ge 0 \mbox{ a.e. on } \D\}$, the iteration scheme (\ref{dgm.42}) becomes
\begin{align}\label{SMD.NN}
x_n^\d = \max\{\xi_n^\d, 0\}, \qquad
\xi_{n+1}^\d  = \xi_n^\d - t_n^\d A_{I_n}^* (A_{I_n} x_n^\d- y_{In}^\d)
\end{align}
with an initial guess $\xi_0^\d =0$, where $I_n$ is a randomly selected subset of $\{1, \cdots, p\}$ via the uniform distribution with $|I_n|=b$ for a preassigned batch size $b$. 

As a testing example, we consider the computed tomography which consists in determining the density of cross sections of a human body by measuring the attenuation of X-rays as they propagate through the biological tissues \cite{N2001}. Mathematically, it requires to determine a function supported on a bounded domain from its line integrals. 
We consider here only the standard 2D parallel-beam tomography; tomography with other scan geometries can be considered similarly.  We consider a full angle problem using $90$ projection angles evenly distributed between $1$ and $180$ degrees, with $367$ lines per projection. Assuming the sought image is discretized on a $256 \times 256$ pixel grid, we may use the function \texttt{paralleltomo} in the MATLAB package AIR TOOLS \cite{HS2012} to discretize the problem. After deleting those rows with zero entries, it leads to a linear system $A x = y$, where $A$ is a matrix with size $29658\times 65536$. This gives a problem of the form (\ref{smd.1}) with $p = 29658$, where $A_i$ corresponds to the $i$-th row of $A$ and $y_i$ is the $i$-th component of $y$. 

\begin{figure}[htpb]
\centering
\includegraphics[width = 0.22\textwidth]{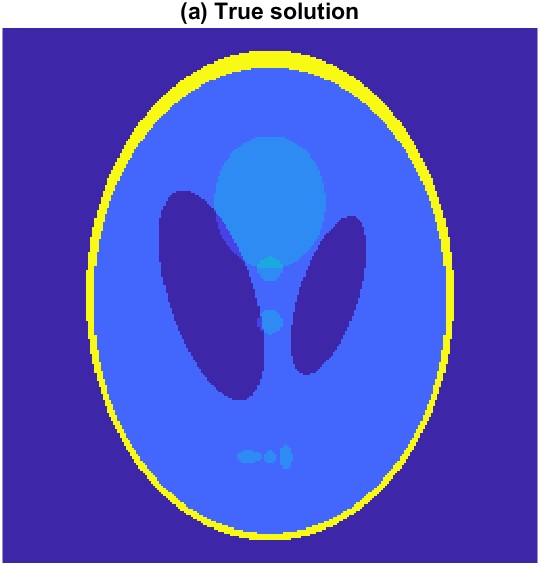}
\includegraphics[width = 0.22\textwidth]{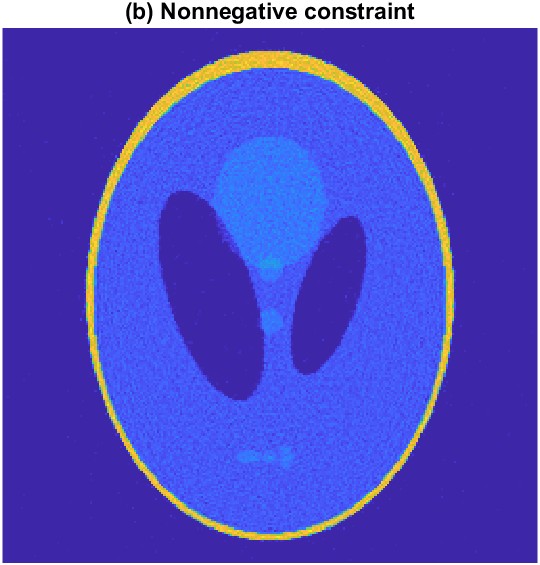}
\includegraphics[width = 0.22\textwidth]{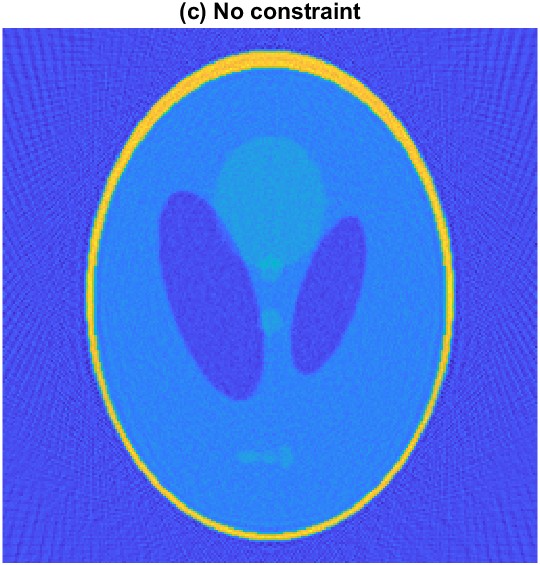}
\includegraphics[width = 0.3\textwidth]{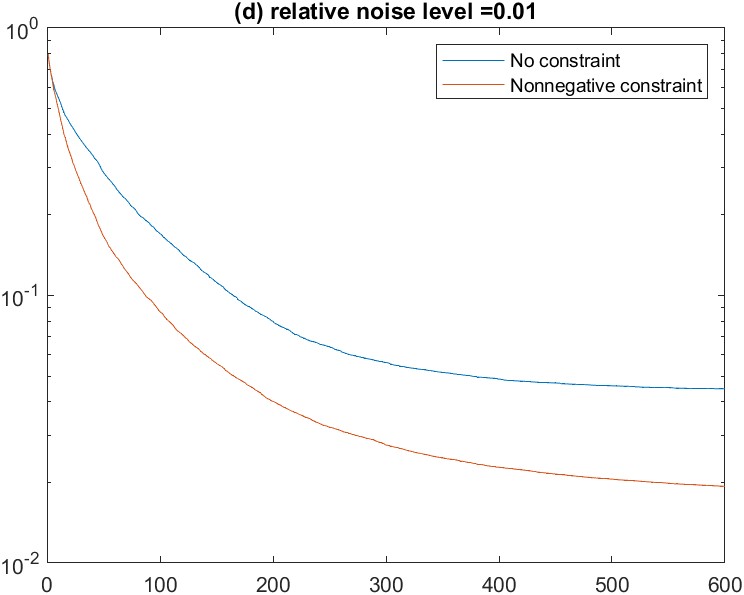}
\caption{Reconstruction of the Shepp-Logan phantom by the methods (\ref{SMD.NN}) and (\ref{SGD.1}) with batch size $b = 400$ and step-size chosen by (s2) with $\mu_0 =1$.}\label{fig.CT}
\end{figure}

We assume that the true image is the Shepp-Logan phantom shown in figure \ref{fig.CT} (a) discretized on a $256 \times 256$ pixel grid with nonnegative pixel values. This phantom is widely used in
evaluating tomographic reconstruction algorithms.
Let $x^\dag$ denote the vector formed by stacking all the columns of the true image and let $y = A x^\dag$ be the true data. By adding Gaussian noise to $y_i$ we get the noisy data $y_i^\d$ of the form (\ref{nd}), where $\ep_i \sim N(0, 1)$ and $\d_{rel} = 0.01$. Since the sought solution is nonnegative, we may reconstruct it by using the method (\ref{SMD.NN}). As a comparison, we also consider reconstructing the sought phantom by the method (\ref{SGD.1}) for which the nonnegative constraint is not incorporated. For the both methods, we use the batch size $b = 400$ and the step-size chosen by (s2) with $\mu_0 = 1$. We execute the methods for 600 iterations and plot the reconstruction results in Figure \ref{fig.CT} (b) and (c).  In Figure \ref{fig.CT} (d) we also plot the squares of the relative errors $\|x_n^\d - x^\dag\|^2/\|x^\dag\|^2$ which indicate that incorporating the nonnegative constraint can produce more accurate reconstruction results. This example demonstrates that available {\it a priori} information on sought solutions should be incorporated into algorithm design to assist with better reconstruction results. 
\end{example}

\begin{example} \label{SMD.ex3}
Let $\D \subset {\mathbb R}^d$ be a bounded domain. Consider the linear system (\ref{smd.1}), 
where $A_i: L^1(\D)\to Y_i$ is a bounded linear operator and $Y_i$ is a Hilbert space for each $i$. 
We assume the sought solution $x^\dag$ is a probability density function, i.e. 
$x^\dag \ge 0$ a.e. on $\D$ and $\int_\D x^\dag = 1$. To determine such a solution, we take 
$
\R(x) := f(x) + \iota_\Delta(x),
$
where $\iota_\Delta$ denotes the indicator function of 
$$
\Delta:= \left\{ x\in L_+^1(\D): \int_\D x^\dag = 1\right\}
$$
and 
$$
f(x) = \left\{\begin{array}{lll}
\int_\D x \log x, & \mbox{ if } x \in L_{+}^1(\D) \mbox{ and } x\log x \in L^1 (\D), \\
+ \infty, & \mbox{ otherwise}
\end{array}\right. 
$$
is the negative of the Boltzmann-Shannon entropy.
Here $L_{+}^1(\D):= \{x\in L^1(\D): x\ge 0 \mbox{ a.e. on } \D\}$. The Boltzmann-Shannon entropy has been used in Tikhonov regularization as a regularization functional to determine nonnegative solutions, see \cite{E1993,EL1993} for instance. According to \cite{BL1991,E1993}, $\R$ is strongly convex on $L^1(\D)$ with modulus of convexity $\sigma = 1/2$. By the Karush-Kuhn-Tucker theory, for any $\xi \in L^\infty(\D)$ the unique minimizer of 
$$
\min_{x\in L^1(\D)} \left\{ \R(x) - \int_\D \xi x \right\}
$$
is given by $\hat x = e^{\xi}/\int_\D e^\xi$. Therefore the corresponding stochastic mirror descent method takes the form 
\begin{align}\label{smd:entropy}
x_n^\d = \frac{1}{\int_\D e^{\xi_n^\d}} e^{\xi_n^\d}, \qquad 
\xi_{n+1}^\d = \xi_n^\d - t_n^\d A_{I_n}^* (A_{I_n} x_n^\d - y_{I_n}^\d), 
\end{align}
where $I_n \subset \{1, \cdots, p\}$ is a randomly selected subset via the uniform distribution with a preassigned batch size $b$ and $t_n^\d\ge 0$ is the step-size. 

For numerical simulations we consider the linear system (\ref{smd.testexample}) with $[a, b] = [0,1]$, $p = 1000$ and $k(s, t) := 4 e^{-(s-t)^2/0.0064}$. We assume the sought solution is 
$$
x^\dag(t):= c \left(e^{-60(t-0.3)^2} + 0.3 e^{-40(t-0.8)^2}\right),
$$
where $c>0$ is a constant to ensure that $\int_0^1 x^\dag (t) dt = 1$ so that $x^\dag$ is a probability density function. By adding random noise to the exact data $y_i:=A_i x^\dag$ we get the noisy data $y_i^\d$; we then use these noisy data and the method (\ref{smd:entropy}) with batch size $b = 1$ to reconstruct the sought solution $x^\dag$; the integrals involved in the method are approximated by the trapezoidal rule based on the partition of $[0,1]$ into $p-1$ subintervals of equal length. 

\begin{figure}[htpb]
\centering
\includegraphics[width = 0.32\textwidth]{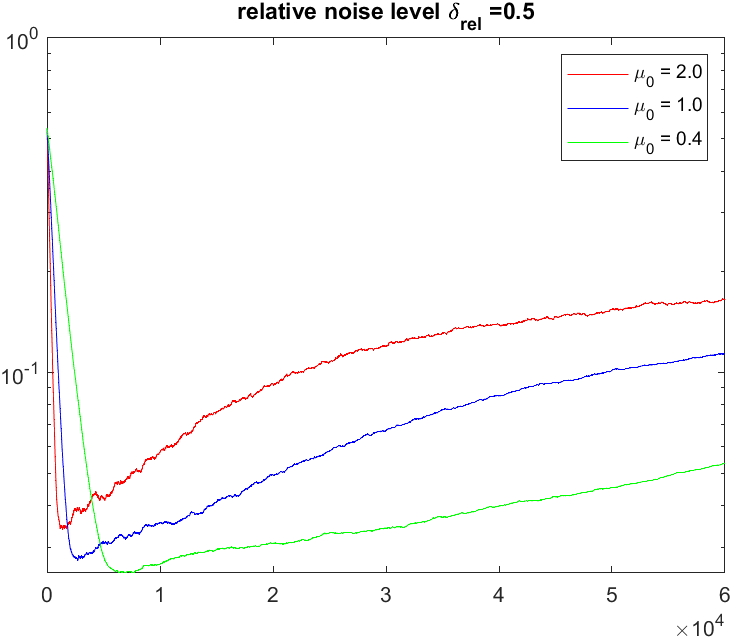}
\includegraphics[width = 0.32\textwidth]{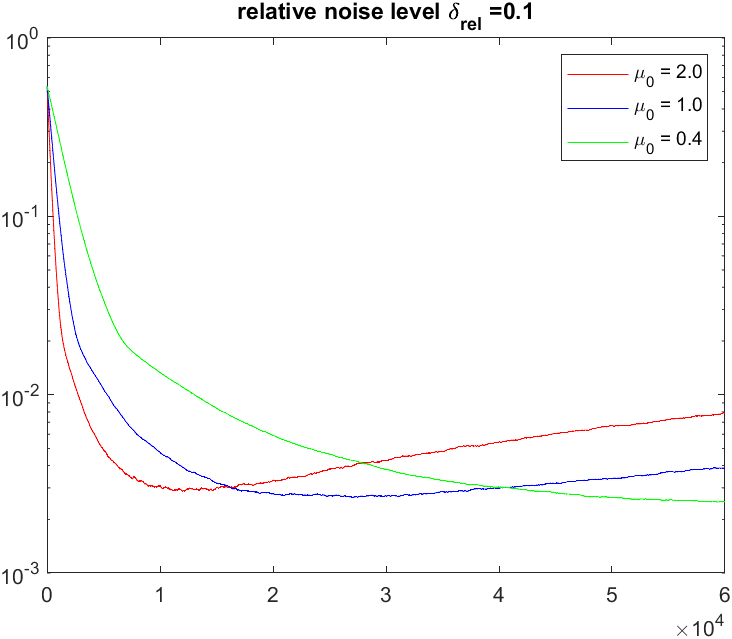}
\includegraphics[width = 0.32\textwidth]{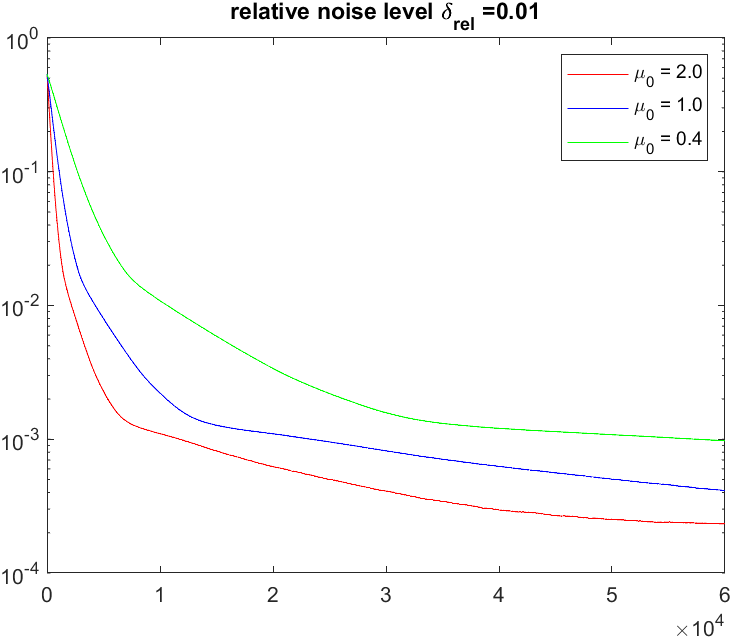}
\includegraphics[width = 0.32\textwidth]{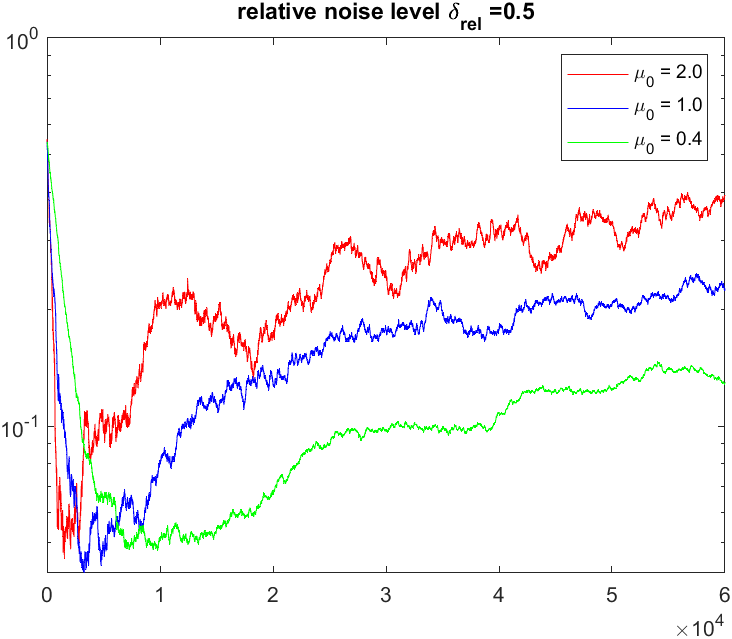}
\includegraphics[width = 0.32\textwidth]{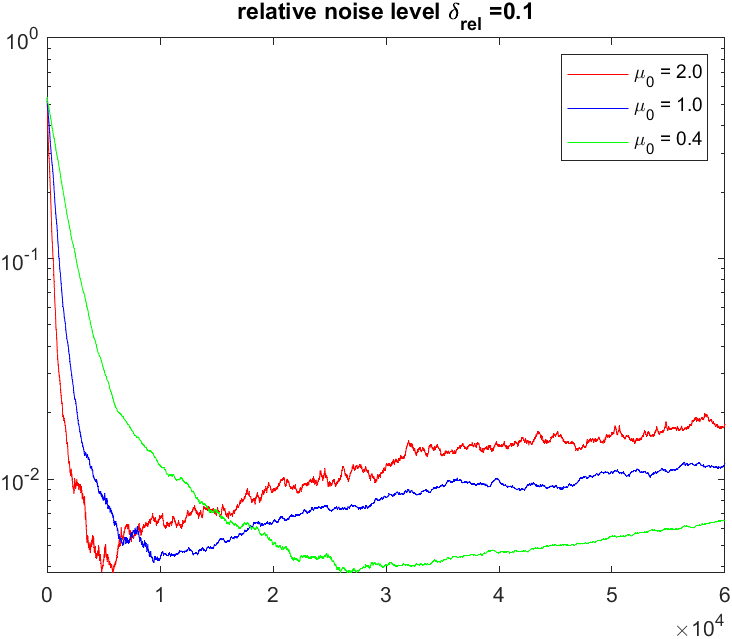}
\includegraphics[width = 0.32\textwidth]{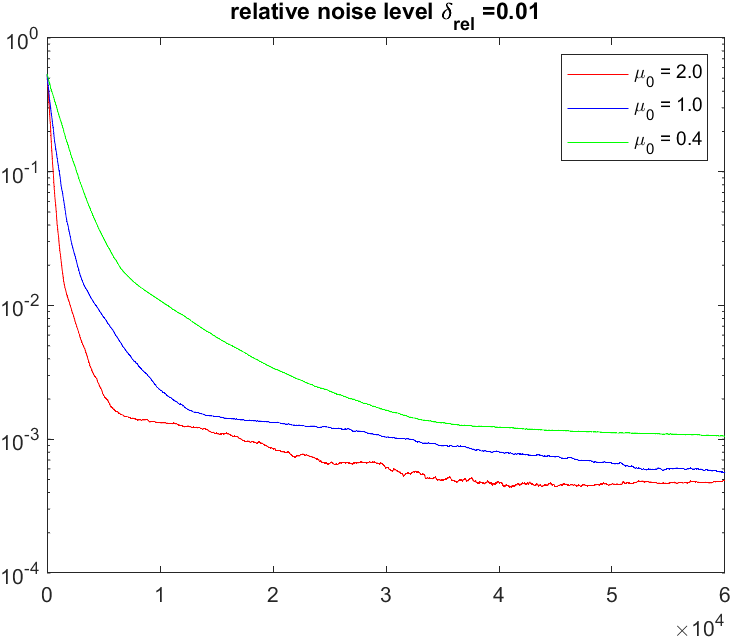}
%
\caption{Relative errors for reconstructing a solution, that is a probability density function, by the method (\ref{smd:entropy}) using  step-size (s2) with various values of $\mu_0$.  }\label{fig:entropy1}
\end{figure}

We first test the performance of the method (\ref{smd:entropy}) using the noisy data given by (\ref{nd}) corrupted by Gaussian noise, where $\ep_i \sim N(0, 1)$ and $\d_{rel}>0$ is the relative noise level. We execute the method (\ref{smd:entropy}) using the batch size $b = 1$, the initial guess $\xi_0^\d =0$ and the step-size $t_n^\d$ chosen by (s2) with three distinct values $\mu_0 = 2.0, 1.0$ and $0.4$. In Figure \ref{fig:entropy1} we plot the reconstruction errors for three distinct relative noise levels $\d_{rel} = 0.5, 0.1$ and $0.01$. The first row plots the mean square errors $\EE[\|x_n^\d - x^\dag\|_{L^1}^2]$ which are calculated approximately by the average of 100 independent runs. The second row plots $\|x_n^\d - x^\dag\|_{L^1}^2$ for a particular individual run.  These plots demonstrate that the method (\ref{smd:entropy}) admits the semi-convergence property and the iterates exhibit dramatic oscillations. Furthermore, using large step-sizes can make convergence fast at the beginning but the iterates can diverge quickly; while using small step-sizes can suppress the oscillations and delay the occurrence of semi-convergence, it however makes the method converge slowly. 

\begin{figure}[htpb]
\centering
\includegraphics[width = 0.32\textwidth]{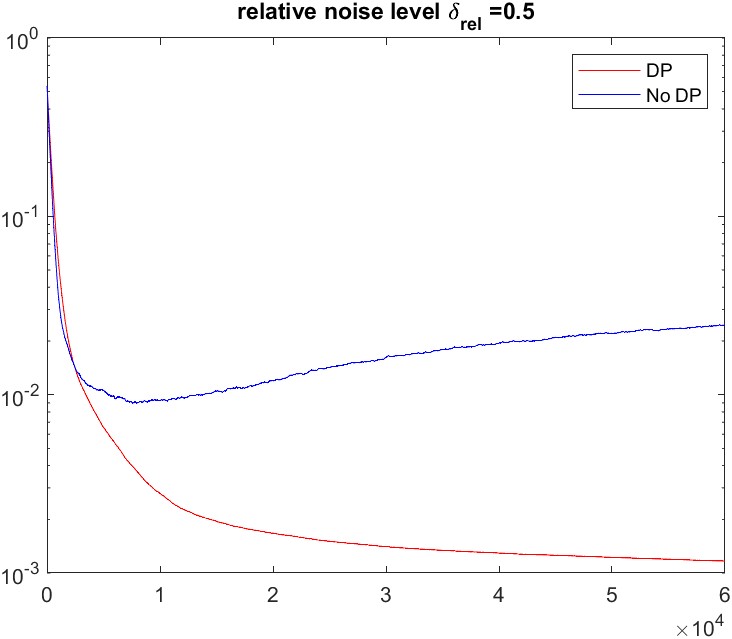}
\includegraphics[width = 0.32\textwidth]{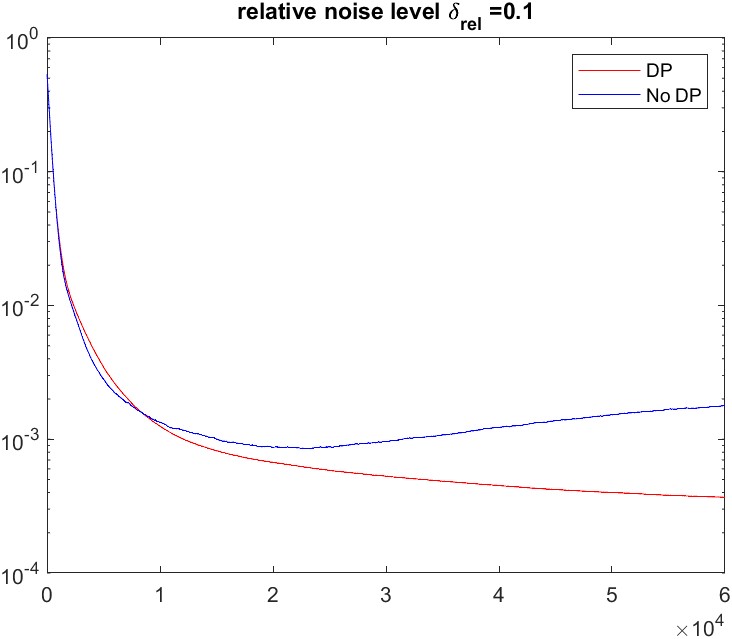}
\includegraphics[width = 0.32\textwidth]{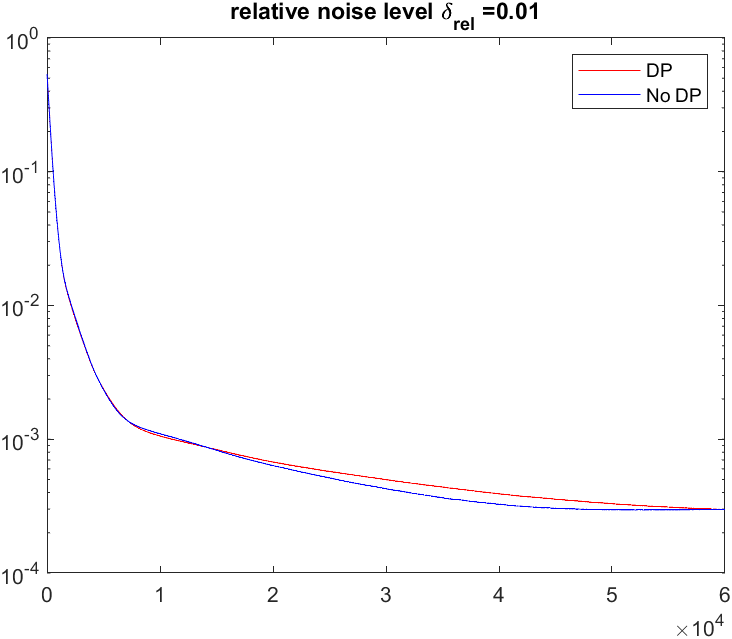}
\includegraphics[width = 0.32\textwidth]{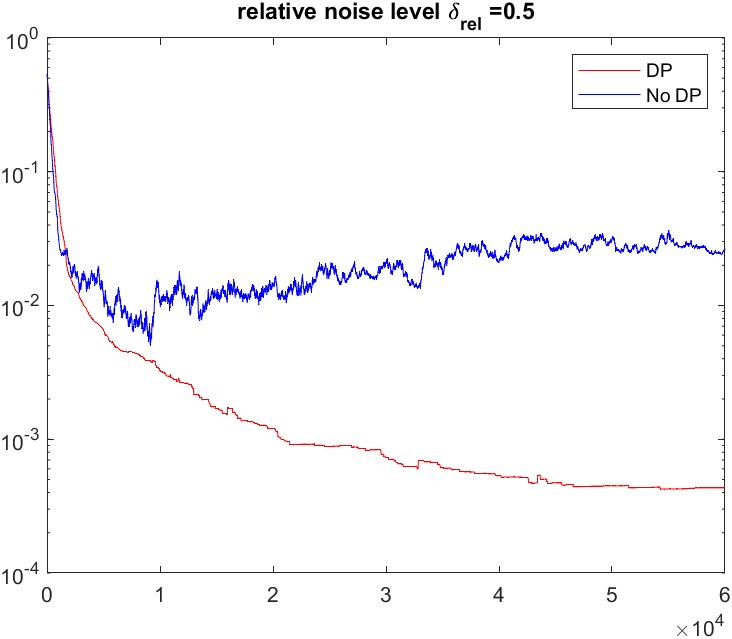}
\includegraphics[width = 0.32\textwidth]{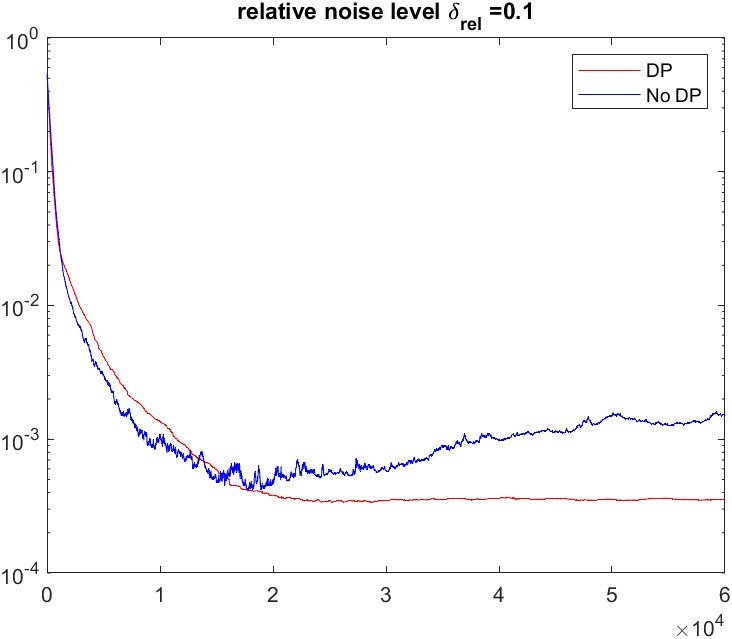}
\includegraphics[width = 0.32\textwidth]{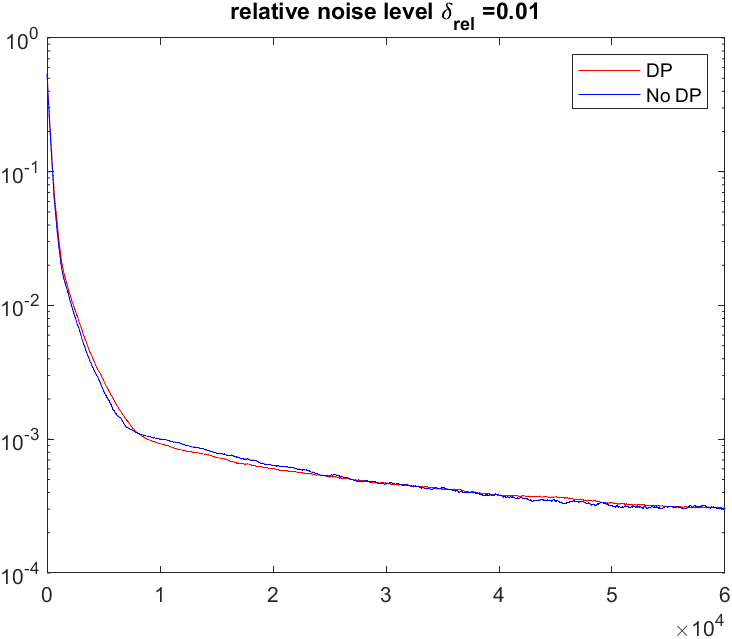}
%
\caption{Comparison of reconstruction results for a solution that is a probability density function by the method (\ref{smd:entropy}) using step-sizes chosen by (s2) and (s3). }\label{fig:entropy2}
\end{figure}

In order to efficiently remove oscillations and semi-convergence, we next consider the method (\ref{smd:entropy}) using the step-size chosen by (s3) whose realization relies on the information of noise level. We consider noisy data of the form (\ref{nd}), where $\ep_i$ are uniform noise distributed on $[-1, 1]$. Note that $\d_i:=\d_{rel} |y_i|$ are the noise levels such that $|y_i^\d - y_i|\le \d_i$ for all $i$; we assume the information on $\d_i$ is known. To illustrate the advantage of the step-size chosen by (s3), we compare the computed results with the ones obtained by the step-size chosen by (s2). In Figure \ref{fig:entropy2} we plot the reconstruction errors by the method (\ref{smd:entropy}) for various relative noise levels using the step-sizes chosen by (s2) and (s3) with $\mu_0 = 2$ and $\tau = 1$, where ``\texttt{DP}" and ``\texttt{No DP}" represent the results corresponding to the step-size chosen by (s3) and (s2) respectively. The first row in Figure \ref{fig:entropy2} plots the mean square errors calculated approximately by the average of 100 independent runs and the second row plots the reconstruction errors for a particular individual run. The results show clearly that using step-size by (s3) can significantly suppress the oscillations in iterates and relieve the method from semi-convergence.
\end{example}

\begin{example}\label{SMD.ex4} 
Let $\D \subset {\mathbb R}^d$ be a bounded domain. We consider reconstructing a sparse solution in the linear system (\ref{smd.1}), 
where $A_i: L^2(\D) \to Y_i$ is a bounded linear operator and $Y_i$ is a Hilbert space for each $i$. To determine such a solution, we take $\R(x) := \beta \|x\|_{L^1(\D)} + \frac{1}{2} \|x\|_{L^2(\D)}^2$, where $\beta>0$ is a sufficiently large number whose choice reflects the role of the term $\|x\|_{L^1(\D)}$ in the reconstruction of sparse solutions. Note that for any $\xi \in L^2(\D)$ we have 
$$
\arg\min_{x\in L^2(\D)} \left\{\R(x) - \l \xi, x\r \right\} = \mbox{sign}(\xi)\max\{|\xi|-\beta, 0\}.
$$
Therefore the corresponding stochastic mirror descent method takes the form 
\begin{align}\label{smd:sparse}
x_n^\d = \mbox{sign}(\xi_n^\d) \max\{|\xi_n^\d| -\beta, 0\}, \quad 
\xi_{n+1}^\d = \xi_n^\d - t_n^\d A_{I_n}^* (A_{I_n} x_n^\d - y_{I_n}^\d), 
\end{align}
where $I_n \subset \{1, \cdots, p\}$ is a randomly selected subset satisfying $|I_n|=b$ via the uniform distribution with the preassigned batch size $b$ and $t_n^\d\ge 0$ is the step-size. We remark that, when $|I_n|=1$ and $I_n = n \ (\mbox{mod } p) +1$, the corresponding method is the sparse-Kaczmarz method which was first proposed in \cite{JW2013} to reconstruct sparse solutions of ill-posed inverse problems. For finite-dimensional linear systems with well-conditioned matrices, a randomized sparse-Kaczmarz method of the form (\ref{smd:sparse}) was introduced in \cite{SL2019}
in which $I_n$ is chosen with a probability proportional to $\|A_{I_n}\|^2$ and a linear convergence rate is derived. The theoretical result in \cite{SL2019} however is not applicable to ill-posed problems because the analysis depends heavily on the finite-dimensionality of the underlying spaces and the well-conditioning property of the matrices.

For numerical simulations we consider the linear system (\ref{smd.testexample}) with $[a,b] = [0,1]$ and $k(s, t) := (0.1^2+ (s-t)^2)^{-3/2}$. We assume the sought solution is 
$$
x^\dag(t):= \chi_{\{0.19\le t\le 0.22\}} - \chi_{\{0.50\le t\le 0.52\}} + 0.5 \chi_{\{0.78\le t\le 0.80\}} 
$$
which is sparse on $[0, 1]$. By adding random noise to the exact data $y_i:=A_i x^\dag$ we get the noisy data $y_i^\d$; we then use these noisy data and the method (\ref{smd:sparse}) with batch size $b = 1$ to reconstruct the sought solution $x^\dag$. In our numerical computation we take $p = 1000$ and $\beta = 80$ and the integrals involved in the method are approximated by the trapezoidal rule based on the partition of $[0,1]$ into $p-1$ subintervals of equal length. 

\begin{figure}[htpb]
\centering
\includegraphics[width = 0.32\textwidth]{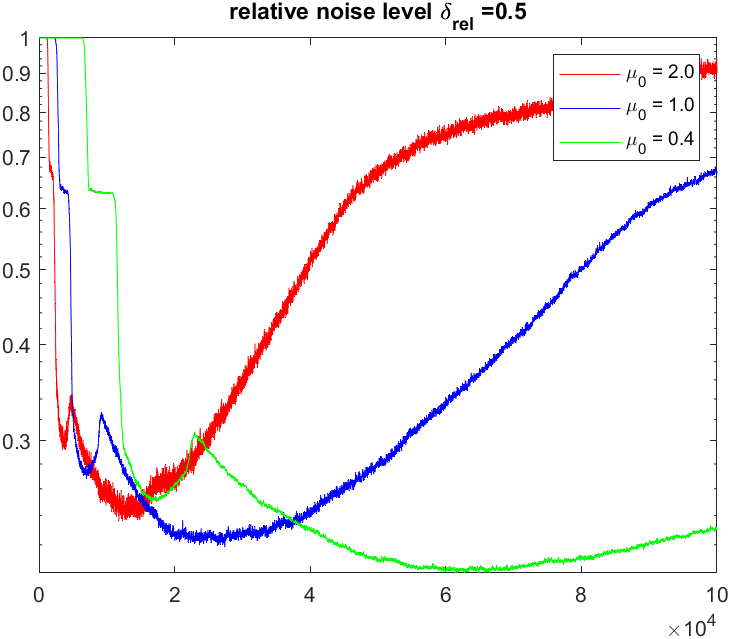}
\includegraphics[width = 0.32\textwidth]{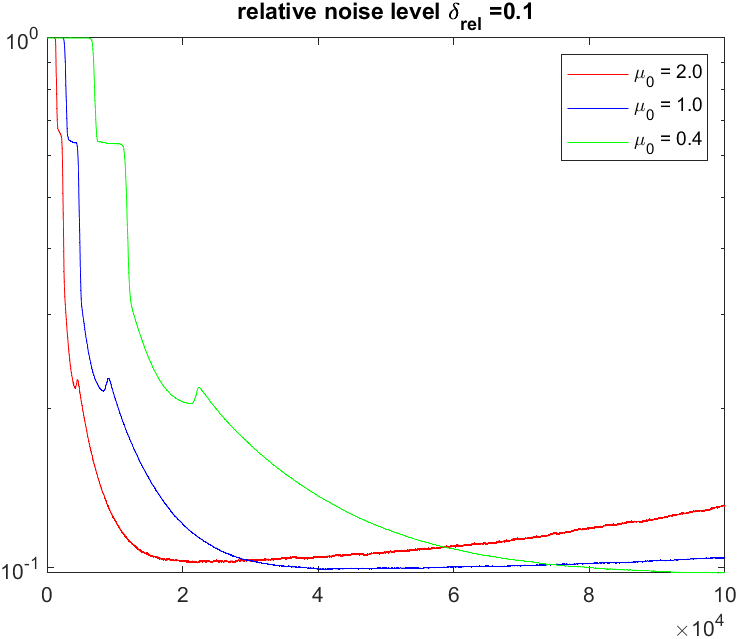}
\includegraphics[width = 0.32\textwidth]{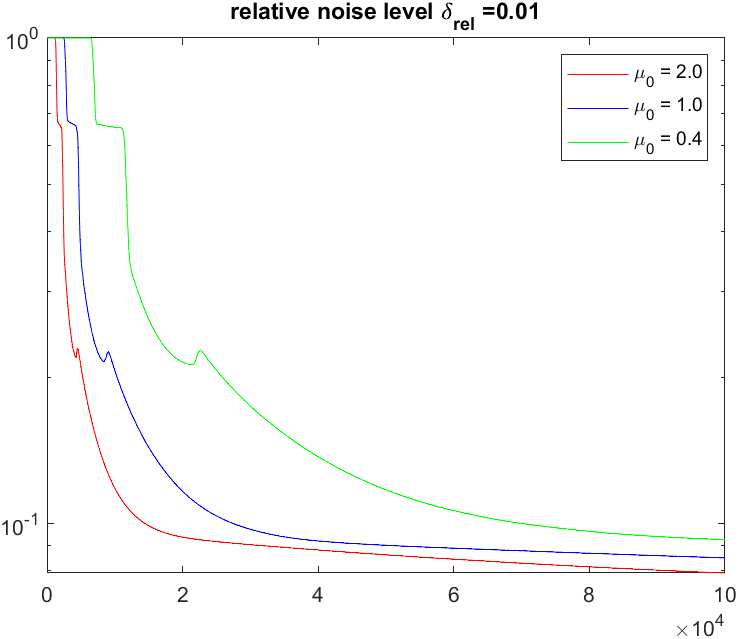}
\includegraphics[width = 0.32\textwidth]{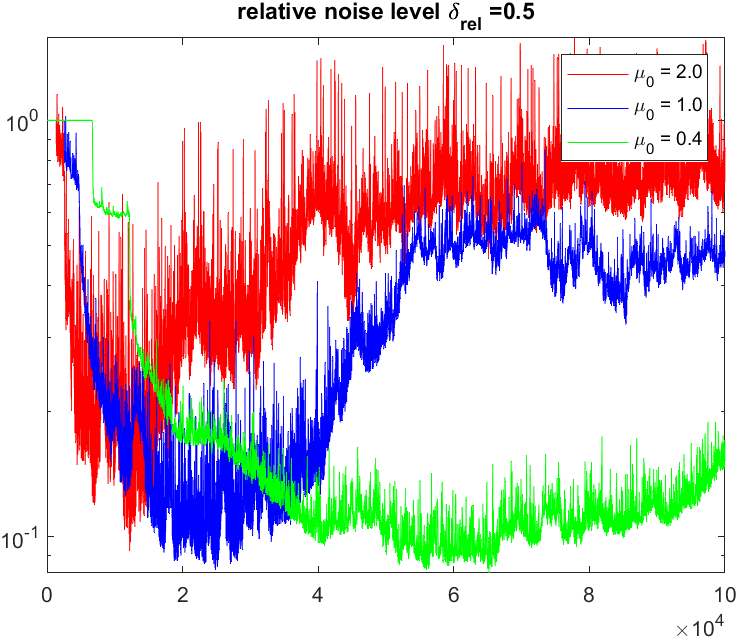}
\includegraphics[width = 0.32\textwidth]{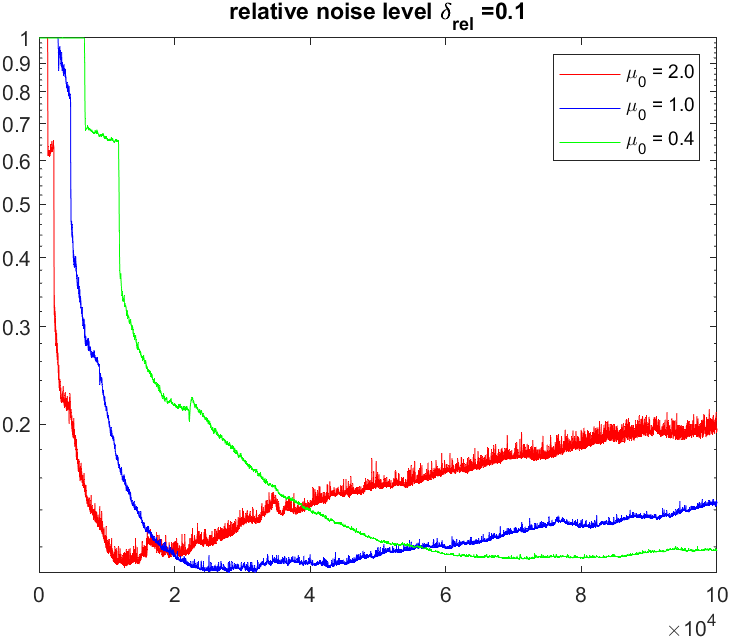}
\includegraphics[width = 0.32\textwidth]{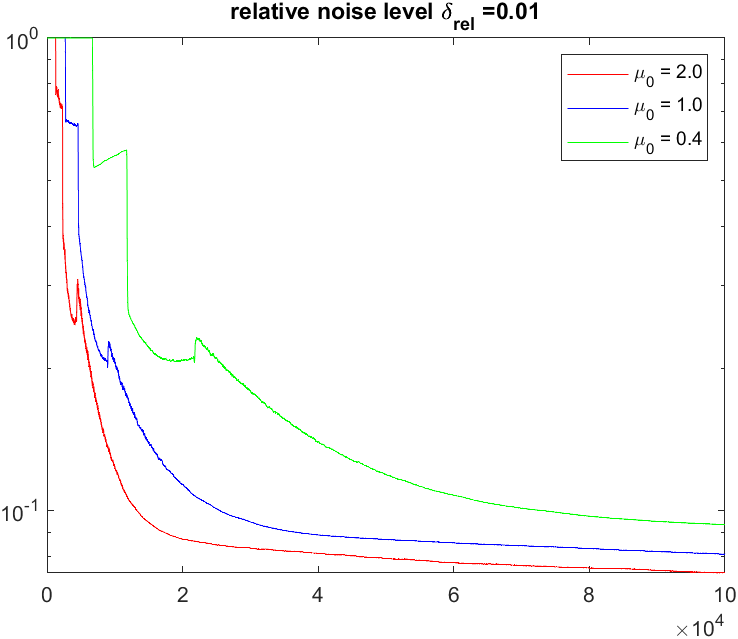}
%
\caption{Relative errors for reconstructing a sparse solution by the method (\ref{smd:sparse}) using  step-size (s2) with various values of $\mu_0$. }\label{fig:sparse1}
\end{figure}

We first test the performance of the method (\ref{smd:sparse}) using the noisy data given by (\ref{nd}) corrupted by Gaussian noise, where $\ep_i \sim N(0, 1)$. We execute the method (\ref{smd:sparse}) using the batch size $b = 1$, the initial guess $\xi_0^\d =0$ and the step-size $t_n^\d$ chosen by (s2) with three distinct values $\mu_0 = 2.0, 1.0$ and $0.4$. In Figure \ref{fig:sparse1} we plot the reconstruction errors for three distinct relative noise levels $\d_{rel} = 0.5, 0.1$ and $0.01$. The first row plots the relative mean square errors $\EE[\|x_n^\d - x^\dag\|_{L^2}^2/\|x^\dag\|_{L^2}^2]$ calculated approximately by the average of 100 independent runs, and the second row plots the relative errors $\|x_n^\d - x^\dag\|_{L^2}^2/\|x^\dag\|_{L^2}^2$ for a particular individual run.  From these plots we can observe the oscillations in iterates, the semi-convergence of the method, and the influence of the magnitude of step-sizes. 

\begin{figure}[htpb]
\centering
\includegraphics[width = 0.32\textwidth]{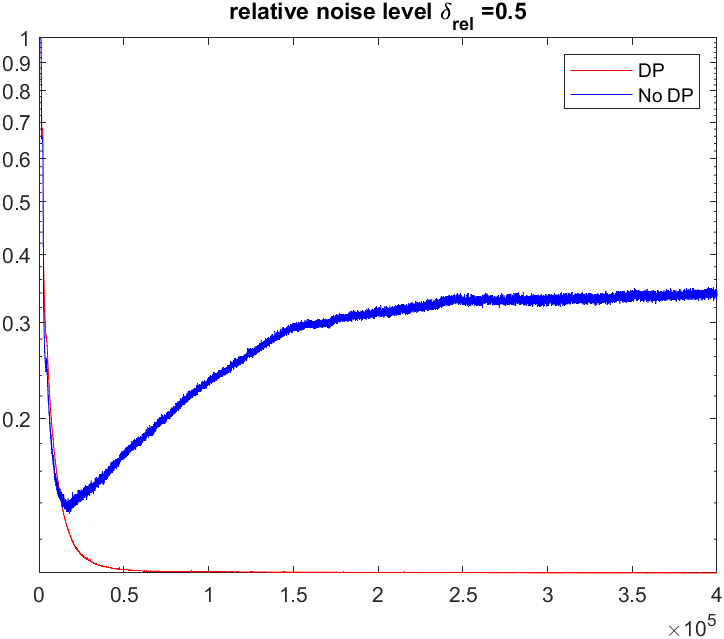}
\includegraphics[width = 0.32\textwidth]{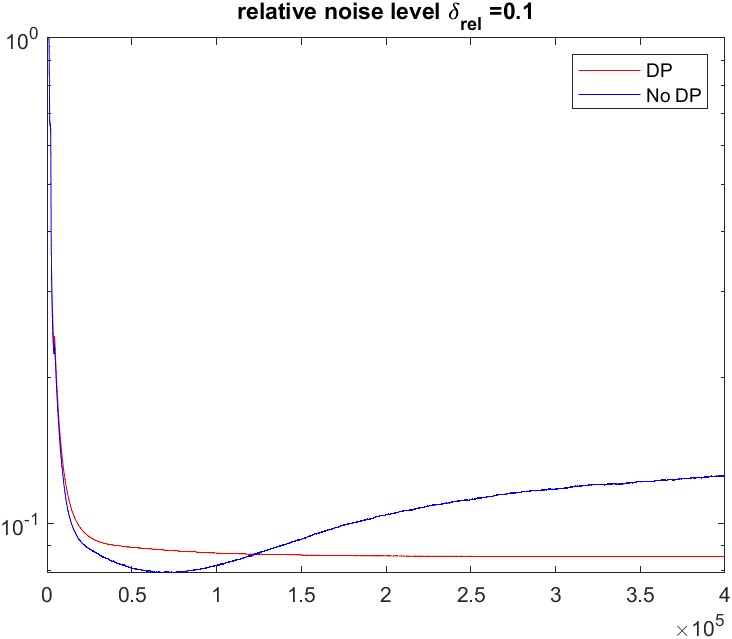}
\includegraphics[width = 0.32\textwidth]{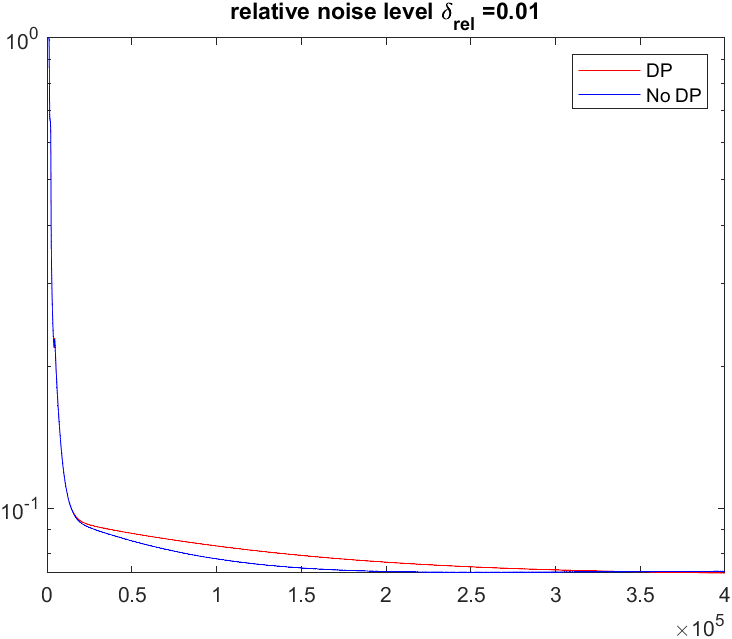}
\includegraphics[width = 0.32\textwidth]{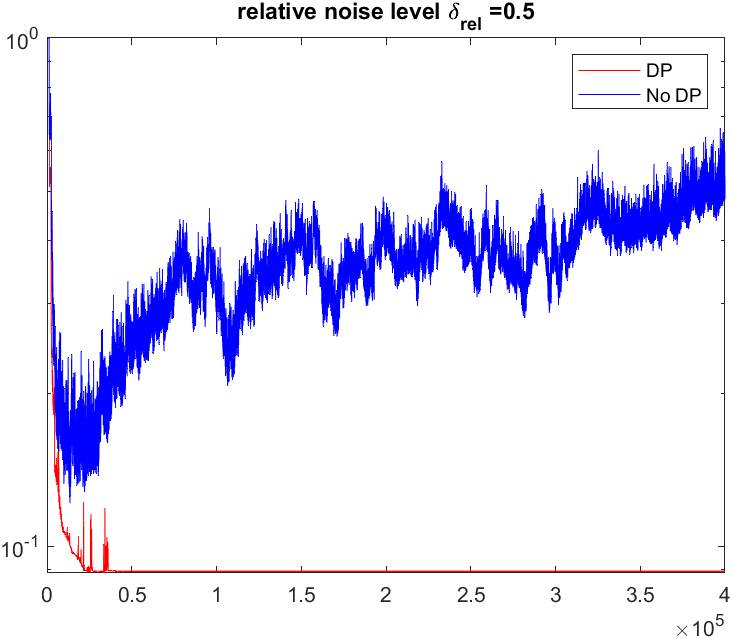}
\includegraphics[width = 0.32\textwidth]{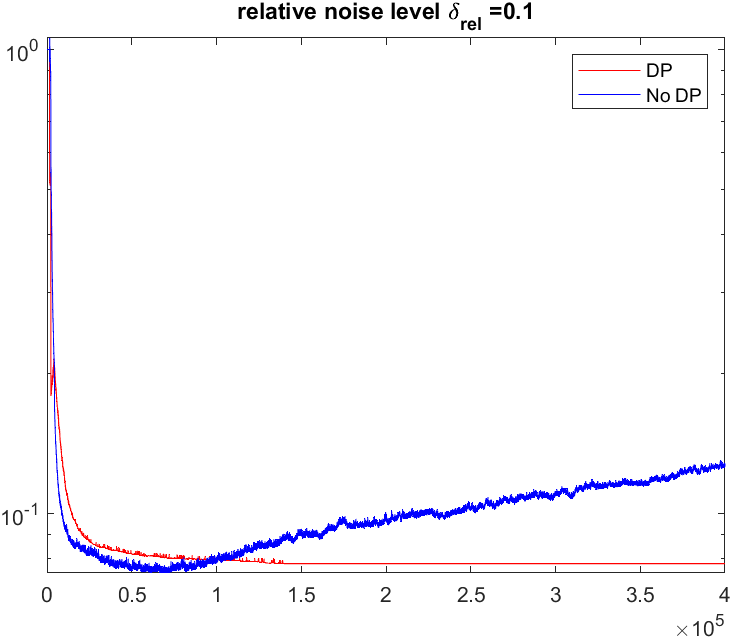}
\includegraphics[width = 0.32\textwidth]{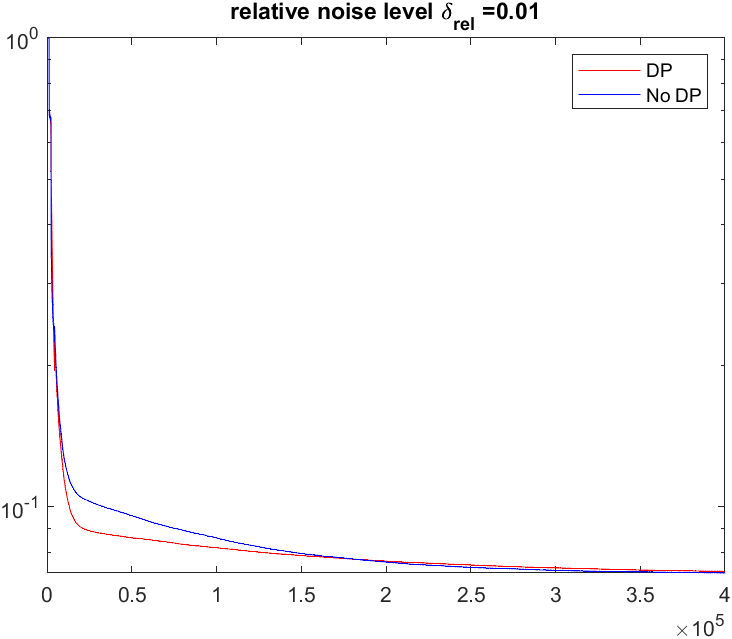}
%
\caption{Reconstruction of sparse solutions by the method (\ref{smd:sparse}) using step-sizes chosen by (s2) and (s3). The first row plots the relative mean square errors calculated approximately by the average of 100 independent runs and the second row plots the relative errors for a particular individual run. }\label{fig:sparse2}
\end{figure}

We next consider the method (\ref{smd:sparse}) using the step-size chosen by (s3). We use noisy data of the form (\ref{nd}) with $\ep_i$ being uniformly distributed over $[-1, 1]$ and assume that the noise levels $\d_i = \d_{rel} |y_i|$ are known. In Figure \ref{fig:sparse2} we plot the reconstruction errors of the method (\ref{smd:sparse}) using the step-sizes chosen by (s3) with $\mu_0 = 2$ and $\tau = 1.01$ which are labelled as ``\texttt{DP}"; as comparisons we also plot the corresponding results using step-sizes chosen by (s2) with $\mu_0 = 2$ which are labelled as ``\texttt{No DP}".  Sharply contrast to (s2), these results illustrate that using step-size by (s3) can significantly suppress the oscillations in iterates and reduce the effect of semi-convergence.
\end{example}

\begin{example} \label{SMD.ex5}
In this final example we consider using the stochastic mirror descent method to reconstruct piecewise constant solutions. We consider again the linear system (\ref{smd.testexample}) with $[a, b] = [0,1]$ and $k(s,t) = 4e^{-(s-t)^2/0.0064}$ and assume that the sought solution $x^\dag$ is piecewise constant. By dividing $[0,1]$ into $p-1$ subintervals of equal length and approximating integrals by the trapezoidal rule, we have a discrete ill-posed system $\bA_i \bx = y_i$, where $\bA_i$ is a row vector for each $i$. To reconstruct a piecewise constant solution, we use the model
\begin{align}\label{TV.1}
\min\left\{\R(\bx):=\beta \|\bD \bx\|_1 + \frac{1}{2} \|\bD \bx\|_2^2 + \frac{1}{2} \|\bx\|_2^2: \bA_i \bx = y_i, i =1, \cdots, p \right\},
\end{align}
where $\bD$ denotes the discrete gradient operator and $\beta$ is a large positive number. If we apply the stochastic mirror descent method to solve (\ref{TV.1}) directly, we need to solve a minimization problem related to $\R$ to obtain $\bx_n^\d$ at each iteration step. This can make the algorithm time-consuming since those minimization problems can not be solved explicitly. To circumvent this difficulty, we introduce $\bz := \bD \bx$ and rephrase (\ref{TV.1}) as 
\begin{align}\label{TV.2}
\min\left\{\beta \|\bz\|_1 + \frac{1}{2} \|\bz\|_2^2 + \frac{1}{2} \|\bx\|_2^2: \bB_i \begin{pmatrix} \bx \\ \bz \end{pmatrix} = \begin{pmatrix} y_i \\ {\bf 0} \end{pmatrix}, i =1, \cdots, p \right\},
\end{align}
where $\bB_i:=\begin{pmatrix} \bA_i & {\bf 0} \\ \bD & -{\bf I} \end{pmatrix}$. Assuming, instead of $y_i$, we have the noisy data $y_i^\d$. By applying 
the stochastic mirror descent method with batch size $b=1$ to (\ref{TV.2}) we can obtain 
\begin{align*}
& \begin{pmatrix}
\bx_n^\d \\[1ex] \bz_n^\d 
\end{pmatrix} = \arg\min_{\bx, \bz} \left\{\beta\|\bz\|_1 + \frac{1}{2}\|\bz\|_2^2 + \frac{1}{2} \|x\|_2^2 - \l \xi_n^\d, \bx\r -\l \eta_n^\d, \bz\r \right\}, \\
& \begin{pmatrix}
\xi_{n+1}^\d \\[1ex] \eta_{n+1}^\d 
\end{pmatrix} = \begin{pmatrix}
\xi_{n}^\d \\[1ex] \eta_{n}^\d 
\end{pmatrix} - t_n^\d \bB_{i_n}^T\left(\bB_{i_n}\begin{pmatrix} \bx_n^\d \\[1ex] \bz_n^\d \end{pmatrix} -\begin{pmatrix} y_{i_n}^\d \\[1ex] {\bf 0} \end{pmatrix}\right).
\end{align*}
By calculating $\bx_n^\d, \bz_n^\d$ and noting $\bx_n^\d = \xi_n^\d$, we therefore obtain the following iteration scheme
\begin{align}\label{smd:TV}
\begin{split}
\bz_n^\d & = \mbox{sign}(\eta_n^\d) \max\{|\eta_n^\d|-\beta, 0\}, \\
\bx_{n+1}^\d & = \bx_n^\d - t_n^\d \left(\bA_{i_n}^T(\bA_{i_n} \bx_n^\d - y_{i_n}^\d) + \bD^T (\bD \bx_n^\d - \bz_n^\d)\right), \\
\eta_{n+1}^\d & = \eta_n^\d - t_n^\d (\bz_n^\d - \bD \bx_n^\d) 
\end{split}
\end{align}
with the initial guess $\eta_0^\d={\bf 0}$ and $\bx_0^\d ={\bf 0}$, where $i_n\in \{1, \cdots, p\}$ is a randomly selected index via the uniform distribution and $t_n^\d\ge 0$ denotes the step-size.

\begin{figure}[htpb]
\centering
\includegraphics[width = 0.32\textwidth]{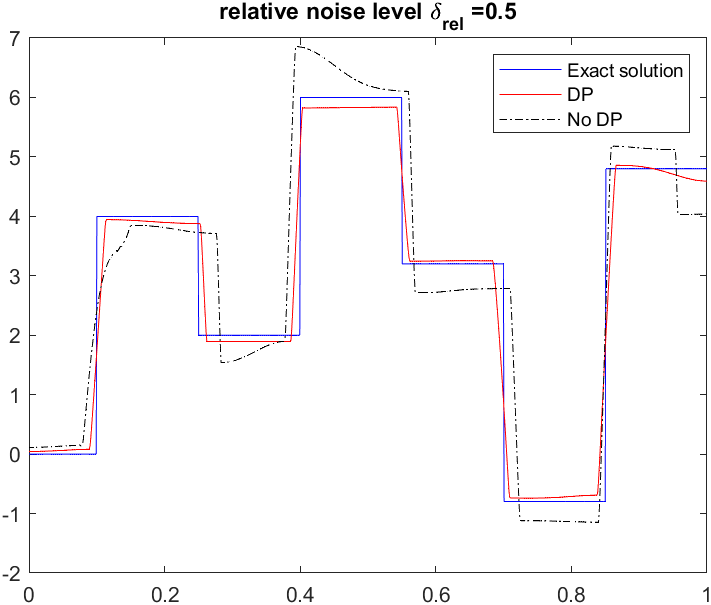}
\includegraphics[width = 0.32\textwidth]{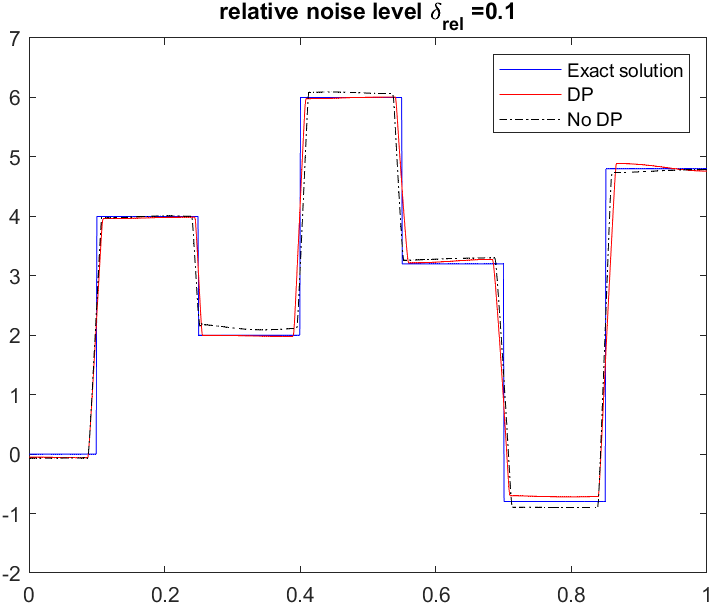}
\includegraphics[width = 0.32\textwidth]{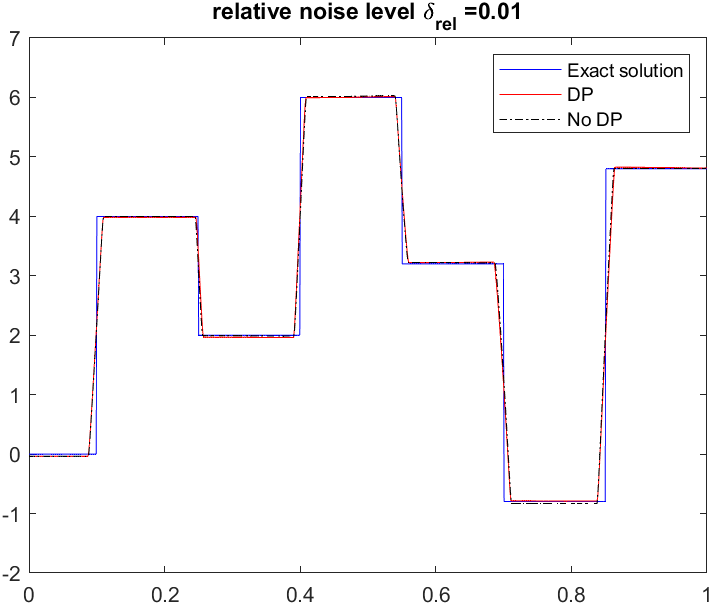}
%
\caption{Reconstruction results of a piecewise constant solution by (\ref{smd:TV}) using  step-sizes chosen by (\ref{TV.s2}) and (\ref{TV.s3}) with $\mu_0=1$ and $\tau =1$ after $5\times 10^5$ iterations. }\label{fig:TV1}
\end{figure}

Assuming the sought solution $\bx^\dag$ is piecewise constant whose graph is plotted in Figure \ref{fig:TV1}. By adding noise to the exact data $y_i:= \bA_i \bx^\dag$, we produce the noisy data $y_i^\d$. We then use these noisy data in the method (\ref{smd:TV}) to reconstruct $\bx^\dag$. In the numerical computation we use $p=1000$ and $\beta = 400$. For the method (\ref{smd:TV}) using step-size chosen by (s2), i.e. 
\begin{align}\label{TV.s2}
t_n^\d = \frac{\mu_0(|\bA_{i_n} \bx_n^\d - y_{i_n}^\d|^2 + \| \bD \bx_n^\d - \bz_n^\d\|_2^2)}{\|\bA_{i_n}^T(\bA_{i_n} \bx_n^\d - y_{i_n}^\d)\|_2^2 + \|\bD^T (\bD \bx_n^\d - \bz_n^\d)\|_2^2}
\end{align}
we have performed the numerical computation for several distinct values of $\mu_0$ and observed semi-convergence property, oscillations of iterates and the effect of the magnitude of the step-sizes. Since these observations are very similar to the previous examples, we will not report them here. 
Instead we will focus on the computational effect of the step-size chosen by (s3). As before, the noisy data are generated by (\ref{nd}) with $\ep_i$ being uniform noise on $[-1,1]$ and we assume the noise levels $\d_i:= \d_{rel} |y_i|$ are known. The step-size chosen by (s3) then takes the formula
\begin{align}\label{TV.s3}
t_n^\d = \left\{\begin{array}{lll}
\frac{\mu_0(|\bA_{i_n} \bx_n^\d - y_{i_n}^\d|^2 + \| \bD \bx_n^\d - \bz_n^\d\|_2^2)}{\|\bA_{i_n}^T(\bA_{i_n} \bx_n^\d - y_{i_n}^\d)\|_2^2 + \|\bD^T (\bD \bx_n^\d - \bz_n^\d)\|_2^2} & \mbox{ if } |\bA_{i_n} \bx_n^\d- y_{i_n}^\d|> \tau \d_{i_n}, \\
0 & \mbox{ otherwise}. 
\end{array}\right.
\end{align}
In Figure \ref{fig:TV1} we plot the reconstruction results by the method (\ref{smd:TV}) after $5\times 10^5$ iterations using noise data for three distinct relative noise levels $\d_{rel} = 0.5, 0.1$ and $0.01$, where ``\texttt{DP}" and ``\texttt{No DP}" represent the reconstruction results by the step-size chosen by (\ref{TV.s3}) and (\ref{TV.s2}) respectively; we use $\mu_0 = 1.0$ and $\tau = 1.0$.
These results demonstrate that the method (\ref{smd:TV}) can capture the piecewise constant feature very well. In order to give further comparison on the reconstruction results by the step-sizes (\ref{TV.s2}) and (\ref{TV.s3}), we present in Figure \ref{fig:TV2} the reconstruction errors, the first row plots approximations of the relative mean square errors $\EE[\|\bx_n^\d-\bx^\dag\|_2^2/\|\bx^\dag\|_2^2]$ by an average of 100 independent runs and the second row plots the reconstruction errors $\|\bx_n^\d-\bx^\dag\|_2^2/\|\bx^\dag\|_2^2$ for a typical individual run. The results show clearly that using the step-size (\ref{TV.s3}) can significantly suppress the oscillations in the iterates and reduce the effect of semi-convergence of the method.
\end{example}

\begin{figure}[htpb]
\centering
\includegraphics[width = 0.32\textwidth]{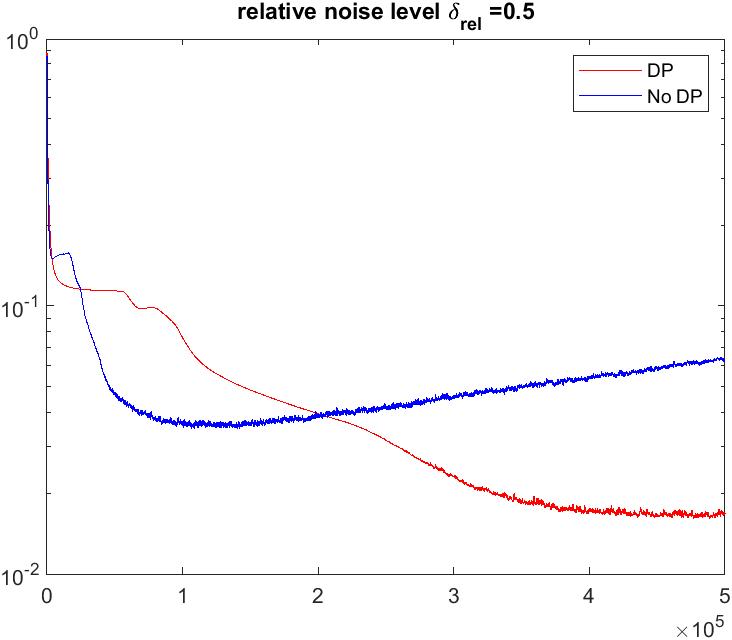}
\includegraphics[width = 0.32\textwidth]{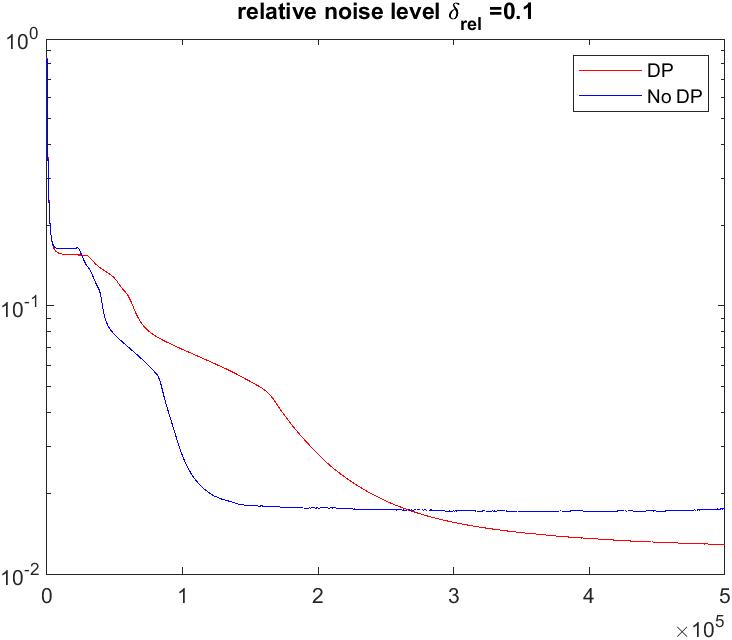}
\includegraphics[width = 0.32\textwidth]{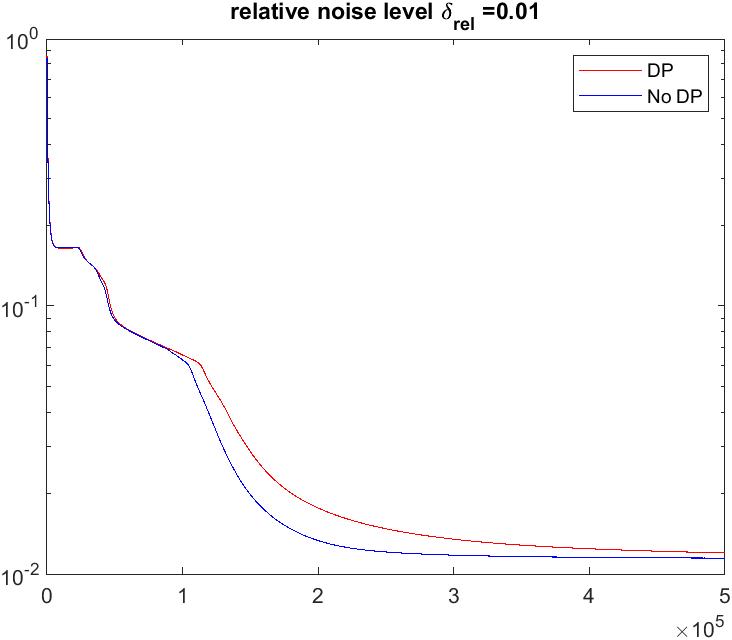}
\includegraphics[width = 0.32\textwidth]{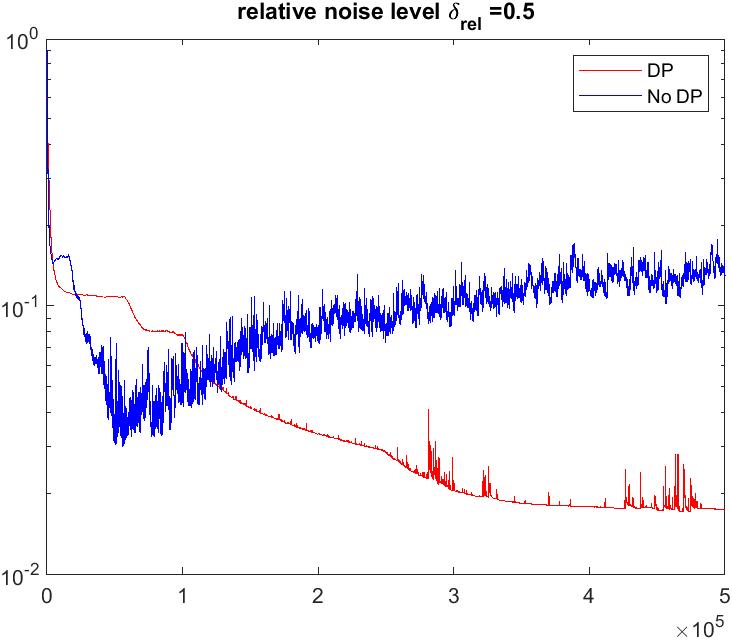}
\includegraphics[width = 0.32\textwidth]{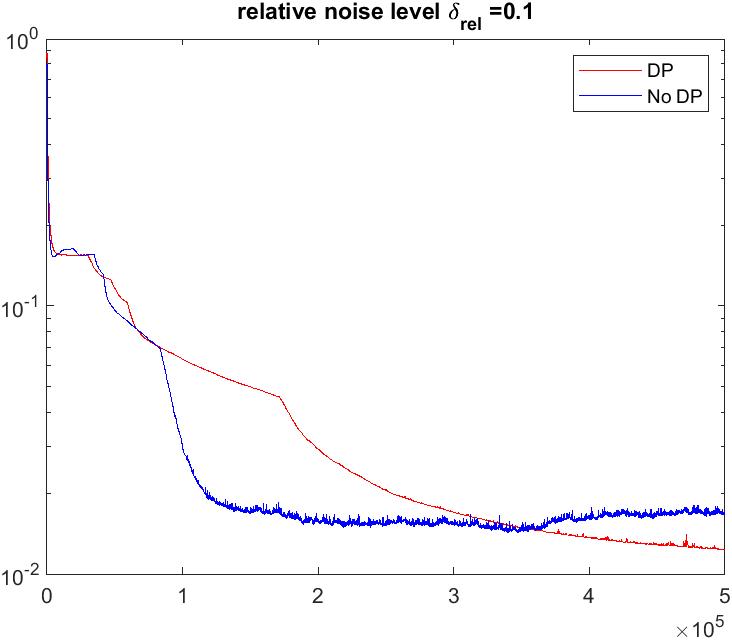}
\includegraphics[width = 0.32\textwidth]{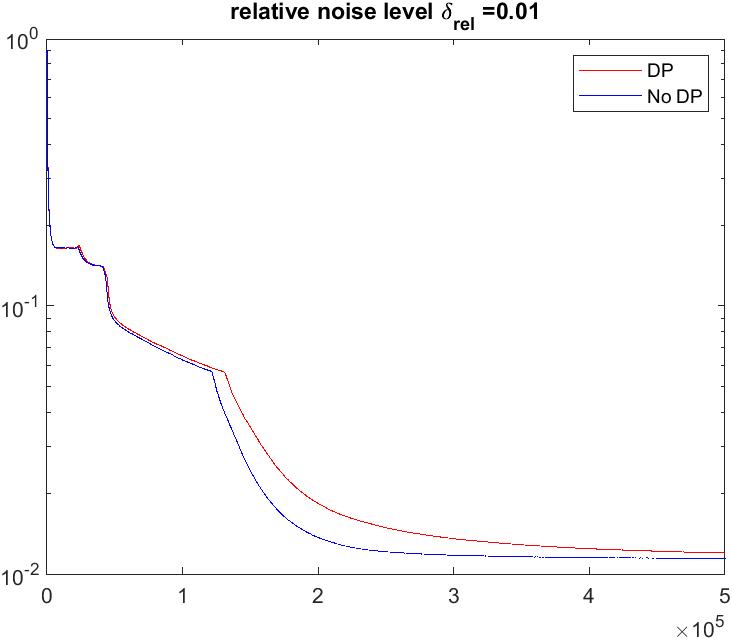}
\caption{Comparison of reconstruction results for a piecewise constant solution by (\ref{smd:TV}) using step-sizes (\ref{TV.s2}) and (\ref{TV.s3}). }\label{fig:TV2}
\end{figure}

\section*{\bf Acknowledgement} 

Q. Jin would like to thank Peter Math\'{e} from Weierstrass Institute for discussions on stochastic gradient descent. The work of Q. Jin is partially supported by the Future Fellowship of the Australian Research Council (FT170100231). The work of X. Lu is partially supported by the National Key Research and
Development Program of China (No. 2020YFA0714200) and the National Science
Foundation of China (No. 11871385).

\end{document}